\documentclass[10pt, letterpaper]{amsart}
\usepackage{latexsym, amsmath, amstext, amssymb, amsfonts, amscd, bm, array, multirow, amsbsy, mathrsfs}

\usepackage{multirow}
\usepackage{hhline}
\usepackage{colortbl}
\usepackage[usenames,dvipsnames]{xcolor}
\usepackage{t1enc}
\usepackage{array}
\usepackage[mathscr]{eucal}
\usepackage{indentfirst}
\usepackage{pb-diagram}
\usepackage{graphicx}
\usepackage{fancyhdr}
\usepackage{fancybox}
\usepackage{enumerate}
\usepackage{color}
\usepackage[all]{xy}
\usepackage[colorlinks=true, linkcolor=black]{hyperref}
\usepackage{tikz}
\usepackage{xparse}
\usetikzlibrary{matrix}
\usepackage{float}
\def\0{\hat{0}}
\def\1{\hat{1}}
\def\lin{\mathrm{lin}}
\def\alin{\mathrm{alin}}

\colorlet{linkequation}{blue}

\newcommand*{\SavedEqref}{}
\let\SavedEqref\eqref
\renewcommand*{\eqref}[1]{%
	\begingroup
	\hypersetup{
		linkcolor=blue,
		linkbordercolor=blue,
	}%
	\SavedEqref{#1}%
	\endgroup
}

\DeclareSymbolFont{extraup}{U}{zavm}{m}{n}
\DeclareMathSymbol{\varheart}{\mathalpha}{extraup}{86}
\DeclareMathSymbol{\vardiamond}{\mathalpha}{extraup}{87}

\usepackage{fancyhdr}
\usepackage{url}
\usepackage[sort&compress,comma]{natbib}
\bibpunct{[}{]}{,}{n}{}{,}
\newtheorem{thm}{Theorem}[section]
\newtheorem{theorem}{Theorem}[section]

\newtheorem{prop}[thm]{Proposition}
\newtheorem{proposition}[thm]{Proposition}

\newtheorem{lemma}[thm]{Lemma}

\newtheorem{cor}[thm]{Corollary}

\theoremstyle{definition}
\newtheorem{definition}[thm]{Definition}

\theoremstyle{remark}
\newtheorem{remark}[thm]{Remark}

\newcommand{\im}{\mathrm{im}}

\newcommand{\End}{\mathrm{End}}
\newcommand{\Aut}{\mathrm{Aut}}

\newcommand{\Mat}{\mathrm{Mat}}
\newcommand{\ev}{\mathrm{even}}
\newcommand{\eqdef}{\stackrel{{\rm def.}}{=}}

\DeclareFontFamily{U}{rsf}{}
\DeclareFontShape{U}{rsf}{m}{n}{<5> <6> rsfs5 <7> <8> <9> rsfs7 <10-> rsfs10}{}
\DeclareMathAlphabet\Scr{U}{rsf}{m}{n}

\def\Z{\mathbb{Z}}
\def\C{\mathbb{C}}
\def\R{\mathbb{R}}

\def\S{\mathbb{S}}
\def\s{\mathbb{s}}

\def\H{\mathbb{H}}

\def\S{\mathbb{S}}
\def\s{\mathbb{s}}
\def\rk{{\rm rk}}

\def\GL{\mathrm{GL}}

\def\Stab{\mathrm{Stab}}

\def\Ad{\mathrm{Ad}}

\def\tmu{\tilde{\mu}}

\def\tF{\tilde{F}}
\def\rGamma{\mathrm{\Gamma}}

\def\i{\mathbf{i}}
\newcommand{\be}{\begin{equation*}}
\newcommand{\ee}{\end{equation*}}
\newcommand{\ben}{\begin{equation}}
\newcommand{\een}{\end{equation}}
\newcommand{\beqa}{\begin{eqnarray*}}
	\newcommand{\eeqa}{\end{eqnarray*}}
\newcommand{\beqan}{\begin{eqnarray}}
\newcommand{\eeqan}{\end{eqnarray}}

\newcommand{\nn}{\nonumber}

\newcommand{\twopartdef}[4]
{
	\left\{
	\begin{array}{ll}
		#1 & \mbox{ if } #2 \\
		#3 & \mbox{ if } #4
	\end{array}
	\right .
}

\newcommand{\id}{\mathrm{id}}

\def\Cl{\mathrm{Cl}}
\def\Gr{\mathrm{Gr}}

\def\odd{\mathrm{odd}}
\def\Spin{\mathrm{Spin}}

\def\Pin{\mathrm{Pin}}
\def\Spin{\mathrm{Spin}}

\def\SO{\mathrm{SO}}
\def\O{\mathrm{O}}
\def\U{\mathrm{U}}

\def\cD{\mathcal{D}}

\def\cE{\mathcal{E}}
\def\cI{\mathcal{I}}
\def\cP{\mathcal{P}}

\def\G_2{\mathrm{G_2}}
\def\mG{\mathbb{G}}

\def\cL{\mathcal{L}}

\def\s{\mathfrak{s}}

\def\P{\mathbb{P}}

\newcommand{\Hom}{{\rm Hom}}

\def\Aut{\mathrm{Aut}}

\def\ClB{\mathrm{ClB}}
\def\Alg{\mathrm{Alg}}

\def\tw{\mathrm{tw}}
\def\ClRep{\mathrm{ClRep}}

\def\tAd{\widetilde{\Ad}}

\def\Re{\mathrm{Re}}
\def\Im{\mathrm{Im}}

\def\tN{\tilde{N}}
\def\ttau{\tilde{\tau}}

\def\G{\mathrm{G}}

\def\T{\mathbb{T}}
\def\R{\mathbb{R}}

\def\Tw{\mathrm{Tw}}
\def\rS{\mathrm{S}}
\def\w{\mathrm{w}}
\def\Pic{\mathrm{Pic}}
\def\cPic{\mathcal{P}ic}
\def\trho{\tilde{\rho}}
\def\tlambda{\tilde{\lambda}}

\def\Prin{\mathrm{Prin}}

\def\hP{\hat{P}}
\def\fD{\mathfrak{D}}
\def\A{\mathbb{A}}
\def\TU{\mathrm{TU}}

\def\rGamma{\mathrm{\Gamma}}
\def\frc{\mathfrak{c}}
\def\frC{\mathfrak{C}}
\def\conj{\mathrm{conj}}

\begin{document}
\title{Dirac operators on real spinor bundles of complex type}
 
\author[C. Lazaroiu]{C. Lazaroiu} \address{Center for Geometry and
  Physics, Institute for Basic Science, Pohang, Republic of Korea}
\email{calin@ibs.re.kr}

\author[C. S. Shahbazi]{C. S. Shahbazi} \address{Department of Mathematics, University of Hamburg, Germany}
\email{carlos.shahbazi@uni-hamburg.de}


\begin{abstract}
Let $(M,g)$ be a pseudo-Riemannian manifold of signature $(p,q)$. We compute the obstruction for a vector bundle $S$ over $(M,g)$ to admit a Dirac operator whose principal symbol induces on $S$ the structure of a bundle of irreducible real Clifford modules of complex type, that is, a real spinor bundle of irreducible complex type. In order to do this, we use the theory of Lipschitz structures in signature $p-q\equiv_8 3,7$ to reformulate the problem as the obstruction problem for $(M,g)$ to admit a $\mathrm{Spin}^{o}_{\alpha}$ structure with $\alpha = -1$ if $p-q \equiv_{8} 3$ or $\alpha = +1$ if $ p-q \equiv_{8} 7$, where $\mathrm{Spin}^o_+(p,q)=\mathrm{Spin}(p,q)\cdot\mathrm{Pin}_{2,0}$ and $\mathrm{Spin}^o_-(p,q)=\mathrm{Spin}(p,q)\cdot \mathrm{Pin}_{0,2}$. This allows computing the obstruction in terms of the Karoubi Stiefel-Whitney classes of $(M,g)$ and the existence of an auxiliary $\O(2)$ bundle with prescribed characteristic classes. Furthermore, we explicitly show how a $\mathrm{Spin}^o_{\alpha}$ structure can be used to construct $S$ and we give geometric characterizations (in terms of associated bundles) of the conditions under which the structure group of $S$ reduces to certain natural subgroups of $\mathrm{Spin}^o_{\alpha}$. Finally, we prove that certain codimension two submanifolds of spin manifolds and certain products of tori with Grassmanians, which were not known to admit irreducible real spinor bundles, do admit $\mathrm{Spin}^{o}_{\alpha}$ structures and therefore do admit real spinor bundles of irreducible complex type.
\end{abstract}

\maketitle

\setcounter{tocdepth}{1} %doesn't display subsections in TOC 
\tableofcontents

% % % % % % % % % % % % % % % % % % % % % % % % % % % % % % % % % % % % % % 
% % % % % % % % % % % % % % % % % % % % % % % % % % % % % % % % % % % % % % 

\section*{Introduction}

% % % % % % % % % % % % % % % % % % % % % % % % % % % % % % % % % % % % % % 
% % % % % % % % % % % % % % % % % % % % % % % % % % % % % % % % % % % % % % 

Let $S$ be a real vector bundle over a connected pseudo-Riemannian manifold $(M,g)$ of signature $(p,q)$ and dimension $d$. A \emph{Dirac operator} $D\colon C^{\infty}(S)\to C^{\infty}(S)$ on $S$ is by definition a first-order differential operator whose square $D^2\colon  C^{\infty}(S)\to C^{\infty}(S)$ has principal symbol $\sigma$ given by:
\begin{equation*}
	\sigma(m,\xi) = g(\xi,\xi)\, \mathrm{Id}_S\, , \qquad m\in M\, ,\qquad \xi \in T^{\ast}M\, .
\end{equation*}

\noindent
where $\mathrm{Id}_S$ denotes the identity endomorphism of $S$. Since by the previous equation the square of a Dirac operator $D$ satisfies the Clifford relation condition, the symbol of $D$ canonically extends to a morphism of bundles of unital and associative algebras:
\begin{equation*}
	\gamma\colon \Cl(M,g) \to End(S)\, ,
\end{equation*}

\noindent
whence $(S,\gamma)$ becomes a bundle of irreducible real spinors, that is, $\gamma$ defines by restriction an irreducible real representation of each fiber of $\Cl(M,g)$. The equivalence class $[\eta]$ (in the appropriate category, see \cite[\textsection 1]{Lipschitz}) of these fiberwise representations is well defined and is called the \emph{type} of $(S,\gamma)$. Conversely, every bundle of irreducible real spinors of type $[\eta]$ admits an irreducible Dirac operator\footnote{If the type $[\eta]$ of the spinor bundle structure defined by $D$ is given by an irreducible Clifford module we say that $D$ is \emph{irreducible}.}. Aside from the obvious rank restriction on $S$, which we will assume as granted in the following, existence of an irreducible Dirac operator on a given vector bundle $S$ is obstructed. The obstruction was computed in \cite{Lipschitz}, see also \cite{FriedrichTrautman,SpinorNote}, when the signature of $(M,g)$ satisfies $(p-q) \equiv_{8} 0,1,2,4,5,6$. The main goal of this paper is to compute the obstruction in the remaining signatures $(p-q) \equiv_8 3,7$, which correspond to $[\eta]$ being of complex type, and investigate the geometric structure and reductions of the corresponding bundles of irreducible spinors. For this, we will use the theory of Lipschitz structures developed in \cite{Lipschitz}, where an equivalence between the groupoid of weakly faithful spinor bundles of fixed type $[\eta]$ and the groupoid of Lipschitz structures of type $[\eta]$ was established. Applying this correspondence for the irreducible case in signature $(p-q) \equiv_{8} 3,7$ we conclude that $S$ admits an irreducible Dirac operator if and only if $(M,g)$ admits an {\em adapted	$\Spin^o$ structure}. The latter is defined as a principal $\Spin^o_{p,q}$-bundle $Q$ over $M$ endowed with a $\tlambda$-equivariant map to the orthonormal (co)frame bundle of $(M,g)$, where $\Spin^o_{p,q}$ is the group defined as follows:
\be
\Spin^o_{p,q}\eqdef \twopartdef{\Spin^o_-(p,q)\eqdef \Spin_{p,q}\cdot \Pin_{0,2}}{p-q\equiv_8 3}{\Spin^o_+(p,q)\eqdef\Spin_{p,q}\cdot \Pin_{2,0}}{p-q\equiv_8 7}\, ,
\ee
and $\tlambda: \Spin^o_{p,q}\rightarrow \O(p,q)$ is a certain surjective group morphism constructed from the untwisted adjoint representation of $\Spin_{p,q}$ and from the twisted adjoint representation of $\Pin_{0,2}$ or $\Pin_{2,0}$ respectively. The two groups $\Spin^o_{+}(p,q)$ and $\Spin^o_{-}(p,q)$  can in fact be considered in any dimension and signature and the same applies to the corresponding structures, which in full generality we call {\em $\Spin^o_\pm$ structures}. When $p-q\equiv_8 3,7$, an adapted $\Spin^o$ structure is thus a $\Spin^o_{\alpha_{p,q}}$ structure, where:
\be
\alpha_{p,q}\eqdef \twopartdef{-1}{p-q\equiv_8 3}{+1}{p-q\equiv_8 7}~~.  
\ee
Given an adapted $\Spin^o$ structure $Q$ over $M$, a bundle of irreducible real spinors can be constructed as the associated real vector bundle $S=Q\times_{\gamma_o}S_0$, where $\gamma_0:\Spin^o_{p,q}\rightarrow \End_\R(S_0)$ is an {\em elementary
	real pinor representation} of $\Spin^o_{p,q}$. The latter is constructed from an irreducible Clifford representation
$\gamma_0:\Cl_{p,q}\rightarrow \End_\R(S_0)$ in an $\R$-vector space $S_0$ of dimension:
\be
N=\dim_\R S_0=2^{\frac{d+1}{2}}~~.
\ee
It follows that $S$ admits a well-defined Clifford multiplication $TM\otimes S\rightarrow S$, which makes it into a bundle of simple modules over the fibers of $\Cl(M,g)$. When $M$ is orientable, the real vector bundle $S$ admits a complex structure, $Q$ reduces to a $\Spin^c$ structure $Q_0$ and $S$ can be viewed as the ordinary bundle of elementary {\em complex} pinors associated to $Q_0$ (see \cite{ComplexLipschitz}). When $M$ is unorientable, $S$ need {\em not} admit a complex structure and hence its global sections {\em cannot} be interpreted as complex spinors. In this case, $S$ admits a so-called {\em semilinear structure} --- a weakening of the concept of complex structure that still allows one to define the notions of linear and antilinear endomorphisms of $S$. Furthermore, the structure group of $S$ reduces to $\Pin_{p,q}$ or $\Pin_{q,p}$ if and only if $S$ admits a globally-defined {\em conjugation}. When a conjugation exists, the
structure group further reduces to $\Spin_{p,q}$ if and only if $M$ is orientable. The subtle interplay between these conditions leads to a web of possibilities which is considerably richer than what occurs in the ordinary case of $\Pin^\pm$ structures. Our main result is (see Theorem \ref{thm:equiv}):
\begin{theorem}
Let $(M,g)$ be a pseudo-Riemannian manifold of signature $(p,q)$ with $p-q\equiv_8 3,7$. Then $(M,g)$ admits a bundle of irreducible real spinors $(S,\gamma)$ if and only if there exists a principal $\O(2)$-bundle $E$ over $M$ such that the following conditions are satisfied:
\begin{eqnarray*}
&& \w_1^+(M)+\w_1^-(M)=\w_1(E)\\
&& \w_2^+(M)+\w_2^-(M)+\w_1(E)(p\w_1^+(M)+q\w_1^-(M))=\w_2(E) \\ && + \left[\delta(p,q)+\frac{p(p+1)}{2}+\frac{q(q+1)}{2}\right]\w_1(E)^2\, .\nn
\end{eqnarray*}
where: 
\begin{equation*}
\delta(p,q)=\twopartdef{1}{p-q\equiv_8 3}{0}{p-q\equiv_8 7}~~.
\end{equation*}
Let $[\eta]$ be the type of $(S,\gamma)$. In that case, and relative to $[\eta]$, there exists an adapted $\Spin^{o}_{\alpha}(V,h)$ structure	$Q(S,\gamma)$ on $(M,g)$, unique up to isomorphism, such that $(S,\gamma)$ is naturally associated to $Q(S,\gamma)$ as a bundle of irreducible Clifford modules, and Clifford multiplication in $S$ is implemented by the morphism of vector bundles $\mathfrak{C}\colon T^{\ast}M\otimes S\to S$ defined in equation \eqref{eq:frC}.
\end{theorem}

\noindent
The previous theorem encodes the obstruction to the existence of an irreducible Dirac operator in terms of Karoubi Stiefel-Whitney classes $\w_i$ \cite{Karoubi} and characterizes explicitly its associated spinor bundle structure. As a corollary to this theorem, we obtain a large class of generically non-spin manifolds that admit real spinor bundles of irreducible type.
\begin{cor}
Let $X$ be a $(2k+1)$-dimensional manifold which is oriented and spin and let $Y$ be an embedded $(2k-1)$-dimensional submanifold of $X$.
\begin{enumerate}
\item Assume that $2k-1 \equiv_{8} 7$ and that $X$ is endowed with a Riemannian metric $g$. Then $(Y,g|_Y)$ admits a $\Spin^{o}_{+}$ structure whose characteristic $\O(2)$-bundle $E$ is the orthogonal frame bundle of the normal bundle to $Y$ in $X$.
\item Assume that $2k-1 \equiv_{8} 3$ and that $X$ is endowed with a negative Riemannian metric $g$. Then $(Y,g|_Y)$ admits a
$\Spin^{o}_{-}$ structure whose characteristic $\O(2)$-bundle $E$ is the orthogonal frame bundle of the normal bundle to $Y$ in $X$.
\end{enumerate}
\end{cor}
 
\noindent
The reader is referred to Proposition \ref{prop:cod2spino} for the proof of this result. Furthermore, we apply the previous theorem to obtain several families of non-spin products of tori with Grassmanians that admit bundles of irreducible spinors of complex type. 

\begin{thm}
Assume that $n+1\equiv_4 0$. Then $\Gr_{2,n}$ is stably $\Spin^{o}$, namely:
\begin{enumerate}
\item For $j \equiv_{8} 7-2n$, the manifold $\Gr^{j}_{2,n}$ carries a $\Spin^{o}_{+}$ structure of positive-definite signature with characteristic $\O(2)$-bundle given by the orthogonal frame bundle of $\mathcal{L}_{2,n}$.
\item For $j \equiv_{8} -(3+2n)$, the manifold $\Gr^{j}_{2,n}$ carries a $\Spin^{o}_{-}$ structure of negative-definite signature characteristic $\O(2)$-bundle given by the orthogonal frame bundle of $\mathcal{L}_{2,n}$.
\end{enumerate}
\end{thm}

\noindent
The reader is refereed to Theorem \ref{thm:grasspino} for the proof of this result. These results show that non-spin manifolds admitting \emph{irreducible} $\Spin^{o}_{\alpha}$ structures are abundant. It would be interesting to study \emph{\'a la Lichnerowicz} the irreducible Dirac operator on these non-spin manifolds in order to obtain possible constraints on the existence of metrics with prescribed scalar curvature.

% % % % % % % % % % % % % % % % % % % % % % % % % % % % % % % % % % % % % % 

\subsection*{Organization of the paper} 

% % % % % % % % % % % % % % % % % % % % % % % % % % % % % % % % % % % % % %

In Section \ref{sec:O}, we discuss the groups $\Pin_{0,2}$ and $\Pin_{2,0}$ as well as their twisted and untwisted adjoint representations, describing their abstract group-theoretical models as well as certain isomorphic realizations as the groups $\Spin_{1,2}$ and $\Spin_{3,0}$. In Section \ref{sec:spinogroups}, we describe the groups $\Spin^o_\pm(p,q)$ and their real irreducible representations. We also give realizations of $\Spin^o_\pm(p,q)$ as subgroups of certain higher dimensional $\Pin$ and $\Spin$ groups, realizations which will be useful later on, and we discuss certain relevant subgroups of these groups. In Section \ref{sec:spinostructures}, we introduce the notion of $\Spin^o_\pm$ structures in arbitrary dimension and signature and extract the topological obstructions to their existence in odd dimension. Section \ref{sec:elementarypinor} considers the case of signatures satisfying the condition $p-q\equiv_8 3,7$, showing how a specific representation of $\Spin^o_{\alpha}$ can be used to construct an irreducible real Clifford module for $\Cl_{p,q}$. We also discuss certain natural subspaces associated to such representations. Section \ref{sec:elementary} discusses elementary real pinor bundles $S$ for $p-q\equiv_8 3,7$, focusing on their explicit realization in terms of adapted $\Spin^{o}$. In the same section, we also discuss certain sub-bundles of the endomorphism bundle of $S$ as well as the conditions under which the structure group of $S$ reduces to various subgroups of the group $\Spin^o_{p,q}$. In Section \ref{sec:examples} we give several examples of manifolds admitting $\Spin^{o}_{p,q}$ structures. Appendix \ref{Spincalpha} briefly discusses $\Spin^c_\pm$ structures in
arbitrary dimension and signature, contrasting them with $\Spin^o_\pm$ structures. Appendix \ref{app:semilinear} summarizes for completeness the theory semilinear structures on a vector bundle, a concept which is useful for understanding the properties of real pinor bundles associated to $\Spin^o$ structures.

% % % % % % % % % % % % % % % % % % % % % % % % % % % % % % % % % % % % % % 

\subsection*{Relation to the spinorial structure used by N. Nakamura} 

% % % % % % % % % % % % % % % % % % % % % % % % % % % % % % % % % % % % % %

In arbitrary dimension and arbitrary signature $(p,q)$, the groups $\Spin^o_\pm(p,q)$ are (non-central) extensions of the group $\Z_2$ by the group $\Spin^c_{p,q} = \Spin_{p,q}\cdot \U(1)$. For positive signature $q=0$, the groups $\Spin^o_\pm(p,q)$ have appeared before in \cite{Nakamura1,Nakamura2}, where they were used to define so-called
$\Spin^c_\pm$ structures and where $\Spin^c_-$ structures were employed to define and study a certain variant of monopole equations in four dimensions. A $\Spin^c_\pm$ structure is defined similarly to a $\Spin^o_\pm$ structure, but using a {\em different} representation $\lambda:\Spin^o_{p,q}\rightarrow \SO(p,q)$. The latter is constructed from the adjoint representation of $\Spin_{p,q}$ and from the twisted adjoint representation of $\Pin_{0,2}$ or $\Pin_{2,0}$ and covers only
the {\em special} pseudo-orthogonal group $\SO(p,q)$ --- unlike the representation $\tlambda$ mentioned above, which covers the {\em full} pseudo-orthogonal group. Due to this difference, $\Spin^c_\pm$ structures can exist only when $(M,g)$ is orientable and lead to a notion of ``real pinors'' which {\em differs} from the one considered in the present paper: in signature $p-q\equiv_8 3,7$, a bundle $S$  of {\em simple} modules over $\Cl(M,g)$ is always associated to a $\Spin^o$ structure and not to a $\Spin^c_+$ or $\Spin^c_-$ structure.

% % % % % % % % % % % % % % % % % % % % % % % % % % % % % % % % % % % % % % 

\subsection*{Notations and conventions} 

% % % % % % % % % % % % % % % % % % % % % % % % % % % % % % % % % % % % % % 

We use the conventions and notations of \cite{Lipschitz}.

% % % % % % % % % % % % % % % % % % % % % % % % % % % % % % % % % % % % % % 
% % % % % % % % % % % % % % % % % % % % % % % % % % % % % % % % % % % % % % 

\section{Abstract models for certain Clifford and pin groups}
\label{sec:O}

% % % % % % % % % % % % % % % % % % % % % % % % % % % % % % % % % % % % % % 
% % % % % % % % % % % % % % % % % % % % % % % % % % % % % % % % % % % % % % 

We start by discussing certain groups which will be relevant for the 
description of $\Spin^{o}_\pm(p,q)$.

% % % % % % % % % % % % % % % % % % % % % % % % % % % % % % % % % % % % % % 

\subsection{The abstract groups $\O_2(\alpha)$ and $\O_2^\C(\alpha)$}

% % % % % % % % % % % % % % % % % % % % % % % % % % % % % % % % % % % % % % 

\noindent Let $\alpha\in\{-1,1\}$ be a sign factor.

\begin{definition}
	Let $\O_2^\C(\alpha)$ be the non-compact non-Abelian Lie group with
	underlying set $\C^\times\times \Z_2$ and composition given by: 
	
	\beqan
	\label{Ualpha}
	&& (z_1,\hat{0})(z_2, \hat{0})=(z_1z_2, \hat{0})\, , \quad (z_1,\hat{0})(z_2,\hat1)=(z_1z_2,\hat1)\, ,\nn\\
	&& (z_1,\hat1)(z_2,\hat{0})=(z_1\bar{z}_2,\hat1)\, , \quad (z_1,\hat1)(z_2,\hat1)=(\alpha z_1\bar{z}_2,\hat{0})\, ,
	\eeqan
	where $\mathbb{Z}_{2} = \left\{\hat{0},\hat{1}\right\}$. The unit of
	$\O_2^\C(\alpha)$ is given by $1\equiv (1,{\hat 0})$. The {\bf square
		norm} is the group morphism $N:\O_2^\C(\alpha)\rightarrow \R_{>0}$
	given by:
	\be
	N(z,t)\eqdef |z|^2~~,~~\forall (z,t)\in \C^\times \times \Z_2~~.
	\ee
	We define $\O_2(\alpha)\eqdef \ker N$ to be the compact non-Abelian
	Lie subgroup of $\O_2^\C(\alpha)$ with underlying set $\U(1)\times
	\Z_2$.
\end{definition}

The group $\C^\times$ embeds into $\O_2^\C(\alpha)$ as the {\em
	non-central} subgroup $\C^\times \times \{{\hat 0}\}$, so we
identify $\i\in \C^\times$ with the element $(\i,{\hat 0})\in
\O_2^\C(\alpha)$ and $-1\in \C^\times$ with the element $(-1,{\hat
	0})$. The {\em conjugation element} $c\eqdef (1,{\hat 1})\in
\O_2(\alpha)$ satisfies $c^2=\alpha 1$ and $c^{-1}=\alpha
c=(\alpha,{\hat 1})$. This element generates the subgroup:
\be
\Gamma_\alpha\simeq \twopartdef{\Z_2}{\alpha=+1}{\Z_4}{\alpha=-1}~~. 
\ee
The element $c$ and the subgroup $\C^\times$ generate $\O_2^\C(\alpha)$,
while $c$ and $\U(1)$ generate $\O_2(\alpha)$. We have:
\be
\Ad(c)(x)=K(x)\, , \quad \forall\, x\in \O_2^\C(\alpha)\, ,
\ee
where $K:\O_2^\C(\alpha)\rightarrow \O_2^\C(\alpha)$ is the {\em
	conjugation automorphism}, given by:
\be
K(z,t)=({\bar z},t)\, , \quad \forall\, z\in \C^\times\, , \quad \forall t\in \Z_2\, .
\ee
Notice that $K^2=\id_{\O_2^\C(\alpha)}$ and $K(c)=c$.  The centers of
$\O_2^\C(\alpha)$ and $\O_2(\alpha)$ are given by
$Z(\O_2^\C(\alpha))=\R^\times 1$ and $Z(\O_2(\alpha))=\{-1,1\}\simeq
\Z_2$, respectively. The fixed point set of the conjugation automorphism in
$\O_2(\alpha)$ is the subgroup $\Gamma'_\alpha$ generated by $-1$ and
$c$, which is given by:
\be
\Gamma'_\alpha=\twopartdef{\mG_2\times \Gamma_+\simeq D_2}{\alpha=+1}{\Gamma_-\simeq \Z_4}{\alpha=-1}~~.
\ee
Here, $\mG_2=\{-1,1\}\simeq \Z_2$ denotes the multiplicative group of second
order roots of unity, while $D_2=\Z_2\times \Z_2$ denotes the Klein
group (the dihedral group with four elements).

\begin{definition}
	The {\bf generalized determinant} is the morphism
	$\eta_\alpha:\O_2(\alpha)\rightarrow \mG_2$ defined through:
	\be
	\eta_\alpha(z,t)\eqdef (-1)^t~~,~~\forall (z,t)\in \O_2(\alpha)~~.
	\ee
\end{definition}

\noindent This induces a $\Z_2$-grading of $\O_2(\alpha)$ whose
homogeneous pieces are the connected components of $\O_2(\alpha)$:
\beqa
&&\O_2^+(\alpha)\eqdef \ker \eta_\alpha =\{(z,{\hat 0})|z\in \U(1)\}= \U(1)1=\SO(2)1~~\nn\\
&&\O_2^-(\alpha)\eqdef \eta_\alpha^{-1}(\{-1\})=\{(z,{\hat 1})|z\in \U(1)\}=\U(1) c=\SO(2)c~~
\eeqa
and gives a short exact sequence: 
\ben
\label{ext}
1\longrightarrow \U(1) \longrightarrow \O_2(\alpha)\stackrel{\eta_\alpha}{\longrightarrow} \mG_2\longrightarrow 1~~.
\een
For any $u=(z,t)\in \O_2(\alpha)$, let $-u\eqdef (-1)u=(-z,t)\in
\O_2(\alpha)$. Notice that $\eta_\alpha(-u)=\eta_\alpha(u)$. 

\begin{prop}
	\label{O2}
	The short exact sequence \eqref{ext} is a non-central extension of
	$\mG_2$ by $\U(1)$.  Moreover, by exactness of \eqref{ext} we obtain:
	\begin{enumerate}[1.]
		\itemsep 0.0em
		\item For $\alpha=+1$, the sequence \eqref{ext} splits and we have
		$\O_2(+)\simeq \O(2)$.
		\item For $\alpha=-1$, the sequence \eqref{ext} presents $\O_2(-)$ as
		a non-split extension of $\mG_2\simeq \Z_2$ by $\U(1)$.  In
		particular, we have $\O_2(-)\not \simeq \O(2)$.
	\end{enumerate}
\end{prop}

\begin{proof}
	
	\begin{enumerate}[1.]
		\itemsep 0.0em
		\item For $\alpha=+1$, the element $c$ has order two and the map
		$\psi:\mG_2\rightarrow \O_2(+)$ given by $\psi(1)=1$ and
		$\psi(-1)=c$ is a group morphism which right-splits \eqref{ext}. Thus
		$\O_2(+)\simeq \U(1)\rtimes_{\psi} \mG_2\simeq \O(2)$. Here, the
		last isomorphism follows by noticing that the map $\Phi_0:
		\U(1)\rtimes_{\psi} \mG_2\rightarrow \O(2)$ given by:
		\ben
		\label{Phi0}
		\Phi_0(e^{i\theta},+1)(z)=e^{i\theta}z~~,~~ \Phi_0(e^{i\theta},-1)(z)=e^{i\theta}\bar{z}\, , \quad z\in \C\, ,
		\een
		induces an isomorphism of groups $\U(1)\rtimes_{\psi}\mG_2\simeq
		\O(2)$.
		\item For $\alpha=-1$, the sequence \eqref{ext} does not
		split. Indeed, a right-splitting morphism $\psi:\mG_2\rightarrow
		\O_2(-)$ would give an order two element $x\eqdef \psi(-1)\in
		\O_2(-)$ which must be of the form $x=(z,{\hat 1})$ in order for
		condition $\eta_{-}\circ \psi_{-} = \id_{\mG_{2}}$ to be
		satisfied. The order two condition $x^2=1$ gives then $-|z|^2=1$
		(since $\alpha=-1$), a contradiction. 
	\end{enumerate}
\end{proof}

\begin{remark} 
	Any reflection $C$ of $\R^2$ determines two isomorphisms of groups
	$\Phi_C^{(\pm)}:\O_2(+)\stackrel{\sim}{\rightarrow} \O(2)$ given by:
	\ben
	\label{PhiC}
	\Phi_C^{(\pm)}(e^{i\theta},{\hat 0})=R(\pm \theta)\, , \quad \Phi_C^{(\pm)}(e^{i\theta},{\hat 1})=R(\pm\theta) C\, ,
	\een
	where $\theta\in \R$ and: 
	\ben
	\label{rot}
	R(\theta)=\left[\begin{array}{cc}\cos(\theta) & -\sin(\theta) \\ \sin(\theta) & \cos(\theta) \end{array}\right]\in \SO(2)
	\een
	is the counterclockwise rotation of $\R^2$ by angle $\theta \!\mod 2\pi$
	with respect to its canonical orientation (that orientation in which
	the canonical basis $\epsilon_1=\left[\begin{array}{c} 1 \\0 \end{array}
	\right],\epsilon_2=\left[\begin{array}{c} 0 \\1 \end{array}\right]$ is
	positive):
	\beqa
	&& R(\theta)e_1=~\cos(\theta)e_1+\sin(\theta)e_2~~\\
	&& R(\theta)e_2=-\sin(\theta)e_1+\cos(\theta)e_2~~.
	\eeqa
	Notice that $R(-\theta)$ is the clockwise rotation by the same angle. 
	We have:
	\be
	\Phi_C^{(\pm)}(c)=C~~,~~\Phi_C^{(\pm)}(-1)=-I_2~~.
	\ee
	The isomorphisms $\Phi_C^{(\pm)}$ map the $\Z_2$-grading of $\O_2(+)$
	into the $\Z_2$-grading of $\O(2)$ given by:
	\be
	\O^+(2)\eqdef \SO(2)~~,~~\O^-(2)\eqdef \SO(2)C~~,
	\ee
	hence $\Phi_C^{(\pm)}$ can be viewed as isomorphisms of $\Z_2$-graded
	groups. These two isomorphisms are related by the conjugation
	automorphism of $\O_2(+)$:
	\be
	\Phi_C^{(-)}=\Phi_C^{(+)}\circ K~~
	\ee
	and we have: 
	\be
	\eta_+=\det \circ\, \Phi_C^\pm~~.
	\ee
	In particular, we can consider the reflection:
	\ben
	\label{C0}
	C_0=\left[\begin{array}{cc}1 & 0 \\ 0 & -1 \end{array}\right]\in \O(2)
	\een 
	with respect to the horizontal axis of $\R^2$ (i.e. the real axis of $\C$) to
	define:
	\ben
	\Phi_0^{(\pm)}\eqdef \Phi_{C_0}^{(\pm)}\, .\nn
	\een
	Then the isomorphism $\Phi_0$ of equation \eqref{Phi0} corresponds to:
	\be
	\Phi_0=\Phi_{C_0}^{(+)}\, .\nn
	\ee 
\end{remark}

\begin{definition}
	The {\bf squaring morphism} is the surjective group morphism
	$\sigma_\alpha:\O_2(\alpha)\rightarrow \O_2(+)$ given by:
	\be
	\sigma_\alpha(z,t)\eqdef (z^2,t)\, , \quad \forall (z,t)\in \O_2(\alpha)~~.
	\ee
\end{definition}

\noindent Notice the relations:
\ben
\label{etadet}
\sigma_\alpha(c)=c~~,~~\eta_\alpha=\eta_+\circ \sigma_\alpha=\det \circ\, \Phi_C^{(\pm)}\circ \sigma_\alpha~~.
\een
Since $\ker \sigma_\alpha=\{-1,1\}$, we have a short
exact sequence:
\be
1\longrightarrow \mG_2 \hookrightarrow \O_2(\alpha)\stackrel{\sigma_\alpha}{\longrightarrow} \O_2(+)\longrightarrow 1~~.
\ee
In particular, $\O_2(+)$ and $\O_2(-)$ are inequivalent central
extensions of $\O(2)$ by $\Z_2$. Notice that $\sigma_\alpha$ can be
viewed as a morphism of $\Z_2$-graded groups.

\subsection{Realization of $\O_2^\C(\alpha)$ as Clifford groups}

\label{sec:spinorgroups}

Let:
\be
\Cl_2(\alpha)\eqdef \twopartdef{\Cl_{2,0}}{\alpha=+1}{\Cl_{0,2}}{\alpha=-1}~~.
\ee
Then $\Cl_2(\alpha)$ has generators $e_1,e_2$ with relations:
\be
e_1^2=e_2^2=\alpha~~,~~e_1e_2=-e_2e_1~~.
\ee
Let us define $e_3\eqdef e_1 e_2\in \Cl^+_2(\alpha)$, which satisfies
$e_3^2=-1$. The elements $e_1,e_2$ and $e_3$ mutually anticommute and
span $\Cl_2(\alpha)$ over $\mathbb{R}$. The algebra
$\Cl_2(+)=\Cl_{2,0}$ is isomorphic with the algebra $\P$ of split
quaternions (and hence with the matrix algebra $\Mat(2,\R)$), while
$\Cl_2(-)$ is isomorphic with the quaternion algebra $\H$. Set
$J\eqdef e_3$ and $D\eqdef e_1$. Then $J$ and $D$ satisfy the
relations:
\begin{equation*}
J^2=-1\, , \quad D^2=\alpha\, ,\quad JD=-DJ\, .
\end{equation*}
Notice that $J$ is even while $D$ is odd with respect to the canonical
$\Z_2$-grading of $\Cl_2(\alpha)$. We have $Z(\Cl_2(\alpha))=\R$,
hence the extended Clifford group\footnote{See Definition 1.7 in
	reference \cite{Lipschitz}.} of $\Cl_2(\alpha)$ coincides with its
ordinary Clifford group, which we denote by:
\be
\G_2(\alpha)\eqdef\twopartdef{\G_{2,0}}{\alpha=+1}{\G_{0,2}}{\alpha=-1}~~.
\ee
We have $\Cl_2^+(\alpha)=\R\oplus \R J\simeq \C$ and $\Cl_2^-(\alpha)=\R
e_1\oplus \R e_2=\R D\oplus \R e_2=\Cl_2^+(\alpha)D$. 

\begin{prop}
	There exists an isomorphism of $\Z_2$-graded groups
	$\varphi_\alpha^\C:\O_2^\C(\alpha)\stackrel{\sim}{\rightarrow}
	\G_2(\alpha)$ which satisfies:
	\ben
	\label{Isom}
	\varphi_\alpha^\C(\i)=J\eqdef e_1e_2~~\mathrm{and}~~\varphi^\C_\alpha(c)=D\eqdef e_1~~.
	\een
\end{prop}

\begin{proof}
	It is easy to see that $\G_2(\alpha)=\Cl_2^+(\alpha)^\times \sqcup
	(\Cl_2^+(\alpha)^\times D)$ and $\Cl_2^+(\alpha)^\times \simeq
	\C^\times$. The map $\varphi_\alpha^\C:\O_2^\C(\alpha)\rightarrow
	\G_2(\alpha)$ given by:
	\be
	\varphi_\alpha^\C(x+\i y,{\hat 0})=x+yJ~~,~~\varphi_\alpha^\C(x+\i y,{\hat 1})=(x+yJ)D~~\mathrm{for}z=x+\i y\in \C^\times
	\ee
	is an isomorphism of $\Z_2$-graded groups which satisfies \eqref{Isom}.
	 \end{proof}

\subsection{Realization of $\O_2(\alpha)$ as $\Pin$ groups}

\noindent Let: 
\be
\Pin_2(\alpha)\equiv \twopartdef{\Pin_{2,0}}{\alpha=+1}{\Pin_{0,2}}{\alpha=-1}
\ee
denote the Pin group of $\Cl_2(\alpha)$. This group is $\Z_2$-graded
by the decomposition $\Pin_2(\alpha)=\Spin_2(\alpha)\sqcup
\Spin_2(\alpha)D$, where $\Spin_2(\alpha)=\Spin_{2,0}=\Spin_{0,2}$ and
$\Spin_2(\alpha)D=\Pin_2^-(\alpha)$.

\begin{prop}
	\label{prop:varphi}
	We have $\varphi_\alpha^\C(\O_2(\alpha))=\Pin_2(\alpha)$, hence
	$\varphi_\alpha^\C$ restricts to an isomorphism of $\Z_2$-graded
	groups
	$\varphi_\alpha:\O_2(\alpha)\stackrel{\sim}{\rightarrow}\Pin_2(\alpha)$
	which satisfies:
	\ben
	\label{IsomO2}
	\varphi_\alpha(\i)=J\eqdef e_1e_2~~\mathrm{and}~~\varphi_\alpha(c)=D\eqdef e_1~~.
	\een
\end{prop}

\begin{proof} 
	We have:
	\be
	\Pin_2(\alpha)\eqdef \ker (|\tN|:\G_2(\alpha)\longrightarrow \R_{>0})~~,
	\ee
	where $\tN$ is the twisted Clifford norm of $\Cl_2(\alpha)$.  For
	$\alpha=+1$, the twisted reversion $\ttau$ of
	$\Cl_2(+)=\Cl_{2,0}\simeq \P$ coincides with the conjugation of split
	quaternions and hence the twisted Clifford norm $\tN$ coincides with
	the split quaternion modulus:
	\be 
	\tN(q_0 +q_1 e_1+ q_2 e_2+q_3 e_3)=q_0^2+q_3^2-q_1^2-q_2^2~~.  
	\ee
	This implies $\varphi_\alpha(\O_2(+))=\Pin_2(+)=\U(1)\sqcup \U(1) D$.
	Notice that $\tN(\Cl_{2,0})\subset \R$ and that $\tN$ is positive
	definite on $\Cl_{2,0}^+$ and negative definite on $\Cl_{2,0}^-$.

	\noindent For $\alpha=-1$, the twisted reversion $\ttau$ of
	$\Cl_2(-)=\Cl_{0,2}\simeq \H$ coincides with quaternion
	conjugation. Hence the twisted Clifford norm $\tN$ coincides with the
	squared quaternion norm:
	\be
	\tN(q_0 +q_1 e_1+ q_2 e_2+q_3 e_3)=q_0^2+q_1^2+q_2^2+q_3^2
	\ee
	This implies $\varphi_\alpha(\O_2(-))=\Pin_2(-)=\U(1)\sqcup \U(1)D$.
	Notice that $\tN(\Cl_{0,2})\subset \R$, $\tN(J)=1$ and
	$\tN(D)=-\alpha$.
	 \end{proof}

\begin{remark}
	Setting:  
	\ben
	\label{z1}
	z(\theta)\eqdef \varphi_\alpha(e^{\i\theta},{\hat 0})=\cos(\theta)+\sin(\theta)J~~,
	\een
	we have:
	\ben
	\label{z2}
	\varphi_\alpha(e^{\i\theta},{\hat 1})=z(\theta)D
	\een
	and:
	\be
	\Spin_2(\alpha)=\{z(\theta)|\theta\in \R\}\simeq \U(1)\, , \quad \Pin_2^-(\alpha)=\Spin_2(\alpha)D~~.
	\ee 
\end{remark}

\noindent 
Let $\Ad_0^{(2)}:\Pin_2(\alpha)\rightarrow \O(2)$ and
$\tAd_0^{(2)}:\Pin_2(\alpha)\rightarrow \O(2)$ denote the vector and
twisted vector representations of $\Pin_2(\alpha)$, viewed as
morphisms of $\Z_2$-graded groups. Notice that $\ker \Ad_0^{(2)}=\ker
\tAd_0^{(2)}=\{-1,1\}$ and that $\det \circ \Ad_0^{(2)}=\det\circ
\tAd_0^{(2)}$.

\begin{prop}
	The untwisted vector representation of $\Pin_2(\alpha)$ agrees with
	the squaring morphism $\sigma_\alpha$ through the isomorphisms
	$\varphi_\alpha$ and $\Phi_0^{(-\alpha)}$:
	\ben
	\label{murep}
	\Ad_0^{(2)}\circ \varphi_\alpha =\Phi_0^{(-\alpha)}\circ \sigma_\alpha~~.
	\een
	This gives the following commutative diagram of morphisms of
	$\Z_2$-graded groups, where we also indicate the images of $c\in
	\O_2(\alpha)$ through the various maps:
	\be
	\label{diag1}
	\scalebox{1}{
		\xymatrix{
			c\in \O_2(\alpha)\ar[d]_{\sigma_\alpha}~ \ar[r]^{\varphi_\alpha} & \Pin_2(\alpha)\ni D ~ \ar[d]^{\Ad_0^{(2)}} \\
			c\in \O_2(+)~ \ar[r]_{\Phi_0^{(-\alpha)}}~ & \O(2)\ni C_0
	}}
	\ee
	Moreover, the generalized determinant agrees with the grading morphism
	$\det\circ \Ad_0^{(2)}$ of $\Pin_2(\alpha)$:
	\ben
	\label{etarep}
	\det\circ \Ad_0^{(2)}\circ \varphi_\alpha=\eta_\alpha~~,
	\een
	i.e. we have a commutative diagram of morphisms of $\Z_2$-graded groups: 
	\be
	\label{diag2}
	\scalebox{1}{
		\xymatrix{
			c\in \O_2(\alpha)\ar[d]_{\eta_\alpha}~ \ar[r]^{\varphi_\alpha} & \Pin_2(\alpha)\ni D ~ \ar[d]^{\Ad_0^{(2)}} \\
			-1 \in \mG_2 ~  ~& \ar[l]_{\det}\O(2)\ni C_0
	}}
	\ee
\end{prop}

\begin{proof}
	Since $J=e_1e_2$ anticommutes with $e_1$ and $e_2$, we have
	$\Ad_0^{(2)}(z(\theta))(e_k)=z(\theta)e_k z(-\theta)=z(\theta)^2
	e_k=z(2\theta) e_k$ for $k=1,2$. Using the relations $Je_1=-\alpha
	e_2$ and $Je_2=\alpha e_1$, this gives:
	\beqa
	&&\Ad_0^{(2)}(z(\theta))(e_1)=\cos(2\theta)e_1-\alpha\sin(2\theta)e_2=\cos(2\alpha \theta)e_1-\sin(2\alpha \theta)e_2\nn\\
	&&\Ad_0^{(2)}(z(\theta))(e_2)=\alpha\sin(2\theta)e_1+\cos(2\theta)e_2=\sin(2\alpha \theta)e_1+\cos(2\alpha \theta)e_2~~
	\eeqa
	and hence $\Ad_0^{(2)}(z(\theta))=R(-2\alpha \theta)$. Since $D=e_1$,
	we have $\Ad_0(D)(e_1)=e_1$ and $\Ad_0^{(2)}(D)(e_2)=-e_2$,
	i.e. $\Ad_0^{(2)}(D)=C_0$. Thus
	$\Ad_0^{(2)}(z(\theta)D)=R(-2\alpha\theta)\circ C_0$. Relation
	\eqref{murep} now follows from \eqref{PhiC}, \eqref{z1} and
	\eqref{z2}, while relation \eqref{etarep} follows from \eqref{murep}
	and \eqref{etadet}.  
\end{proof}

\subsection{Realization of $\O_2(\alpha)$ as spin groups}

Let:
\be
\Cl_3(\alpha)\eqdef \twopartdef{\Cl_{2,1}}{\alpha=+1}{\Cl_{0,3}}{\alpha=-1}~~,~~\Spin_3(\alpha)\eqdef \twopartdef{\Spin_{2,1}}{\alpha=+1}{\Spin_{0,3}}{\alpha=-1}
\ee
and let $e_1,e_2,e_3$ denote the canonical basis of $\R^3$. 
We have:
\be
e_1^2=\alpha~~,~~e_2^2=\alpha~~,~~e_3^2=-1~~.
\ee
There exists a unique unital monomorphism of $\R$-algebras ${\hat
	s}_\alpha:\Cl_2(\alpha)\rightarrow \Cl_3(\alpha)$ which satisfies
${\hat s}_\alpha(v)=ve_3$ for all $v\in \R^2$. Thus ${\hat
	s}_\alpha(e_1)=e_1e_3$, ${\hat s}_\alpha(e_2)=e_2e_3$ and ${\hat
	s}_\alpha(e_1e_2)=e_1e_2$. We have ${\hat
	s}_\alpha(\Cl_2(\alpha))=\Cl_3^+(\alpha)$ and ${\hat
	s}_\alpha(\Pin_2(\alpha))= \Spin_3(\alpha)$, where $\Cl_3^+(\alpha)$
is generated by $e_1e_3$ and $e_2e_3$ due to the relation $(e_1 e_3)
(e_2 e_3)=e_1 e_2$. The morphism ${\hat s}_\alpha$ takes $J=e_1e_2$
and $D=e_1$ respectively into $J'\eqdef J$ and $D'\eqdef
e_1e_3=De_3$. In particular, we have ${\hat
	s}_\alpha|_{\Spin_2(\alpha)}=\id_{\Spin_2(\alpha)}$. The group
$\O(2)$ embeds into the group:
\be
\SO_3(\alpha)\eqdef \twopartdef{\SO(2,1)}{\alpha=+1}{\SO(0,3)}{\alpha=-1}
\ee
through the injective morphisms:
\ben
\label{Sigma}
\Sigma_\alpha (R)\eqdef \left[\begin{array}{cc} (\det R) R & 0 \\ 0 & \det R \end{array}\right]~~,
\een
whose image equals:
\be
\Sigma_\alpha(\O(2))=\mathrm{S}[\O(2)\times \mG_2]~~.
\ee
We have:
\ben
\Sigma_\alpha(C_0)=C_0'\eqdef \left[\begin{array}{cc} - C_0 & 0 \\ 0 & -1 \end{array}\right]~~,
\een
where $-C_0\in \O^-(2)$. Let $\Ad_0^{(3)}:\Spin_3(\alpha)\rightarrow
\SO_3(\alpha)$ denote the vector representation of $\Spin_3(\alpha)$.
 
\begin{prop}
	The restriction of ${\hat s}_\alpha$ induces an isomorphism of groups
	$s_\alpha:\Pin_2(\alpha)\stackrel{\sim}{\rightarrow} \Spin_3(\alpha)$.
	Moreover, the vector representation of $\Spin_3(\alpha)$ agrees with
	the untwisted vector representation of $\Pin_2(\alpha)$ through the
	morphisms $s_\alpha$ and $\Sigma_\alpha$:
	\be
	\Ad_0^{(3)}\circ s_\alpha =\Sigma_\alpha \circ \Ad_0^{(2)}~~,
	\ee
	giving the commutative diagram:
	\ben
	\label{diag3}
	\scalebox{1}{
		\xymatrix{
			D \in \Pin_2(\alpha) ~ \ar@{->>}[d]^{\Ad_0^{(2)}} \ar[r]^{s_\alpha}_{\sim} & \Spin_3(\alpha)\ni D'\ar@{->>}[d]^{\Ad_0^{(3)}}\\
			C_0 \in \O(2) ~\ar@{{>}->}[r]_{\Sigma_\alpha} & \SO_3(\alpha)\ni C_0'~~,
	}}
	\een
	where we also indicates the images of $D$ through the various maps.
\end{prop}

\begin{proof}
	For any $a\in \Spin_2(\alpha)$, we have $s_\alpha(a)={\hat
		s}_\alpha(a)=a$ and
	$s_\alpha(aD)=s_\alpha(a)s_\alpha(D)=ae_1e_3=aDe_3$.  Thus:
	\be
	\Ad_0^{(3)}(s_\alpha(a))=\Ad_0^{(3)}(a)~~,~~\Ad_0^{(3)}(s_\alpha(aD))=\Ad_0^{(3)}(a)\Ad_0^{(3)}(e_1e_3)=\Ad_0^{(3)}(aD)\Ad_0^{(3)}(e_3)
	\ee
	and hence: 
	\be
	\Ad_0^{(3)}(s_\alpha(a))=\left[\begin{array}{cc}\Ad_0^{(2)}(a) & 0\\0 & 1 \end{array}\right]=\Sigma_\alpha(\Ad_0^{(2)}(a))
	\ee
	and: 
	\be
	\Ad_0^{(3)}(s_\alpha(aD))=\left[\begin{array}{cc} -\Ad_0^{(2)}(aD) & 0\\0 & -1 \end{array}\right]=\Sigma_\alpha(\Ad_0^{(2)}(aD))~~,
	\ee
	because $\Ad_0^{(3)}(a)(e_3)=e_3$, $\Ad_0^{(3)}(e_1e_3)(e_3)=-e_3$ and
	$\Ad_0^{(3)}(e_3)(e_k)=-e_k$ for $k=1,2$ while $\det
	\Ad_0^{(2)}(a)=+1$ and $\det \Ad_0^{(2)}(aD)=-1$. Notice that:
	\be
	\Ad_0^{(3)}(D')=\Ad_0^{(3)}(s_\alpha(D))=\left[\begin{array}{cc} -C_0 & 0\\0 & -1 \end{array}\right]=C_0'~~,
	\ee
	since $\Ad_0^{(2)}(D)=C_0$. 
	 \end{proof} 

\noindent Composing $\varphi_\alpha$ and $s_\alpha$, respectively
$\Phi_0^{(-\alpha)}$ and $\Sigma_\alpha$ gives morphisms of groups:
\beqan
&& \psi_\alpha\eqdef s_\alpha\circ \varphi_\alpha:\O_2(\alpha)\stackrel{\sim}{\rightarrow} \Spin_3(\alpha)~~\nn\\
&& \Psi_\alpha\eqdef \Sigma_\alpha\circ \Phi_0^{(-\alpha)}:\O_2(+) \rightarrow \SO_3(\alpha)~~.
\eeqan
The situation is summarized in the commutative diagrams: 
\ben
\label{diag4}
\scalebox{1}{
	\xymatrix{
		\O_2(\alpha)\ar@{->>}[d]_{\sigma_\alpha}~ \ar[r]^{\varphi_\alpha}_{\sim}\ar@/^1.5pc/[rr]^{\psi_\alpha}_\sim & \Pin_2(\alpha) ~ \ar@{->>}[d]^{\Ad_0^{(2)}} \ar[r]^{s_\alpha}_{\sim} & \Spin_3(\alpha)\ar@{->>}[d]^{\Ad_0^{(3)}}\\
		\O_2(+)~ \ar[r]_{\Phi_0^{(-\alpha)}}^{\sim}~\ar@/_1.5pc/[rr]_{\Psi_\alpha} & \O(2) ~\ar@{{>}->}[r]_{\Sigma_\alpha} & \SO_3(\alpha)
}}
\een
and:
\ben
\label{diag5}
\scalebox{1}{
	\xymatrix{
		c\in \O_2(\alpha)\ar[d]_{\sigma_\alpha}~ \ar[r]^{\psi_\alpha}_{\sim} & \Spin_3(\alpha)\ni D' ~ \ar[d]^{\Ad_0^{(3)}} \\
		c \in \O_2(+) ~  ~\ar@{{>}->}[r]_{\Psi_\alpha} & \SO_3(\alpha)\ni C_0'
}}
\een

\section{The groups $\Spin^o_\alpha$}
\label{sec:spinogroups}

Let $(V,h)$ be a real quadratic space of signature $(p,q)$ and
dimension $d=p+q$.  Thus $V$ is an $\R$-vector space of dimension $d$
and $h:V\times V\rightarrow \R$ is an $\R$-bilinear symmetric form
defined on $V$ and having signature $(p,q)$, where $p$ and $q$
respectively count the numbers of positive and negative eigenvalues.

\begin{definition}
	Define: 
	\be
	\Spin^o_\alpha(V,h)\eqdef \Spin(V,h)\cdot \Pin_2(\alpha)\eqdef [\Spin(V,h)\times \Pin_2(\alpha)]/\{-1,1\}~~.
	\ee
\end{definition}

\noindent The unit of $\Spin^o(V,h)$ is given by $1\equiv [1,1]=[-1,-1]$,
while the {\em twisted unit} is the element:
\be
{\tilde 1}\eqdef [1,-1]=[-1,1]~~,
\ee
which satisfies ${\tilde 1}^2=1$ and generates the center: 
\ben
\label{Z}
Z_2\eqdef Z(\Spin^o_{\alpha}(V,h))=\{{\tilde 1}, 1\}\simeq \Z_2\, ,
\een
of $\Spin^o_{\alpha}(V,h)$. The groups $\Spin(V,h)$ and $\Pin_2(\alpha)$ 
identify with the following subgroups of $\Spin^o_{\alpha}(V,h)$: 
\begin{eqnarray*}
	{\widehat \Spin}(V,h) &\eqdef & \{[a,1]\in \Spin^o_{\alpha}(V,h)\,\, |\,\, a\in \Spin(V,h)\}\, ,\\{\widehat \Pin}_2(\alpha) &\eqdef & \{[1,b]\in \Spin^o_{\alpha}(V,h)\,\, |\,\, b\in \Pin_2(\alpha)\}~~.
\end{eqnarray*}
This gives the decomposition: 
\be
\Spin^o_{\alpha}(V,h)={\widehat \Spin}(V,h){\widehat \Pin}_2(\alpha)\, .
\ee
We have ${\widehat \Spin}(V,h)\cap {\widehat \Pin}_2(\alpha)=Z_2$ and
${\tilde 1}\in {\widehat \Spin}(V,h)$.  Using the isomorphism
$\varphi_\alpha$ of Proposition \ref{prop:varphi} to identify
$\O_2(\alpha)$ with $\Pin_2(\alpha)$, we obtain an isomorphism:
\be
\Spin^o_\alpha(V,h)\simeq \underline{\Spin}^o_\alpha(V,h)\eqdef \Spin(V,h)\cdot \O_2(\alpha)\eqdef [\Spin(V,h)\times \O_2(\alpha)]/\{-1,1\}~~.
\ee
Since $\U(1)=\SO(2)$ is a {\em non-central} subgroup of
$\O_2(\alpha)$, it embeds as the non-central subgroup of
$\underline{\Spin}^o_\alpha(V,h)$ consisting of the elements
$[1,z]=[-1,-z]$ with $z\in \U(1)$.

\subsection{Various subgroups}
The element ${\hat D}\eqdef [1,D]\in \Spin^o(V,h)$ satisfies: 
\be
{\hat D}^2=[1,\alpha]=[\alpha,1]=\twopartdef{1}{\alpha=+1}{{\tilde 1}}{\alpha=-1}~~
\ee 
and generates a cyclic subgroup: 
\ben
\label{Gammao}
\Gamma_{o,\alpha}\simeq \twopartdef{\Z_2}{\alpha=+1}{\Z_4}{\alpha=-1}\, ,
\een
which is neither normal nor central. The group $\Spin^c(V,h)\eqdef
\Spin(V,h)\cdot \Spin_{2,0}=\Spin(V,h)\cdot \U(1)$ embeds as the
following normal subgroup of $\Spin^o(V,h)$:
\be
{\widehat \Spin}^c(V,h)=\{[a, z]|a\in \Spin(V,h)\, ,~~z\in \Spin_2\}\, .
\ee 
In particular, $\Spin^o_\alpha(V,h)$ contains the normal $\U(1)$
subgroup:
\ben
\label{U}
\U\eqdef \{z(\theta)=e^{\theta J}|\theta\in \R\}\subset {\widehat \Spin}^c(V,h)\subset \Spin^{o}_{\alpha}(V,h)\, .
\een
Writing $\Pin_2(\alpha)=\Spin_2\sqcup \Spin_2 D=\U(1)\sqcup \U(1)D$
gives the decomposition:
\be
\Spin^o_\alpha(V,h)={\widehat \Spin}^c(V,h)\sqcup {\widehat \Spin}^c(V,h) {\hat D}~~.
\ee
We have a short exact sequence:
\ben
\label{Spincseq}
1\longrightarrow \Spin^c(V,h) \longrightarrow \Spin^o_\alpha(V,h)\stackrel{\tilde{\eta}_\alpha}{\longrightarrow} \Z_2 \longrightarrow 1~~,
\een
where $\tilde{\eta}_\alpha$ is the {\em $\Z_2$-grading morphism} given
by:
\be
\tilde{\eta}_\alpha([a,b])\eqdef \eta_\alpha(b)\, ,\quad\forall\, a\in \Spin(V,h)\, ,\quad\forall\, b\in \Pin_2(\alpha)\, .
\ee

\begin{prop}
	\label{prop:SpincZ2}
	We have $\Spin^o_+(V,h)\simeq \Spin^c(V,h)\rtimes_{\xi} \Z_2$, where
	$\xi:\Z_2\rightarrow \Aut(\Spin^c(V,h))$ is the group morphism given
	by $\xi({\hat 0})=\id_{\Spin^c(V,h)}$ and:
	\be
	\xi({\hat 1})([a,z])=[a,\bar{z}]~~\forall\,\, [a,z]\in \Spin^c(V,h)~~(a\in \Spin(V,h), z\in \U(1))\, .
	\ee
\end{prop}

\begin{proof}
	In order to show that $\Spin^o_+(V,h)\simeq \Spin^c(V,h)\rtimes_{\xi}
	\Z_2$, it suffices to show that the short exact sequence
	\eqref{Spincseq} splits on the right. Such a splitting $R_{\xi}\colon
	\mathbb{Z}_{2}\to \Spin^o_+(V,h)$ is given by:
	\begin{equation*}
	R_{\xi}(0) = [1,1]\, , \qquad R_{\xi}(0) = [1,D]\, .
	\end{equation*}
	The morphism $R_{\xi}$ in turn defines a morphism:
	\begin{equation*}
	\xi\colon\mathbb{Z}_{2}\to\Aut(\Spin^c(V,h))\, ,
	\end{equation*}
	given by the adjoint action\footnote{Recall that $\Spin^{c}$ is a
		normal subgroup of $\Spin^{o}_{\alpha}$.} of the image $\Im R_{\xi}$
	of $R_{\xi}$ in $\Spin^{o}_{\alpha}$. It is easy to see that the
	adjoint action of $[1,D]$ is by complex conjugation of $z$, where
	$[a,z]\in \Spin^{c}(V,h)$.   
\end{proof}

\begin{remark}
	The right-splitting used in the proof of Proposition
	\ref{prop:SpincZ2} does not work for the group $\Spin^{o}_{-}(V,h)$.
	In that case, one has $D^{2} = -1$ and hence $\left\{[1,1],
	[1,D]\right\}$ is not a $\mathbb{Z}_{2}$-subgroup of
	$\Spin^{o}_{-}(V,h)$.
\end{remark}

\noindent In the proof of Proposition \ref{prop:SpincZ2} we used the
fact that the adjoint action of $\hat{D}$ on $\Spin^{c}(V,h)\subset
\Spin^{o}_{\alpha}(V,h)$ is by complex conjugation of $z$. Define the
{\em conjugation automorphism} $K\in \Aut(\Spin^o_\alpha(V,h))$
through:
\be
K\eqdef \Ad({\hat D})~~.
\ee
This satisfies $K({\hat D})={\hat D}$ and $K(z)={\bar z}$ for all $z\in
\U$. The fixed point set of $K$ is the non-normal subgroup:
\ben
\label{Pinalpha}
{\widehat \Pin}_\alpha(V,h)={\widehat \Spin}(V,h)\sqcup {\widehat \Spin}(V,h) {\hat D}~~.
\een
Since ${\hat D}$ commutes with all elements of ${\widehat \Spin}(V,h)$
and ${\hat D}^2=[1,\alpha]=[\alpha,1]\in {\widehat \Spin}(V,h)$, any
choice of an element $v\in V$ such that $h(v,v)= \epsilon_{v}$ with
$\epsilon_{v}\in \left\{-1,1\right\}$, gives an isomorphism of groups:
\ben
{\widehat \Pin}_\alpha(V,h)\stackrel{\sim}{\rightarrow} \Pin(V,\alpha \epsilon_{v} h)~~
\een
which takes any $[a,1]\in {\widehat \Spin}(V,h)$ into $a\in
\Spin(V,\alpha \epsilon_{v} h)=\Spin(V,h)$ and any $[a,1] {\hat D} =
[a,D]\in {\widehat \Spin}(V,h){\hat D}$ 
into $a v\in \Pin^-(V,\alpha \epsilon_{v} h)$. The $\Z_2$-grading
\eqref{Pinalpha} corresponds through this isomorphism to the canonical
$\Z_2$-grading:
\be
\Pin(V,\alpha \epsilon_{v} h)=\Spin(V,\alpha \epsilon_{v} h)\sqcup \Pin^-(V,\alpha \epsilon_{v} h)~~
\ee
of $\Pin(V,\alpha \epsilon_{v} h)$. We have a short exact sequence: 
\ben
\label{Pinseq}
1\longrightarrow \U \longrightarrow \Spin^o_\alpha(V,h)\stackrel{\pi_\alpha}{\longrightarrow} {\widehat \Pin}_\alpha(V,h) \longrightarrow 1~~,
\een
where: 
\ben
\label{Useq}
\pi_\alpha([a,b])=\twopartdef{\left[a,1\right]}{b\in \Spin_2}{\left[a,D\right]}{b\in \Spin_2 D}~~,~~\forall\,\, [a,b] \in \Spin^{o}_{\alpha}(V,h)\, .
\een
We leave the proof of the following proposition to the reader.

\begin{prop}
	\label{prop:PinSubgroup}
	The inclusion morphism $j_\alpha:{\widehat
		\Pin}_\alpha(V,h)\hookrightarrow \Spin^o_\alpha(V,h)$ splits the
	sequence \eqref{Pinseq} from the right. Thus
	$\Spin^o_\alpha(V,h)\simeq \U\rtimes_{\zeta_\alpha} \Pin_\alpha(V,h)$,
	where $\zeta_\alpha:\Pin_\alpha(V,h)\rightarrow \Aut(\U)$ is the group
	morphism given by:
	\be
	\zeta_\alpha(x)(z)=\twopartdef{z}{x = [a,1]\in {\widehat \Spin}(V,h)}{{\bar z}}{x = [a,D]\in {\widehat \Spin}(V,h)\hat{D}}\, .
	\ee
	Here, $\U$ is the normal $\U(1)$ subgroup of $\Spin^o_\alpha(V,h)$
	which was defined in \eqref{U}.
\end{prop}

\begin{remark}
	The quotient $\Spin^o_\alpha(V,h)/{\widehat \Pin_\alpha(V,h)}$ is a principal 
	$\U(1)$-homogeneous space ($\U(1)$ torsor).
\end{remark}

\subsection{Elementary representations}
\label{sec:elementaryrep}

We introduce several representations of $\Spin^{o}_{\alpha}(V,h)$
which will be relevant later on. Let\footnote{A word of caution about
	notation: the group $\check{\O}(V,h)$ should not be confused with the group ${\hat
		\O}(V,h)$ used in \cite{Lipschitz}, which is given by the {\em
		opposite} dichotomy.}:
\ben
\label{checkO}
{\check \O}(V,h)\eqdef \twopartdef{\SO(V,h)}{d=\ev}{\O(V,h)}{d=\odd}~~
\een
and: 
\ben
\check{\Pin}_2(\alpha)\eqdef \twopartdef{\Pin_2(\alpha)}{d=\ev}{\Spin_2(\alpha)\simeq \U(1)}{d=\odd}~~.
\een
We have $\check{\Pin}_2(\alpha)\simeq \check{\O}_2(\alpha)$, where: 
\ben
\check{\O}_2(\alpha)\eqdef  \twopartdef{\O_2(\alpha)}{d=\ev}{\O_2^+(\alpha)\simeq \U(1)}{d=\odd}~~.
\een

\begin{definition}
	\label{def:elementaryreps}
	
	\begin{enumerate}[1.]
		\itemsep 0.0em
		\item The {\bf characteristic representation} of $\Spin^o_\alpha(V,h)$
		is the group epimorphism
		$\mu_\alpha:\Spin^o_\alpha(V,h)\rightarrow\O(2)$ defined through:
		\be
		\mu_\alpha([a,u])\eqdef \Ad_0^{(2)}(u)~~\forall [a,u]\in \Spin^o_\alpha(V,h)~~,
		\ee
		where $\Ad_0^{(2)}$ is the untwisted vector representation of
		$\Pin_2(\alpha)$.
		\item The {\bf twisted characteristic representation} of
		$\Spin^o_\alpha(V,h)$ is the group epimorphism
		$\tmu_\alpha:\Spin^o_\alpha(V,h)\rightarrow\O(2)$ defined through:
		\be
		\tmu_\alpha([a,u])\eqdef \tAd_0^{(2)}(u)\, , \quad \forall\, [a,u]\in \Spin^o_\alpha(V,h)\, ,
		\ee
		where $\tAd_0^{(2)}$ is the twisted vector representation of
		$\Pin_2(\alpha)$.
		\item The {\bf vector representation} of $\Spin^o_\alpha(V,h)$ is the
		group epimorphism $\lambda_\alpha:\Spin^o_\alpha(V,h)\rightarrow
		\SO(V,h)$ defined through:
		\be
		\lambda_\alpha([a,u])\eqdef \Ad_0(a)~~,~~\forall [a,u]\in \Spin^o_\alpha(V,h)~~,
		\ee
		where $\Ad_0$ is the vector representation of $\Spin(V,h)$.
		\item The {\bf twisted vector representation} of $\Spin^o_\alpha(V,h)$
		is the group epimorphism
		$\tlambda_\alpha:\Spin^o_\alpha(V,h)\rightarrow \check{\O}(V,h)$
		defined through:
		\be
		\tlambda_\alpha([a,u])\eqdef \det(\Ad_0^{(2)}(u))\Ad_0(a)~~\forall [a,u]\in \Spin^o_\alpha(V,h)~~.
		\ee
		\item The {\bf basic representation} is the group epimorphism
		$\rho_\alpha\eqdef \lambda\times \tmu:\Spin^o_\alpha(V,h)\rightarrow
		\SO(V,h)\times \O(2)$:
		\be
		\rho_\alpha([a,u])\eqdef (\Ad_0(a),\tAd_0^{(2)}(u))~~.
		\ee
		\item The {\bf twisted basic representation} is the group epimorphism
		$\trho_\alpha\eqdef \tlambda\times
		\mu:\Spin^o_\alpha(V,h)\rightarrow \O(V,h){\hat \times} \O(2)$:
		\be
		\trho_\alpha([a,u])\eqdef (\det (\Ad_0^{(2)}(u))\Ad_0(a),\Ad_0^{(2)}(u))~~,
		\ee
		where: 
		\ben
		\O(V,h){\hat \times} \O(2)\eqdef \twopartdef{\SO(V,h)\times \O(2)}{d=\ev}{\mathrm{S}[\O(V,h)\times \O(2)]}{d=\odd}~~.
		\een
	\end{enumerate}
\end{definition}

\noindent Notice that $\tAd_0^{(2)}(u)=(\det
\Ad_0^{(2)}(u))\Ad_0^{(2)}(u)$ for all $u\in \Pin_2(\alpha)$. The
various representations introduced above induce the following short
exact sequences:
\beqan
\label{ses}
&& 1\longrightarrow \Pin_2(\alpha) \longrightarrow \Spin^o_\alpha(V,h)\stackrel{\lambda_\alpha}{\longrightarrow} \SO(V,h)\longrightarrow 1\nn\\
&& 1\longrightarrow \check{\Pin}_2(\alpha) \longrightarrow \Spin^o_\alpha(V,h)\stackrel{\tlambda_\alpha}{\longrightarrow} \check{\O}(V,h)\longrightarrow 1\nn\\
&& 1\longrightarrow \Spin(V,h) \longrightarrow \Spin^o_\alpha(V,h)\stackrel{\mu_\alpha}{\longrightarrow}\O(2)\longrightarrow 1\nn\\
&& 1\longrightarrow \Spin(V,h) \longrightarrow \Spin^o_\alpha(V,h)\stackrel{\tmu_\alpha}{\longrightarrow}\O(2)\longrightarrow 1~~\\
&& 1\longrightarrow Z_2 \longrightarrow \Spin^o_\alpha(V,h)\stackrel{\rho_\alpha}{\longrightarrow} \SO(V,h)\times \O(2)\longrightarrow 1~~\nn\\
&& 1\longrightarrow Z_2 \longrightarrow \Spin^o_\alpha(V,h)\stackrel{\trho_\alpha}{\longrightarrow} \O(V,h){\hat \times} \O(2)\longrightarrow 1~~,\nn
\eeqan
where $Z_2$ is the center \eqref{Z} of $\Spin^o(V,h)$. Therefore,
$\Pin_{2}(\alpha)$, $\check{\Pin}_2(\alpha)$ and $\Spin(V,h)$ are
normal subgroups of $\Spin^{o}_{\alpha}(V,h)$. In particular,
${\widehat \Spin}(V,h)$ is a normal subgroup of $\Spin^o(V,h)$ with
corresponding short exact sequence given either by the untwisted or
the twisted vector representations.

\begin{remark}
	When $d$ is odd, the second exact sequence above shows that
	$\Spin^o_\pm(V,h)$ are distinct {\em non-central} extensions of
	$\O(V,h)$ by $\U(1)$, both of which differ from the {\em central}
	extension provided by the group $\Pin^c(V,h)\eqdef \Pin(V,h)\cdot
	\U(1)$. When $d$ is even, $\Spin^o_\pm(V,h)$ are extensions of
	$\SO(V,h)$ by $\O_2(\alpha)$.
\end{remark}

\subsection{Realization of $\Spin_\alpha^o(V,h)$ as a subgroup of a pin group}

Let ${\hat V}\eqdef V\oplus \R^2$ and ${\hat h}_\alpha:{\hat V} \times
{\hat V}\rightarrow \R$ be the non-degenerate bilinear form given by
${\hat h}_\alpha(x\oplus u,y\oplus v)=h(x,y)+\alpha \langle u, v
\rangle$ for all $x,y\in V$ and $u,v\in \R^2$, where
$\langle~,~\rangle$ is the canonical Euclidean scalar product of
$\R^2$.  If $h$ has signature $(p,q)$, then ${\hat h}_\alpha$ has
signature $(p+2,q)$ when $\alpha=+1$ and signature $(p,q+2)$ when
$\alpha=-1$. Let $(e_{d+1},e_{d+2})$ be the canonical basis of
$\R^2$. The following relations hold in $\Cl({\hat V},{\hat
	h}_\alpha)$:
\be
e_{d+1}^2=e_{d+2}^2=\alpha~~.
\ee
Notice that $\Cl(\R^2, \alpha\langle~,~\rangle)=\Cl_2(\alpha)$ embeds
as the subalgebra of $\Cl({\hat V},{\hat h}_\alpha)$ generated by
$e_{d+1}$ and $e_{d+2}$.  Let $G:\O(V,h)\times \O(2)\rightarrow
\O({\hat V},{\hat h}_\alpha)$ be the injective morphism of groups
given by:
\ben
\label{G}
G(A,B)\eqdef A\oplus B~~.
\een
We have $G(\O(V,h){\hat \times} \O(2))\subset {\hat \O}({\hat V},{\hat
	h}_\alpha)$, where:
\be
{\hat \O}({\hat V},{\hat h}_\alpha)\eqdef\twopartdef{\O({\hat V},{\hat h}_\alpha)}{d=\ev}{\SO({\hat V},{\hat h}_\alpha)}{d=\odd}~~.
\ee 

\begin{prop}
	\label{prop:j}
	The map $j:\Spin^o_\alpha(V,h)\rightarrow \Pin({\hat V},{\hat
		h}_\alpha)$ defined through:
	\ben
	\label{j}
	j([a,b])=ab~~\forall~~a\in \Spin(V,h)~~,~~b\in \Pin_2(\alpha)~~.
	\een
	is an injective morphism of groups which satisfies: 
	\ben
	\label{Eq1}
	{\widehat \Ad}_0\circ j=G\circ \trho_\alpha~~,
	\een
	and: 
	\ben
	\label{Eq2}
	{\widehat \tAd}_0\circ j=G\circ\rho_\alpha~~,
	\een
	where ${\widehat \Ad}_0$ and ${\widehat \tAd}_0$ are the untwisted and
	twisted vector representations of $\Pin({\widehat V},{\hat
		h}_\alpha)$, respectively.
\end{prop}

\begin{proof}
	We have $\Pin_2(\alpha)=\Pin(\R^2,\alpha\langle~,~\rangle)\subset
	\Cl({\hat V},{\hat h}_\alpha)$. Since $\Spin(V,h)\cap
	\Pin_2(\R^2,\alpha\langle, \rangle)=\{-1,1\}$, the group morphism
	given by \eqref{j} is injective. The element $D\eqdef e_{d+1}$
	satisfies $Dv=-vD~~\forall v\in V$. We have
	$\Pin_2(\alpha)=\Spin_2(\alpha)\sqcup \Spin_2(\alpha)D$ and:
	\be
	\det (\Ad_0^{(2)}(b))=\twopartdef{+1}{b\in \Spin_2(\alpha)}{-1}{b\in \Spin_2(\alpha)D}~~,
	\ee
	which implies: 
	\be
	{\widehat \Ad}_0(b)(v)=\det(\Ad_0^{(2)}(b)) v~~\forall v\in V~~\mathrm{and}~~b\in \Pin_2(\alpha)~~.
	\ee
	For any $a\in \Spin(V,h)$, $b\in \Pin_2(\alpha)$, $v\in V$ and $w\in
	\R e_{d+1}\oplus \R e_{d+2}$, we compute:
	\beqa
	&&{\widehat \Ad}_0(ab)(v)=({\widehat \Ad}_0(a)\circ {\widehat \Ad}_0(b))(v)=   \det (\Ad_0^{(2)}(b))\Ad_0(a)(v)\\
	&&{\widehat \Ad}_0(ab)(w)=({\widehat \Ad}_0(a)\circ {\widehat \Ad}_0(b))(w)=\Ad^{(2)}_0(b)(w)~~,
	\eeqa
	which gives \eqref{Eq1}. Since $D\in \Cl_-({\hat V},{\hat h}_\alpha)$, we have:
	\be
	{\widehat \tAd}_0(ab)=\twopartdef{~~{\widehat \Ad}_0(ab)}{b\in \Spin_2(\alpha)}{-{\widehat \Ad}_0(ab)}{b\in \Spin_2(\alpha)D}~~(a\in \Spin(V,h))~~,
	\ee
	i.e. ${\widehat \tAd}_0(ab)=(\det \Ad_0^{(2)}(b)){\widehat \Ad}(ab)$. Thus: 
	\beqa
	&&{\widehat \tAd}_0(ab)(v)= \Ad_0(a)(v)\\
	&&{\widehat \tAd}_0(ab)(w)=(\det \Ad_0^{(2)}(b))\Ad^{(2)}_0(b)(w)=\tAd_0^{(2)}(b)~~,
	\eeqa
	which gives \eqref{Eq2}. 
	 \end{proof}

\noindent Equations \eqref{Eq1} and \eqref{Eq2} imply that the
following relations hold for all $g\in \Spin_\alpha^o(V,h)$:
\beqan
\label{PinModel}
&& \tlambda(g)={\widehat \Ad}_0(j(g))|_V~~,~~\lambda(g)={\widehat \tAd}_0(j(g))|_V~~,\nn\\
&& \mu(g)={\widehat \Ad}_0(j(g))|_{\R^2}~~,~~\tmu(g)={\widehat \tAd}_0(j(g))|_{\R^2}~~.
\eeqan
We conclude that the following commutative diagrams with exact rows
(morphisms of short exact sequences) hold:
\begin{equation}
\label{Gdiagram1}
\scalebox{1}{
	\xymatrix{
		1 \ar[r] & ~\Z_2~\ar@{=}[d] \ar[r]~ &~\Spin^o_\alpha(V,h) ~\ar[r]^{\!\!\!\!\!\!\!\trho_\alpha}~~ \ar[d]^{j} & \O(V,h){\hat \times} \O(2) \ar[d]^{G}~\ar[r] &1\\
		1 \ar[r]~ & ~\Z_2 \ar[r]~& ~\Pin({\hat V},{\hat h}_\alpha) ~\ar[r]^{\!\!\!{\widehat \Ad}_0}& {\hat \O}({\hat V},{\hat h}_{\alpha})\ar[r]&1
}}
\end{equation}
and:
\begin{equation}
\label{Gdiagram2}
\scalebox{1}{
	\xymatrix{
		1 \ar[r] & ~\Z_2~\ar@{=}[d] \ar[r]~ &~\Spin^o_\alpha(V,h) ~\ar[r]^{\!\!\!\!\!\!\!\rho_\alpha}~~ \ar[d]^{j} & \SO(V,h)\times \O(2) \ar[d]^{G}~\ar[r] &1\\
		1 \ar[r]~ & ~\Z_2 \ar[r]~& ~\Pin({\hat V},{\hat h}_\alpha) ~\ar[r]^{\!\!\!{\widehat \tAd}_0}& \O({\hat V},{\hat h}_{\alpha})\ar[r]&1
}}
\end{equation}

\subsection{Realization of $\Spin_\alpha^o(V,h)$ as a subgroup of a spin group}

Let $V'\eqdef V\oplus \R^3$ and $h'_\alpha:V'\times V'\rightarrow \R$
be the nondegenerate symmetric bilinear form given by:
\be
h'_\alpha(x+w+t,x'+w'+t')\eqdef h(x,y)+\alpha\langle w,w'\rangle -ss'~~\forall x,x'\in V~~,~~w,w'\in \R^2~~,~~t,t'\in \R~~.
\ee
Consider the injective group morphisms $F,\tF:\O(V,h)\times
\O(2)\rightarrow \O(V',h'_\alpha)$ given by:
\beqan
\label{F}
&& F(A, B)\eqdef (\det B) A\oplus (\det B)B\oplus \det B\nn\\
&& \tF(A, B)\eqdef A\oplus B\oplus \det B~~
\eeqan
Notice that $F(\O(V,h){\hat \times} \O(2))\subset \SO(V',h'_\alpha)$
(for odd dimension) and $\tF(\SO(V,h)\times \O(2))\subset
\SO(V',h'_\alpha)$.  We have $F\circ \trho_\alpha=\tF\circ
\rho_\alpha$, namely:
\ben
\label{Ftrho}
(F\circ \trho_\alpha)([a,u])=(\tF\circ \rho_\alpha)([a,u])=\Ad_0(a)\oplus (\det \Ad_0^{(2)}(u))\Ad_0^{(2)}(u)\oplus \det \Ad_0^{(2)}(u)~~,
\een
where we noticed that $\det\circ\tAd_0^{(2)}=\det \circ \Ad_0^{(2)}$
and $\tAd_0^{(2)}(u)=(\det \Ad_0^{(2)}(u))\Ad_0^2(u)$.

\begin{prop}
	\label{prop:jprime}
	There exists an injective morphism of groups
	$j':\Spin_\alpha^o(V,h)\rightarrow \Spin(V',h'_\alpha)$ such that:
	\ben
	\label{E1}
	\Ad_0'\circ j'=F\circ \trho_\alpha=\tF\circ \rho_\alpha~~,
	\een
	where $\Ad_0'$ is the vector representation of
	$\Spin(V',h'_\alpha)$.
\end{prop}

\begin{proof} 
	Let $e_{d+1},e_{d+2},e_{d+3}$ be the canonical basis of $\R^3$.  The
	following relations hold in $\Cl(V',h'_\alpha)$:
	\be
	e_{d+1}^2=e_{d+2}^2=\alpha~~,~~e_{d+3}^2=-1~~.
	\ee
	Let $J'\eqdef e_{d+1}e_{d+2}$, $D'\eqdef e_{d+1}e_{d+3}$. Then:
	\be
	(J')^2=-1~~,~~(D')^2=\alpha~~,~~D'J'=-J'D'~~.
	\ee
	Since $(e_{d+1}e_{d+3})^2=(e_{d+2}e_{d+3})^2=\alpha$, there exists a
	unique unital monomorphism of algebras ${\hat s}':\Cl({\hat V}, {\hat
		h}_\alpha)\rightarrow \Cl(V',h'_\alpha)$ which satisfies:
	\be
	{\hat s}'(e_i)=e_i~~\forall i=1\ldots d~~,~~{\hat s}'(e_{d+1})=e_{d+1}e_{d+3}=D'~~,~~{\hat s}'(e_{d+2})=e_{d+2}e_{d+3}~~.
	\ee
	The composition $j'\eqdef s'\circ j:\Spin_\alpha^o(V,h)\rightarrow
	\Spin(V',h'_\alpha)$ is an injective morphism of groups. Notice that
	$j'|_{\Spin(V,h)}=\id_{\Spin(V,h)}$. For $a\in \Spin(V,h)$ and $v\in
	V$, we have:
	\beqa
	&& \Ad_0'(s'(a))(v) = \Ad_0(a)(v)\, , \qquad \Ad_0'(s'(a))(e_{d+1})=e_{d+1}~,~ \nn\\ && \Ad'_0(s'(a))(e_{d+2}) = e_{d+2}\, , \qquad \Ad'_0(s'(a))(e_{d+3})=e_{d+3}\nn\\
	&& \Ad_0'(s'(e_{d+1}))(v)=v\, , \qquad  \Ad'_0(s'(e_{d+1}))(e_{d+1})=-e_{d+1}\nn\\ && \Ad'_0(s'(e_{d+1}))(e_{d+2})=+e_{d+2}\, , \qquad \Ad'_0(s'(e_{d+1}))(e_{d+3})=-e_{d+3}\nn\\
	&& \Ad_0'(s'(e_{d+2}))(v)=v\, , \qquad \Ad'_0(s'(e_{d+2}))(e_{d+1})=+e_{d+1}\, , \nn\\ && \Ad'_0(s'(e_{d+2}))(e_{d+2})=-e_{d+2}\, , \qquad \Ad'_0(s'(e_{d+2}))(e_{d+3})=-e_{d+3}\nn~~,
	\eeqa
	since $s'(e_{d+1})=e_{d+1}e_{d+3}$ and  $s'(e_{d+2})=e_{d+2}e_{d+3}$. On the other hand, we have: 
	\beqa
	&& \Ad_0^{(2)}(e_{d+1})(e_{d+1})=e_{d+1}~~,~~\Ad_0^{(2)}(e_{d+1})(e_{d+2})=-e_{d+2}~~\nn\\
	&& \Ad_0^{(2)}(e_{d+2})(e_{d+1})=-e_{d+1}~~,~~\Ad_0^{(2)}(e_{d+2})(e_{d+2})=+e_{d+2}~~\nn\\
	&& \det (\Ad_0^{(2)}(e_{d+1}))=\det (\Ad_0^{(2)}(e_{d+2}))=-1~~,
	\eeqa
	so comparison with \eqref{Ftrho} gives: 
	\beqa
	&& (\Ad'_0\circ j')([a,1])=(F\circ \trho_\alpha)([a,1])\\
	&& (\Ad'_0\circ j')([1,e_{d+1}])=(F\circ \trho_\alpha)([1,e_{d+1}])\\
	&& (\Ad'_0\circ j')([1,e_{d+2}])=(F\circ \trho_\alpha)([1,e_{d+2}])~~,
	\eeqa
	which implies $\Ad'_0\circ j'=F\circ \trho_\alpha$ since the group
	$\Spin_\alpha^o(V,h)$ is generated by $\Spin(V,h)$ and by the elements
	$[1,e_{d+1}]$ and $[1,e_{d+2}]$.
	 \end{proof}

\noindent 
We conclude that the following morphisms of short exact sequences hold:

\begin{equation}
\label{Fdiagram1}
\scalebox{1}{
	\xymatrix{
		1 \ar[r] & ~\Z_2~\ar@{=}[d] \ar[r]~ &~\Spin^o_\alpha(V,h) ~\ar[r]^{\!\!\!\!\!\!\!\trho_\alpha}~~ \ar[d]^{j'} & \O(V,h){\hat \times} \O(2) \ar[d]^F~\ar[r] &1\\
		1 \ar[r]~ & ~\Z_2 \ar[r]~& ~\Spin(V',h'_\alpha) ~\ar[r]^{\!\!\!\Ad'_0}& \SO(V',h'_{\alpha})\ar[r]&1
}}
\end{equation}
and:
\begin{equation}
\label{Fdiagram2}
\scalebox{1}{
	\xymatrix{
		1 \ar[r] & ~\Z_2~\ar@{=}[d] \ar[r]~ &~\Spin^o_\alpha(V,h) ~\ar[r]^{\!\!\!\!\!\!\!\rho_\alpha}~~ \ar[d]^{j'} & \SO(V,h)\times \O(2) \ar[d]^{\tF}~\ar[r] &1\\
		1 \ar[r]~ & ~\Z_2 \ar[r]~& ~\Spin(V',h'_\alpha) ~\ar[r]^{\!\!\!\Ad'_0}& \SO(V',h'_{\alpha})\ar[r]&1
}}
\end{equation}

\section{$\Spin_\alpha^o$ structures}
\label{sec:spinostructures}

Let $M$ be a smooth, connected and paracompact manifold of dimension $d$. 

\begin{definition}
	\label{def:spinostructures}
	Let $P_{\check{\O}}$ be a principal $\check{\O}(V,h)$-bundle over
	$M$. A {\bf $\Spin^o_\alpha$ structure} on $P_{\check{\O}}$ is a pair
	$(Q,\tilde{\Lambda}_{\alpha})$ where $Q$ is a principal
	$\Spin^o_\alpha(V,h)$-bundle over $M$ and
	$\tilde{\Lambda}_{\alpha}:Q\rightarrow P_{\check{\O}}$ is a
	$\tlambda_{\alpha}$-equivariant map fitting into the following
	commutative diagram:
	\begin{equation*}
	\scalebox{1.0}{
		\xymatrix{
			Q\times \Spin^{o}_{\alpha}(V,h)\,\ar[d]_{\tilde{\Lambda}_{\alpha}\times\tilde{\lambda}_{\alpha}}\ar[r] ~ &~ Q\ar[d]^{\tilde{\Lambda}_{\alpha}}\\
			P_{\check{\O}}\times \check{\O}(V,h)\ar[r] & ~ P_{\check{\O}}\\
	}}
	\end{equation*}
	Here horizontal arrows denote the right bundle action of the
	corresponding group.
\end{definition}

\begin{remark}
	Notice that $\Spin^{o}_{\alpha}$ structures are defined using the
	{\em twisted} vector representation $\tilde{\lambda}_{\alpha}\colon
	\Spin^{o}_{\alpha}(V,h)\to \check{\O}(V,h)$ instead of the
	usual vector representation $\lambda_{\alpha}\colon
	\Spin^{o}_{\alpha}(V,h)\to \SO(V,h)$. When $d$ is odd, the former
	covers the full orthogonal group $\O(V,h)$ whereas the latter only
	covers the special orthogonal group $\SO(V,h)$. As we shall see
	later, this fact allows us to define $\Spin^{o}_{\alpha}$
	structures on non-orientable odd-dimensional manifolds.
\end{remark}

\begin{remark}
	The notion of $\Spin^{o}_{\alpha}$ structure introduced above may seem
	somewhat artificial. However, as we will explain in Section
	\ref{sec:elementary}, this turns out to be the appropriate spinorial
	structure for describing bundles of simple real Clifford modules in
	signatures $(p,q)$ which satisfy the condition $p-q\equiv_{8} 3, 7$.
\end{remark}

\noindent
Given a $\Spin^{o}_{\alpha}$ structure on a principal bundle
$P_{\check{\O}}$, we can construct certain associated principal
bundles by using the different representations of
$\Spin^{o}_{\alpha}(V,h)$ introduced in Definition
\ref{def:elementaryreps} of Subsection \ref{sec:elementaryrep}:

\begin{definition}
	The {\bf characteristic bundle} $P_{\mu_{\alpha}}$ defined by a
	$\Spin^{o}_{\alpha}$ structure $(Q,\tilde{\Lambda}_{\alpha})$ is the
	principal $\O(2)$-bundle associated to $Q$ through the characteristic
	representation $\mu_{\alpha}\colon \Spin^{o}_{\alpha}(V,h)\to \O(2)$,
	that is:
	\begin{equation*}
	P_{\mu_{\alpha}}\eqdef Q\times_{\mu_{\alpha}} \O(2)\, .
	\end{equation*}
	We denote by $\mathfrak{U}_{\alpha}:Q\rightarrow P_{\mu_{\alpha}}$ the
	corresponding $\mu_{\alpha}$-equivariant bundle map.
\end{definition}

\begin{definition}
	The {\bf twisted basic bundle} $P_{\tilde{\rho}_{\alpha}}$ defined
	by a $\Spin^{o}_{\alpha}$ structure $(Q,\tilde{\Lambda}_{\alpha})$ is
	the principal $\O(V,h)\hat{\times}\O(2)$-bundle associated to $Q$
	through the twisted basic representation $\tilde{\rho}_{\alpha}\colon
	\Spin^{o}_{\alpha}(V,h)\to \O(V,h)\hat{\times}\O(2)$, that is:
	\begin{equation*}
	P_{\tilde{\rho}_{\alpha}}\eqdef Q\times_{\tilde{\rho}_{\alpha}} (\O(V,h)\hat{\times}\O(2))\, .
	\end{equation*}
	We denote by $\mathfrak{R}_{\alpha}:Q\rightarrow
	P_{\tilde{\rho}_{\alpha}}$ the corresponding
	$\tilde{\rho}_{\alpha}$-equivariant bundle map.
\end{definition}

\begin{prop}
	\label{prop:SpinoE}
	Let $P_{\check{\O}}$ be a principal $\check{\O}(V,h)$-bundle over
	$M$. Then the following statements are equivalent:
	\begin{enumerate}[(a)]
		\itemsep 0.0em
		\item $P_{\check{\O}}$ admits a $\Spin^o_\alpha$ structure $(Q,\tilde{\Lambda}_{\alpha})$.
		\item There exists a principal $\O(2)$-bundle $E$ over $M$ such that
		the fibered product $P_{\check{\O}}\times_M E$ admits a $\trho_{\alpha}$ equivariant reduction:
		\begin{equation*}
		\mathfrak{R}^0_{\alpha}\colon Q\rightarrow P_{\check{\O}}\times_M E\, ,
		\end{equation*}
		to a $\Spin_\alpha^o(V,h)$-bundle $Q$. 
	\end{enumerate}
	
	\noindent
	In particular, the image $P_{\tilde{\rho}_{\alpha}} \eqdef \mathfrak{R}^0_{\alpha}(Q)\subset P_{\check{\O}}\times_M E$ of $Q$ by $\mathfrak{R}^0_{\alpha}$ is the twisted basic bundle associated to the $\Spin^{o}_{\alpha}$ structure $(Q,\tilde{\Lambda}_{\alpha})$, and the the corestriction of $\mathfrak{R}^0_{\alpha}$ to its image gives the $\tilde{\rho}_{\alpha}$-equivariant bundle map $\mathfrak{R}_{\alpha}:Q\rightarrow
	P_{\tilde{\rho}_{\alpha}}$.
\end{prop}

\begin{proof} 
	Suppose first that $(a)$ holds, i.e. that $P_{\check{\O}}$ admits a
	$\Spin^o_\alpha$ structure $(Q,\tilde{\Lambda}_{\alpha})$. We take $E
	= P_{\mu_{\alpha}}$ to be the characteristic bundle associated to
	$Q$. Using the fact that $(Q,\tilde{\Lambda}_{\alpha})$ is a $\Spin^{o}_{\alpha}$ 
	structure covering $P_{\check{\O}}$, we define $\mathfrak{R}^0_{\alpha}$ as follows:
	\begin{equation*}
	\mathfrak{R}^0_{\alpha}\colon Q\rightarrow P_{\check{\O}}\times_M E\, , \qquad q \mapsto (\tilde{\Lambda}_{\alpha}(q),\mathfrak{U}_{\alpha}(q)) \, .
	\end{equation*}
	
	\noindent
	Direct computation shows that $\mathfrak{R}^0_{\alpha}$ is $\tilde{\rho}_{\alpha}$ equivariant. The image of $\mathfrak{R}^0_{\alpha}$ is isomorphic to the twisted basic bundle associated to $Q$ through the following isomorphism of principal bundles
	\begin{equation*}
	\cI \colon P_{\tilde{\rho}_{\alpha}}\xrightarrow{\simeq} \mathfrak{R}^0_{\alpha}(Q)\, , \qquad [q,(g_0,u_0)]\mapsto (\tilde{\Lambda}_{\alpha}(q) g_0,\mathfrak{U}_{\alpha}(q) u_0)\, .
	\end{equation*}
	The equivariant map $\mathfrak{R}_{\alpha}\colon Q\rightarrow P_{\trho_{\alpha}}$ corresponds, upon use of $\cI$, to the corestriction of $\mathfrak{R}^0_{\alpha}$ to its image, that is:
	\begin{equation*}
	\mathfrak{R}_{\alpha} = \cI^{-1}\circ \mathfrak{R}^0_{\alpha}\colon Q \to P_{\tilde{\rho}_{\alpha}}\, .
	\end{equation*}
	Now suppose $(b)$ holds. Then the principal
	$\Spin^{o}_{\alpha}(V,h)$-bundle $Q$ becomes a $\Spin^{o}_{\alpha}$
	structure on $P_{\check{\O}}$ when endowed with the bundle map
	$\tilde{\Lambda}_{\alpha}\colon Q\to P_{\check{\O}}$ obtained by
	post-composing $\mathfrak{R}_{\alpha}$ with the fiberwise projection
	to the first factor. Indeed, the map $\tilde{\Lambda}_{\alpha}$
	obtained in this manner is equivariant with respect to the twisted
	vector representation $\tilde{\lambda}_{\alpha}\colon
	\Spin^{o}_{\alpha}(V,h) \to \check{\O}(V,h)$ (this follows from the
	fact that $\mathfrak{R}_{\alpha}$ is
	$\tilde{\rho}_{\alpha}$-equivariant). Hence
	$(Q,\tilde{\Lambda}_{\alpha})$ satisfies the conditions required in
	Definition \ref{def:spinostructures}.
	 
\end{proof}

\noindent We next define isomorphisms of $\Spin^{o}_{\alpha}$ structures.

\begin{definition}
	Let $\tilde{\Lambda}^{1}_{\alpha}\colon Q^{1} \rightarrow
	P_{\check{\O}}$ and $\tilde{\Lambda}^{2}_{\alpha}\colon
	Q^{2}\rightarrow P_{\check{\O}}$ be two $\Spin^{o}_{\alpha}$
	structures over $P$. An {\bf isomorphism of $\Spin^{o}_{\alpha}$
		structures} is an isomorphism $F\colon Q^{1}\to Q^{2}$ of principal
	$\Spin^{o}_{\alpha}$-bundles such that
	$\tilde{\Lambda}^{2}_{\alpha}\circ F= \tilde{\Lambda}^{1}_{\alpha}$,
	i.e., such that the following diagram commutes:
	\begin{equation}
	\scalebox{1.0}{
		\xymatrix{
			Q^{1}\,\ar[d]_{\tilde{\Lambda}^{1}_{\alpha}}\ar[r]^{F} ~ & ~ Q^{2}\ar[d]^{\tilde{\Lambda}^{2}_{\alpha}}\\
			P_{\check{\O}} \ar[r]^{\mathrm{id}} ~ & ~ P_{\check{\O}}\\
	}}
	\end{equation}
\end{definition}

\subsection{$\Spin^o_\alpha$ structures on pseudo-Riemannian manifolds}

\noindent Let $(M,g)$ be a connected pseudo-Riemannian manifold of
dimension $d$ and signature $(p,q)$ and let $(V,h)$ be a quadratic
space of the same dimension and signature. Furthermore, assume that
$M$ is oriented when $d$ is even. In this situation, the
pseudo-Riemannian manifold $(M,g)$ carries a natural
principal $\check{\O}(V,h)$-bundle $P_{\check{O}}(M,g)$, which is
defined as follows (cf. Definition \eqref{checkO}):
\be
P_{\check{O}}(M,g)\eqdef \twopartdef{\mathrm{bundle~of~oriented~pseudo-orthonormal~frames~of}~(M,g)}{d=\ev}{\mathrm{bundle~of~all~pseudo-orthonormal~frames~of} (M,g)}{d=\odd}~~. 
\ee

\begin{definition}
	Let $(M,g)$ be as above. A {\em $\Spin^o_\alpha$}-structure on $(M,g)$
	is a $\Spin^o_\alpha$ structure on the principal
	$\check{\O}(V,h)$-bundle $P_{\check{O}}(M,g)$.
\end{definition}

\begin{definition}
	Let $\mathrm{C}^{o}(M,g)$ be the groupoid whose objects are
	$\Spin^{o}_{\alpha}$ structures on $(M,g)$ and whose arrows are
	isomorphisms of $\Spin^{o}_{\alpha}$ structures.
\end{definition}

\begin{remark}
	Notice that two principal $\Spin^{o}_{\alpha}$-bundles on $(M,g)$ may be
	isomorphic as principal $\Spin^{o}_{\alpha}$-bundles without
	being isomorphic as $\Spin^{o}_{\alpha}$ structures.
\end{remark}

\begin{remark} 
	As explained in Section \ref{sec:spinogroups}, we have a short exact
	sequence:
	\begin{equation}
	\label{eq:shortSpinalpha}
	1\to \Spin^{c}(V,h)\xrightarrow{i} \Spin^{o}_\alpha(V,h) \xrightarrow{\tilde{\eta}_{\alpha}} \mathbb{Z}_{2}\to 1\, ,
	\end{equation}
	which induces the following exact sequence of pointed sets in
	$\check{\mathrm{C}}$ech-cohomology:
	\begin{equation}
	\label{eq:shortSpincohomology}
	H^{1}(M,\Spin^{c}(V,h))\xrightarrow{i_{\ast}} H^{1}(M, \Spin^{o}(V,h))\xrightarrow{\tilde{\eta}_{\alpha \ast}} H^{1}(M, \mathbb{Z}_{2})\, .
	\end{equation}
	Hence every principal $\Spin^{o}_{\alpha}(V,h)$-bundle $Q$ defined on
	$M$ determines a real line bundle $L$ defined on $M$ whose isomorphism
	class satisfies:
	\begin{equation}
	[L]=\tilde{\eta}_{\alpha \ast}([Q])\in H^{1}(M, \mathbb{Z}_{2})\, .
	\end{equation}
	Intuitively, we can view a $\Spin^{o}_{\alpha}$ structure on $(M,g)$
	as a $\Spin^{c}(V,h)$ structure ``twisted'' by the real line bundle
	$L$. It is not hard to see that $L$ is isomorphic with the determinant
	line bundle of the characteristic $\O(2)$-bundle $P_{\mu_{\alpha}}$
	associated to $Q$.  Exactness of \eqref{eq:shortSpincohomology} shows
	that a $\Spin^{o}_{\alpha}$ structure $Q$ reduces to a $\Spin^{c}$
	structure iff $L$ is trivial, i.e. iff $P_{\mu_{\alpha}}$ is
	orientable. We will study in detail reductions of
	$\Spin^{o}_{\alpha}$-bundles in Section \ref{sec:elementarypinor}.
\end{remark}

\subsection{Topological obstructions}

\noindent We now turn our attention to the topological obstruction for
existence of a $\Spin^{o}_{\alpha}$ structure on a principal
$\check{\O}(V,h)$-bundle $P_{\check{\O}}$ defined over $M$. We will
need the following result:

\begin{lemma}
	\label{lemma:seq}
	Consider a commutative diagram of morphisms of Lie groups of the following
	form:
	\begin{equation}
	\label{diagram1}
	\scalebox{1}{
		\xymatrix{
			1 \ar[r] & ~\Z_2\ar@{=}[d] \ar[r]~ &~G ~\ar[r]^{p} \ar[d]^f & K \ar[d]^g~\ar[r] &1\\
			1 \ar[r] & ~\Z_2 \ar[r]~& ~G'~\ar[r]_{p'}& K'\ar[r]&1
	}}
	\end{equation}
	Then a principal $K$-bundle $P$ over $M$ admits a $G$-reduction along
	$p$ iff the principal $K'$-bundle $P'\eqdef g_\ast(P)$ admits a
	$G'$-reduction along $p'$.
\end{lemma}

\begin{proof}
	The given morphism of short exact sequences induces a commutative
	diagram with exact columns in the category of pointed sets, where
	$\partial$ and $\partial'$ are the connecting maps:
	\begin{equation}
	\label{diagram2}
	\scalebox{1}{
		\xymatrix{
			H^1(M,G) \ar[d]_{p_\ast}~  \ar[r]^{f_\ast} & ~H^1(M,G') ~ \ar[d]^{p'_\ast} \\
			H^1(M,K)  \ar[d]_{\partial}~  \ar[r]^{g_\ast} & ~H^1(M,K') ~ \ar[d]^{\partial '} \\
			H^2(M,\Z_2)   \ar@{=}[r]& ~H^2(M,\Z_2) 
	}}
	\end{equation}
	The principal $K$-bundle $P$ admits a $G$-reduction along $p$ iff its
	isomorphism class $[P]\in H^1(M,K)$ belongs to the image of $p_\ast$
	while the principal $K'$-bundle $P'$ admits a $G'$-reduction along
	$p'$ iff its isomorphism class $[P']=g_\ast([P])\in H^1(M,K')$ belongs
	to the image of $p'_\ast$. A simple diagram chase using exactness of
	the columns shows that $[P]\in \im p_\ast$ iff $g_\ast([P])\in \im
	p'_\ast$. 
	 \end{proof}

\noindent
We now consider the topological obstructions to existence of a
$\Spin^{o}_{\alpha}$ structure on a principal $\check{\O}(V,h)$-bundle
when $d=\dim V$ is odd.  We focus on the case when $d$ is odd since
this case is relevant for applications to bundles of simple real
Clifford modules of `complex type'.

\begin{lemma}
	Let $P = P_{\check{\O}}$ be a principal $\check{\O}(V,h)$-bundle defined on $M$, 
	such $d=\dim V$ is odd. Then the following statements are equivalent:
	\begin{enumerate}[(a)]
		\itemsep 0.0em
		\item $P$ admits a $\Spin_\alpha^o$ structure.
		\item There exists a principal $\O(2)$-bundle $E$ such that the
		principal $\O(V', h'_\alpha)$-bundle $P'_\alpha \eqdef [(\det
		E)P]\times_M (\det E)E\times_M \det E$ admits a
		$\Spin$ structure.
	\end{enumerate}
\end{lemma}

\begin{proof}
	Follows from Proposition \ref{prop:SpinoE} by applying Lemma
	\ref{lemma:seq} to the commutative diagram of short exact sequences
	\eqref{Fdiagram1}, where the morphism $F$ is defined in \eqref{F}, the
	embedding $j^{\prime}$ is defined in Proposition \ref{prop:jprime} and
	$\tilde{\rho}_{\alpha}$ denotes the twisted basic representation of
	$\Spin^{o}_{\alpha}(V,h)$. More explicitly, diagram \eqref{Fdiagram1}
	implies the following commutative diagram with exact columns in the
	category of pointed sets:
	\begin{equation}
	\scalebox{1}{
		\xymatrix{
			H^1(M,\Spin^{o}_{\alpha}(V,h)) \ar[d]_{\tilde{\rho}_{\alpha \ast}}~  \ar[r]^{j^{\prime}_\ast} & ~H^1(M,\Spin(V^{\prime},h^{\prime}_{\alpha})) ~ \ar[d]^{\Ad^{\prime}_{\ast}} \\
			H^1(M,\O(V,h)\hat{\times}\O(2))  \ar[d]_{\partial}~  \ar[r]^{F_\ast} & ~H^1(M,\SO(V^{\prime},h^{\prime}_{\alpha})) ~ \ar[d]^{\partial '} \\
			H^2(M,\Z_2)   \ar@{=}[r]& ~H^2(M,\Z_2) 
	}}
	\end{equation}
	Applying Lemma \ref{lemma:seq} we find that $[P]\in
	H^1(M,\O(V,h)\hat{\times}\O(2))$ admits a $\Spin^{o}_{\alpha}$
	structure iff $F_{\ast}([P])$ admits a $\Spin$
	structure. Using the explicit form of $F$, we deduce that
	$F_{\ast}([P])$ can be represented by a principal bundle of the form
	$P^{\prime}_\alpha \eqdef [(\det E)P]\times_M (\det E)E\times_M \det
	E$. Notice that the connecting map $\partial'$ is
	given by \cite{Karoubi}:
	\be
	\partial'=\w_2^++\w_2^-~~.
	\ee
	\noindent
\end{proof}

\noindent 
Let $L$ be a principal $\Z_2$-bundle and $P$ be a principal
$\O(r)$-bundle over $M$, where $r\geq 1$. Recall that the center
$Z(\O(r))\simeq \Z_2$. Then we have \cite[page 404]{Hausmann}:
\be
\w(L P)=\sum_{k=0}^r (1+\w_1(L))^k\w_{r-k}(P)~~,
\ee
which gives:
\beqan
\label{wkrel}
&& \w_1(L P)= \w_1(P)+r\w_1(L) \\
&& \w_2(L P)= \w_2(P)+(r-1)\w_1(P)\w_1(L)+\frac{r(r-1)}{2}\w_1(L)^2~~.\nn
\eeqan
Also recall that $\w_1(\det P)=\w_1(P)$. The total Stiefel-Whitney
class satisfies $\w(P_1\times_M P_2)=\w(P_1)\w(P_2)$ for any principal
$\O(r_i)$-bundles $P_i$ ($i=1,2$), which gives:
\beqan
\label{wrel}
&& \w_1(P_1\times_M P_2)=\w_1(P_1)+\w_1(P_2)\\
&& \w_2(P_1\times_M P_2)=\w_2(P_1)+\w_2(P_2)+\w_1(P_1)\w_1(P_2)~~.\nn
\eeqan

\begin{thm}
	\label{thm:SpinoObs}
	A principal $\O(V,h)$-bundle $P$ defined on $M$ admits a
	$\Spin_\alpha^o$ structure iff there exists a principal $\O(2)$-bundle
	$E$ such that the following conditions are satisfied:
	\begin{eqnarray}
	\label{obso}
	&& \w_1^+(P)+\w_1^-(P)=d\w_1(E)\\
	&& \w_2^+(P)+\w_2^-(P)+\w_1(E)(p\w_1^+(P)+q\w_1^-(P))=\w_2(E)+\nn\\
	&&+\left[\delta_{\alpha,-1}+\frac{p(p+1)}{2}+\frac{q(q+1)}{2}\right]\w_1(E)^2~~.\nn
	\end{eqnarray}
\end{thm}

\begin{proof}
	Recall from \cite{Karoubi} that the principal $\O(V',
	h'_\alpha)$-bundle $P':=P'_\alpha$ admits a $\Spin$ structure iff the
	following conditions are satisfied:
	\ben
	\label{Kconds}
	\w_1^+(P')+\w_1^-(P')=\w_2^+(P')+\w_2^-(P')=0~~.
	\een
	Notice that $P^+$ is an $\O(p)$-bundle while $P^-$ is an
	$\O(q)$-bundle. Also notice that $\w_1(\det E)=\w_1(E)$. Using
	\eqref{wkrel} and the fact that $E$ is an $\O(2)$ bundle, we find:
	\beqan
	\label{Erels}
	&& \w_1((\det E) E)=\w_1(E)~~\nn\\
	&& \w_2((\det E) E)=\w_2(E)~~\nn\\
	&& \w_1((\det E)E \times_M \det E)=0\\
	&& \w_2((\det E)E \times_M \det E)=\w_2(E)+\w_1(E)^2\nn~~.
	\eeqan
	We distinguish the cases:
	\begin{enumerate}[1.]
		\itemsep 0.0em
		\item For $\alpha=+1$, we have $(P')^+=[(\det E)P^+]\times_M (\det
		E)E$ and $(P')^-=[(\det E)P^-]\times_M \det E$. Using \eqref{wkrel},
		\eqref{wrel} and \eqref{Erels}, we compute:
		\beqa
		&& \w_1^+(P')= \w_1^+(P)+(p+1)\w_1(E)~~\nn\\
		&&\w_1^-(P')=\w_1^-(P)+(q+1)\w_1(E)\nn\\
		&&\w_2^+(P')= \w_2^+(P)+(p-1)\w_1^+(P)\w_1(E)+\frac{p(p-1)}{2}\w_1(E)^2+\w_2(E)+\nn\\
		&&+\w_1(E)(\w_1^+(P)+p\w_1(E))=\w_2^+(P)+\w_2(E)+p\w_1^+(P)\w_1(E)+\frac{p(p+1)}{2}\w_1(E)^2\nn\\
		&& \w_2^-(P')=\w_2^-(P)+(q-1)\w_1^-(P)\w_1(E)+\frac{q(q-1)}{2}\w_1(E)^2+\nn\\
		&&+\w_1(E)(\w_1^-(P)+q\w_1(E))=\w_2^-(P)+q\w_1^-(P)\w_1(E)+\frac{q(q+1)}{2}\w_1(E)^2~~.
		\eeqa
		Thus equations \eqref{Kconds} reduce to \eqref{obso}.  
		
		\
		
		\item For $\alpha=-1$, we have $(P')^+=[(\det E)P^+]$ and
		$(P')^-=[(\det E)P^-]\times_M (\det E)E \times_M \det E$. Thus:
		\beqa
		&& \w_1^+(P')= \w_1^+(P)+p\w_1(E) ~~\nn\\
		&& \w_1^-(P')= \w_1^-(P)+q\w_1(E)~~\nn\\
		&& \w_2^+(P')=\w_2^+(P)+(p-1)\w_1(E)\w_1^+(P)+\frac{p(p-1)}{2}\w_1(E)^2 \nn\\
		&& \w_2^-(P')=\w_2^-(P)+(q-1)\w_1(E)\w_1^-(P)+\frac{q(q-1)}{2}\w_1(E)^2+\w_2(E)+\nn\\
		&&+\w_1(E)^2=\w_2^-(P)+\w_2(E)+(q-1)\w_1(E)\w_1^-(P)+[1+\frac{q(q-1)}{2}]\w_1(E)^2~~.
		\eeqa
		Thus equations \eqref{Kconds} reduce to \eqref{obso}, where to obtain the second
		relation in \eqref{obso} we used the first relation and noticed that:
		\be
		1+d+\frac{p(p-1)}{2}+\frac{q(q-1)}{2}=1+\frac{p(p+1)}{2}+\frac{q(q+1)}{2}~~
		\ee
		since $d=p+q$. 
	\end{enumerate}
\end{proof}

\subsection{Application to pseudo-Riemannian manifolds of positive and negative sigature}

\noindent Let us consider Theorem \ref{thm:SpinoObs} in the particular
case of odd-dimensional pseudo-Riemannian manifolds $(M,g)$ of
positive and negative signatures. In this case, we have
$\check{\O}(V,h)\simeq \O(d)$ and $pq=0$, where $d$ is the dimension
of $M$. We focus exclusively on $\Spin^{o}_{-}$ structures in
dimension $d\equiv_{8} 3$ and $\Spin^{o}_{+}$ structures in dimension
$d\equiv_{8} 7$ --- since, as we shall see in Section
\ref{sec:elementarypinor}, these are the cases most relevant for
applications to spinorial geometry. For ease of notation, let $P$
denote the principal bundle of pseudo-orthonormal frames of $(M,g)$.

\paragraph{Positive signature in dimension $d \equiv_{8} 3$}

The topological obstruction to existence of a $\Spin^{o}_{-}$
structure on a Riemannian manifold of dimension $d \equiv_{8} 3$ is
given by the condition:
\begin{equation}
\w_{1}(P) = \w_{1}(E)\qquad \w_{2}(P) = \w_{2}(E)\, .
\end{equation}

\paragraph{Positive signature in dimension $d \equiv_{8} 7$}

The topological obstruction to existence of a
$\Spin^{o}_{+}$ structure on a Riemannian manifold of dimension $d
\equiv_{8} 7$ is given by the condition:
\begin{equation}
\label{eq:topobsR7}
\w_{1}(P) = \w_{1}(E)\qquad \w_{2}(P) + \w_{1}(P)^{2} + \w_{2}(E) = 0\, .
\end{equation}
In this particular case, condition \eqref{eq:topobsR7} also implies
existence of a complex Lipschitz structure \cite{FriedrichTrautman,ComplexLipschitz},
which in turn amounts to existence of a bundle of \emph{faithful
	complex} Clifford modules over the real Clifford bundle of $(M,g)$.

\paragraph{Negative signature in dimension $d \equiv_{8} 3$}

The topological obstruction to existence of a $\Spin^{o}_{-}$
structure on a 'negative Riemannian' manifold of dimension $d \equiv_{8}
3$ is given by:
\begin{equation}
\w_{1}(P) = \w_{1}(E)\qquad \w_{2}(P) + \w_{1}(P)^{2} + \w_{2}(E) = 0\, .
\end{equation}

\paragraph{Negative signature in dimension $d \equiv_{8} 7$}

The topological obstruction to existence of a $\Spin^{o}_{+}$
structure on a negative Riemannian manifold of dimension $d \equiv_{8}
7$ is given by:
\begin{equation}
\w_{1}(P) = \w_{1}(E)\qquad \w_{2}(P) = \w_{2}(E)\, .
\end{equation}

\noindent
In dimension $d\equiv_{8} 3$ with positive signature (respectively in
dimension $d\equiv_{8} 7$ with negative signature), the results above
show that a $\Spin^{o}_{-}$ structure (respectively a $\Spin^{o}_{+}$
structure) exists on $(M,g)$ iff there is there exists an
$\O(2)$-bundle $E$ on $(M,g)$ whose first and second Stiefel-Whitney
classes equal those of the (co)frame bundle. This implies, for example,
that any Riemannian three-manifold of the form:
\begin{equation}
M_{3} = M_{2}\times\mathbb{R}~~,
\end{equation}
where $M_{2}$ is a smooth surface, admits a $\Spin^{o}_{-}$
structure. Indeed, one can take $E$ to be the pull-back of the coframe
bundle of $M_{2}$ through the canonical projection and use stability
of Stiefel-Whitney classes. We will discuss other examples of
manifolds admitting $\Spin^{o}_{\alpha}$ structures in Section
\ref{sec:examples}.

\section{Elementary real pinor representations for $p-q\equiv_8 3,7$}
\label{sec:elementarypinor}

\noindent Let $(V,h)$ be a quadratic space of signature $(p,q)$ and
dimension $d=p+q$.  Throughout this section, we assume that
$p-q\equiv_8 3,7$, i.e. that we are in the so-called {\em almost
	complex case} according to the terminology used in
\cite{Lipschitz}. Notice that $d$ is necessarily odd in this case.

\begin{definition}
	Let:
	\be
	\Spin^o(V,h)\eqdef \Spin^o_{\alpha_{p,q}}(V,h)~~.
	\ee
	where: 
	\ben
	\label{alpha}
	\alpha:=\alpha_{p,q}\eqdef (-1)^{\frac{p-q+1}{4}}=\twopartdef{-1}{p-q\equiv_8 3}{+1}{p-q\equiv_8 7}~~.
	\een
\end{definition}

\begin{definition} An {\bf elementary real pinor representation of
		$\Cl(V,h)$} is an $\R$-irreducible representation
	$\gamma_0:\Cl(V,h)\rightarrow \End_\R(S_0)$ of the unital $\R$-algebra
	$\Cl(V,h)$ in a finite-dimensional $\R$-vector space $S_0$.
\end{definition}

\noindent 
It is well-known that all such representations are equivalent to each
other. Moreover, the finite-dimensional $\R$-vector space $S_0$ has
dimension given by:
\ben
\label{N}
N\eqdef \dim_\R S_0=2^{\frac{d+1}{2}}~~.
\een
For the following, we fix an elementary real pinor representation
$\gamma_0$ of $\Cl(V,h)$.

\subsection{Natural subspaces of $\End_\R(S_0)$}

\begin{definition}
	The {\bf natural subspaces} of $\End_\R(S_0)$ are defined as follows:
	\begin{enumerate}[1.]
		\item The {\bf Schur algebra} $\S\eqdef\S(\gamma_0)$ is the unital
		subalgebra of $\End_\R(S_{0})$ defined as the centralizer of
		$\gamma_{0}$ inside $\End_\R(S_0)$:
		\ben
		\label{Schur}
		\S=\{T\in \End_\R(S_{0}) \,\,|\,\, T\gamma_{0}(v)=\gamma_{0}(v)T~~\forall v\in V\}
		\een
		\item The {\bf anticommutant subspace} $\A\eqdef \A(\gamma_0)$ is the
		subspace of $\End_\R(S_0)$ defined through:
		\ben
		\label{Anticommutant}
		\A=\{T\in \End_\R(S_{0}) \,\,|\,\, T\gamma_{0}(v)=-\gamma_{0}(v)T~~\forall v\in V\}
		\een
		\item The {\bf twist algebra} $\T\eqdef \T(\gamma_0)$ is the unital
		subalgebra of $\End_\R(S_0)$ defined through:
		\ben
		\label{Twist}
		\T=\S + \A\, .
		\een
	\end{enumerate}
\end{definition}

\noindent We have $\S\cap \A=\{0\}$ and $\S \A=\A\S=\A$, so $\A$ is an
$\S$-bimodule. Thus \eqref{Twist} is a direct sum decomposition
$\T=\S\oplus \A$ which gives a $\Z_2$-grading of $\T$ with components:
\be
\T^+=\S~~,~~\T^-=\A~~.
\ee
Since we are in the almost complex case $p-q\equiv_8 3,7$, the Schur algebra
$\S$ is isomorphic with $\C$ (see \cite{Lipschitz}) and can be
described as follows. Recall that any orientation of $V$ determines a
Clifford volume element $\nu\in \Cl(V,h)$, which satisfies:
\be
\nu^2=-1~~.
\ee
The Clifford volume element determined by the opposite orientation of
$V$ equals $-\nu$ and the unordered set $\s_V=\{\nu,-\nu\}$ is
unambiguously determined by $(V,h)$; notice that $\s_V$ is a
semilinear structure on $V$. Setting $J\eqdef \gamma_{0}(\nu)$, we
have $J^2=-\id_{S_0}$ and $\gamma_{0}(-\nu)=-J$. Hence $S_0$ carries a
natural semilinear structure $\s:=\s(\gamma_0)=\{J,-J\}$ (in the sense
of Appendix \ref{app:semilinear}), which is independent of the
orientation of $V$. The unital subalgebra of $\End_\R(S_0)$ determined
by this semilinear structure as in Proposition \ref{prop:S} coincides
with the Schur algebra $\S$, which therefore is given by:
\be
\S=\R\oplus \R J=\R\oplus \R(-J)\, .
\ee

\begin{prop}
	\label{prop:D}
	There exists an $\S$-antilinear automorphism $D\in \Aut_\S^a(S_0)$ such that: 
	\begin{enumerate}
		\itemsep 0.0em
		\item $D\,\gamma_0(v)=-\gamma_0(v)\, D$ for all $v\in V$.
		\item $D^2 = \alpha_{p,q}\id_S$. 
	\end{enumerate}
	Moreover, any $\S$-antilinear operator $D'$ on $S_0$ which satisfies
	these two conditions has the form:
	\be
	D'=e^{\theta J}D\, ,
	\ee
	for some $\theta\in \R$.
\end{prop}

\begin{proof}
	Let us denote by $\Cl(p,q)$ the real Clifford algebra associated to a real pseudo-Riemannian vector space $\mathbb{R}^{p,q}$ of dimension $d = p + q$ and signature $(p,q)$. We first assume that $p-q\equiv_{8} 3$ and consider the Clifford algebra $\Cl(p,q+1)$ associated to $\mathbb{R}^{p,q+1}$. From the
	representation theory of Clifford algebras we know that there is an
	elementary real Clifford representation:
	\begin{equation}
	\gamma^{\prime}_{0}\colon \Cl(p,q+1)\to \End_{\mathbb{R}}(S_{0})\, ,
	\end{equation}
	and that $\Cl(p,q+1)$ is generated by the canonical orthonormal basis
	of $\mathbb{R}^{p,q}$ together with the linearly independent unit
	element $x\in \mathbb{R}^{p,q+1}$ that completes it to the canonical
	basis of $\mathbb{R}^{p,q+1}$. We set then $D \eqdef
	\gamma^{\prime}_{0}(x)\subset \End_{\mathbb{R}}(S_{0})$, which is
	indeed $\mathbb{S}$-antilinear and satisfies $D^{2} = -1$ as well as
	$D\,\gamma_0(v)=-\gamma_0(v)\, D$ for all $v\in V$. The proof for the
	case $p-q\equiv_{8} 7$ is similar but considering $\Cl(p+1,q)$.
	
	Given two $\mathbb{S}$-antilinear endomorphisms $D$ and $D^{\prime}$
	satisfying conditions 1. and 2., we have $D D^{\prime}\in \mathbb{S}$
	and $D^{\prime} D \in \mathbb{S}$. Moreover, we have $(D
	D^{\prime})^{2} = 1$, which implies $D^{\prime} = e^{\theta J} D$ with
	$\theta\in\mathbb{R}$. 
\end{proof}

\noindent 
Let $\fD:=\fD(\gamma_0) \subset \Aut_\R(S_{0})$ denote the
$\U(1)$-torsor consisting of all elements $D\in \Aut_\S^a(S_0)$ which
satisfy the conditions of Proposition \ref{prop:D}. Then the following
result (whose proof we leave to the reader) holds:

\begin{prop}
	$\A$ is a free $\S$-bimodule of rank one and any $D\in \fD$ is a basis
	of this bimodule, hence satisfies $\A=\S D=D\S$.
\end{prop}

\noindent 
The proposition implies that the elements $D$ and $JD$ (where $D\in
\fD$ is arbitrary) form a basis of the $\R$-vector space $\A$:
\be
\A=\R D\oplus \R J D~~.
\ee
Let $e_1,e_2$ be the canonical basis of $\R^2$ and consider the
elements $D_0\eqdef e_1$, $J_0\eqdef e_1e_2$ of $\Cl_2(\alpha_{p,q})$. We
have $e_1^2=e_2^2=\alpha_{p,q}$ and hence:
\be
J_0^2=-1~~,~~D_0^2=\alpha_{p,q}~~,~~J_0D_0=-\alpha_{p,q} e_2~~.
\ee

\begin{prop}
	For any $(\nu,D)\in \s_V\times \fD$, there exists a unique unital
	isomorphism of algebras
	$\gamma^{(2)}_{\nu,D}:\Cl_2(\alpha_{p,q})\stackrel{\sim}\rightarrow
	\T\subset \End_\R(S_0)$ which satisfies the conditions:
	\be
	\gamma^{(2)}_{\nu,D}(e_1)=\gamma^{(2)}_{\nu,D}(D_0)=D~~,~~\gamma^{(2)}_{\nu,D}(e_2)=-\alpha_{p,q} J D~~,
	\ee
	where $J\eqdef \gamma(\nu)$. Moreover, $\gamma^{(2)}_{\nu,D}$ is an
	isomorphism of $\Z_2$-graded algebras and we have
	$\gamma^{(2)}_{\nu,D}(J_0)=J$.
\end{prop}

\begin{proof} 
	We have $J^2=-\id_{S_{0}}$, $D^2=\alpha_{p,q}\id_{S_0}$ and
	$(-\alpha_{p,q} JD)^2=JDJD=-J^2 D^2=D^2=\alpha_{p,q} \id_{S_0}$ in
	$\T$. This implies existence and uniqueness of
	$\gamma^{(2)}_{\nu,D}$. Moreover,
	$\gamma^{(2)}_{\nu,D}(J_0)=\gamma^{(2)}_{\nu,D}(e_1e_2)=\gamma^{(2)}_{\nu,D}(e_1)\gamma^{(2)}_{\nu,D}(e_2)=-\alpha_{p,q}
	DJD=\alpha_{p,q} JD^2=J$.  
\end{proof}

\noindent 
When no confusion is possible, in the following we will write $\alpha$
instead of $\alpha_{p,q}$. Since $\gamma^{(2)}_{\nu,D}$ is an
isomorphism of $\Z_2$-graded algebras, we have:
\ben
\label{gammaAS}
\gamma^{(2)}_{\nu,D}(\R^2)=\A~~,~~\gamma^{(2)}_{\nu,D}(\Cl_2^+(\alpha))=\S~~.
\een
Notice the commutative diagram:
\ben
\label{diaggamma}
\scalebox{1}{
	\xymatrix{
		\Cl_2(\alpha)^\times \ar[d]_{\Ad}~ \ar[r]^{\gamma^{(2)}_{\nu,D}}_{\sim} & ~~\T^\times ~ \ar[d]^{\Ad} \\
		\Aut_\Alg(\Cl_2(\alpha)) \ar[r]^{~~~\sim}_{~~~\Ad(\gamma^{(2)}_{\nu,D})}~ & ~\Aut_\Alg(\T)
}}
\een
Also notice that the untwisted adjoint representation of $\Pin_2(\alpha)$
preserves the subalgebra $\Cl_2^+(\alpha)$. Let
$\Ad^{(2)}_+:\Pin_2(\alpha)\rightarrow \Aut_\Alg(\Cl_2^+(\alpha))$ be
the group morphism given by:
\be
\Ad^{(2)}_+(a)\eqdef \Ad(a)|_{\Cl_2^+(\alpha)}~~\forall a\in \Pin_2(\alpha)~~.
\ee
Recall that $\Pin_2(\alpha)=\Spin_2(\alpha)\sqcup \Spin_2(\alpha)D$
and that $\Spin_2(\alpha)=\{z(\theta)=e^{\theta J_0}|\theta\in
\R\}\simeq \U(1)$ acts trivially in the adjoint representation on
$\Cl_2^+(\alpha)=\R\oplus \R J_0\simeq \C$, while $D_0$ acts in the
adjoint representation by taking $J_0$ into $-J_0$.  Hence:
\ben
\label{AdPlus}
\Ad^{(2)}_+(a)=\id_{\Cl_2^+(\alpha)}~~,~~\Ad^{(2)}_+(aD_0)=c_+~~,~~\forall a\in \Spin_2(\alpha)~~,
\een
where $c_+\in \Aut_\Alg(\Cl_2^+(\alpha))$ denotes the conjugation
automorphism:
\be
c_+(x+y J_0)=x-y J_0~~\forall x, y\in \R~~.
\ee
Since $\Cl_2^+(\alpha)\simeq_\Alg \C$, we have
$\Aut_\Alg(\Cl_2^+(\alpha))=\{\id_{\Cl_2^+(\alpha)},c_+\}\simeq \Z_2$
and \eqref{AdPlus} shows that the group morphism $\Ad^{(2)}_+$ can be
identified with the $\Z_2$-grading morphism of $\Pin_2(\alpha)$.

\begin{prop}
	\label{prop:basicreps}
	We have $(\Ad\circ \gamma_{\nu,D}^{(2)})(\Pin_2(\alpha))(\S)= \S$,
	$(\Ad\circ \gamma_{\nu,D}^{(2)})(\Pin_2(\alpha))(\A)= \A$ and $(\Ad\circ
	\gamma_{\nu,D}^{(2)})(\Pin_2(\alpha))(\T)=\T$. Moreover:
	\begin{enumerate}[1.]
		\itemsep 0.0em
		\item The representation $\Ad^{(2)}_\T: \Pin_2(\alpha)\rightarrow
		\Aut_\Alg(\T)$ given by $\Ad^{(2)}_\T(b)\eqdef
		\Ad\circ\gamma_{\nu,D}^{(2)}(b)|_\T$ is equivalent with the full
		untwisted adjoint representation of $\Pin_2(\alpha)$ on
		$\Cl_2(\alpha)$.
		\item The representation $\Ad^{(2)}_\A: \Pin_2(\alpha)\rightarrow
		\Aut_\R(\A)$ given by $\Ad^{(2)}_\A(b)\eqdef
		\Ad\circ\gamma_{\nu,D}^{(2)}(b)|_\A\in \Aut_\R(\S)$ is equivalent
		with the untwisted vector representation $\Ad^{(2)}$ of
		$\Pin_2(\alpha)$.
		\item The representation $\Ad^{(2)}_\S:\Pin_2(\alpha)\rightarrow
		\Aut_\Alg(\S)\simeq \Z_2$ given by $\Ad^{(2)}_\S(b)\eqdef
		\Ad\circ\gamma_{\nu,D}^{(2)}(b)|_\S$ is equivalent with the
		representation: 
		\begin{equation*}
		\Ad^{(2)}_+:\Pin_2(\alpha)\rightarrow \Aut_\Alg(\Cl_2^+(\alpha))\, ,
		\end{equation*}
		being given by:
		\be
		\Ad^{(2)}_\S(a)=\id_\S~~,~~\Ad^{(2)}_\S(a D)=c~~\forall a\in \Spin_2(\alpha)~~,
		\ee
		where $c\in \Aut_\Alg(\S)$ is the conjugation automorphism of $\S$. 
	\end{enumerate}
\end{prop}

\begin{proof}
	The commutative diagram \eqref{diaggamma} gives $\Ad\circ
	\gamma^{(2)}_{\nu,D}=\Ad(\gamma^{(2)}_{\nu,D})\circ \Ad$. This implies:
	\begin{equation*}
	(\Ad\circ \gamma_{\nu,D}^{(2)})(\Cl_2(\alpha)^\times)(\T)=\T\, ,
	\end{equation*}
	which by restriction gives $(\Ad\circ
	\gamma_{\nu,D}^{(2)})(\Pin_2(\alpha))(\T)=\T$. On the other hand, we
	have:
	\beqa
	&& (\Ad\circ \gamma^{(2)}_{\nu,D})(\Pin_2(\alpha))(\A)=\Ad(\gamma^{(2)}_{\nu,D})(\Ad(\Pin_2(\alpha))(\gamma^{(2)}_{\nu,D}(\R^2))=\\
	&&=\gamma^{(2)}_{\nu,D}(\Ad(\Pin_2(\alpha))(\R^2))=\gamma^{(2)}_{\nu,D}(\R^2)=\A~~.
	\eeqa
	where we used \eqref{gammaAS} and the relation: 
	\ben
	\label{AdRel}
	\Ad(\gamma^{(2)}_{\nu,D})(\Phi)(\gamma^{(2)}_{\nu,D}(x))=\gamma^{(2)}_{\nu,D}(\Phi(x))~~,~~\forall \Phi\in \Aut_\Alg(\Cl_2(\alpha))~~,~~\forall x\in \Cl_2(\alpha)~~,
	\een
	together with the fact that the untwisted adjoint representation of
	$\Pin_2(\alpha)$ preserves the subspace $\R^2\subset
	\Cl_2(\alpha)$. Similarly, we have:
	\beqa
	&& (\Ad\circ \gamma^{(2)}_{\nu,D})(\Pin_2(\alpha))(\S)=\Ad(\gamma^{(2)}_{\nu,D})(\Ad(\Pin_2(\alpha))(\gamma^{(2)}_{\nu,D}(\Cl^+_2(\alpha)))=\\
	&&=\gamma^{(2)}_{\nu,D}(\Ad(\Pin_2(\alpha))(\Cl_2^+(\alpha)))=\gamma^{(2)}_{\nu,D}(\Cl_2^+(\alpha))=\S~~,
	\eeqa
	where we used \eqref{gammaAS} and \eqref{AdRel} and that fact that
	untwisted adjoint representation of $\Pin_2(\alpha)$ preserves the
	subalgebra $\Cl_2^+(\alpha)\subset \Cl_2(\alpha)$.  By restriction,
	diagram \eqref{diaggamma} induces commutative diagrams:
	\be
	\label{diagAS}
	\scalebox{0.8}{
		\xymatrix{
			\Pin_2(\alpha)\ar[d]_{\Ad} \ar[rd]^{\Ad_\T^{(2)}} \\
			\Aut_\Alg(\Cl(V,h)) \ar[r]^{~~~\sim}_{~~~\Ad(\gamma^{(2)}_{\nu,D})}~ & ~\Aut_\Alg(\T)
	}}~~
	\scalebox{0.8}{
		\xymatrix{
			\Pin_2(\alpha)\ar[d]_{\Ad^{(2)}} \ar[rd]^{\Ad_\A^{(2)}} \\
			\Aut_\R(\R^2) \ar[r]^{~~~\sim}_{~~~\Ad(\gamma^{(2)}_{\nu,D})}~ & ~\Aut_\R(\A)
	}}~~\scalebox{0.8}{
		\xymatrix{
			\Pin_2(\alpha)\ar[d]_{\Ad_+^{(2)}} \ar[rd]^{\Ad_\S^{(2)}} \\
			\Aut_\Alg(\Cl_2^+(\alpha)) \ar[r]^{~~~\sim}_{~~~\Ad(\gamma^{(2)}_{\nu,D})}~ & ~\Aut_\Alg(\S)
	}}
	\ee
	which give the equivalences of representations listed in the
	proposition.   \end{proof}

\noindent Let ${\widehat \Pin}(V,h)\eqdef {\widehat
	\Pin}_{\alpha_{p,q}}(V,h)$.  Notice that any $D\in \fD$ gives an
isomorphism of torsors $\tau_D:\Spin^o(V,h)/{\widehat
	\Pin}(V,h)\stackrel{\sim}{\rightarrow} \fD$ such that
$\tau([\hat{D}])=D$, where $[\hat{D}]$ is the class of $\hat{D}$ in
$\Spin^o(V,h)/{\widehat \Pin}(V,h)$.

\begin{proposition}
	\label{prop:equiv}
	Let $\nu,\nu'\in \s_V$ and $D,D'\in \fD$. Then the representations
	$\gamma^{(2)}_{\nu,D}$ and $\gamma^{(2)}_{\nu',D'}$ of $\Cl_2(\alpha)$
	are equivalent.
\end{proposition}

\begin{proof} We have $D'=e^{\theta J} D$ for some $\theta\in \R$. Thus
	$\Ad(e^{\frac{\theta}{2}J})(D)=\Ad(e^{\frac{\theta}{2}JD})(D)=D'$ and
	$\Ad(e^{\frac{\theta}{2}J})(J)=-\Ad(e^{\frac{\theta}{2}JD})(J)=J$.
	Since $\Cl_2(\alpha)$ is generated by $J_0$ and $D_0$, this implies:
	\be
	\Ad(e^{\frac{\theta}{2}J})\circ \gamma_{\nu,D}=\gamma_{\nu,D'}~~,~~\Ad(e^{\frac{\theta}{2}JD})\circ \gamma_{\nu,D}=\gamma_{-\nu,D'}~~.
	\ee
	Since $\nu'$ equals either $\nu$ or $-\nu$, this gives the conclusion.
	 \end{proof}

\begin{definition}
	A {\bf $\gamma$-compatible representation of $\Cl_2(\alpha_{p,q})$} is an
	$\R$-linear representation
	$\gamma^{(2)}:\Cl_2(\alpha_{p,q})\rightarrow \End_\R(S_0)$ which equals
	$\gamma^{(2)}_{\nu,D}$ for some pair $(\nu,D)\in \s_V\times \fD$.
\end{definition}

\noindent By the proposition above, all such representations of
$\Cl_2(\alpha_{p,q})$ are equivalent. In the following, we fix a
$\gamma$-compatible representation $\gamma^{(2)}$ of
$\Cl_2(\alpha_{p,q})$ parameterized by the pair $(\nu,D)\in \s_V\times
\fD$.

\subsection{Twisted elementary real representations} 

\begin{definition}
	The {\bf twisted elementary real representation} of
	$\Spin^o_{\alpha_{p,q}}(V,h)$ induced by $\gamma^{(2)}$ is the group
	morphism $\gamma_o:\Spin^o_{\alpha_{p,q}}(V,h)\rightarrow
	\Aut_\R(S_0)$ given by:
	\be
	\gamma_o([a,b])=\gamma_0(a)\gamma^{(2)}(b)~~\forall a\in \Spin(V,h)~~,~~b\in \Pin_2(\alpha_{p,q})~~.
	\ee
\end{definition}

\noindent By Proposition \ref{prop:equiv}, twisted elementary
representations of $\Spin^o_{\alpha_{p,q}}(V,h)$ are uniquely
determined up to equivalence. The representation
$\Ad^{(2)}_+:\Pin_2(\alpha)\rightarrow \Aut_\Alg(\Cl_2^+(\alpha))$
induces a representation
$\Ad_+:\Spin^o_{\alpha_{p,q}}(\alpha)\rightarrow
\Aut_\Alg(\Cl_2^+(\alpha))$ given by:
\be
\Ad_+([a,b])\eqdef \Ad^{(2)}_+(b)~~\forall a\in \Spin(V,h)~~,~~\forall b\in \Pin_2(\alpha)~~.
\ee
Together with the characteristic representation
\begin{equation*}
\mu:\Spin^o_{\alpha_{p,q}}(V,h)\rightarrow \O(\R^2)=\Aut_\R(\Cl_2^-(\alpha))
\end{equation*}
this gives a representation
\begin{equation*}
\Ad_+\oplus \mu:\Spin^o_{\alpha_{p,q}}(V,h)\rightarrow
\Aut_\Alg(\Cl_2(\alpha))\, .
\end{equation*}

\begin{prop}
	\label{prop:basicreps}
	We have $(\Ad\circ \gamma_o)(\Spin^o_{\alpha_{p,q}}(V,h))(\S)= \S$,
	$(\Ad\circ \gamma_o)(\Spin^o_{\alpha_{p,q}}(V,h))(\A)= \A$ and
	$(\Ad\circ \gamma_o)(\Spin^o_{\alpha_{p,q}}(V,h))(\T)= \T$. Moreover:
	\begin{enumerate}[1.]
		\itemsep 0.0em
		\item The representation $\Ad_\A:
		\Spin^o_{\alpha_{p,q}}(V,h)\rightarrow \Aut_\R(\A)$ given by
		$\Ad_\A(g)\eqdef (\Ad\circ\gamma_o)(g)|_\A$, i.e.:
		\be
		\Ad_\A([a,b])=\Ad_\A^{(2)}(b)~~\forall a\in \Spin(V,h)~~,~~\forall b\in \Pin_2(\alpha_{p,q})
		\ee
		is equivalent with the untwisted characteristic representation
		$\mu:\Spin^o_{\alpha_{p,q}}(V,h)\rightarrow \O(\R^2)=\O(2)$.
		\item The representation
		$\Ad_\S:\Spin^o_{\alpha_{p,q}}(V,h)\rightarrow \Aut_\Alg(\S)\simeq
		\Z_2$ given by $\Ad_\S(a)\eqdef \Ad\circ\gamma_o(a)|_\S$, i.e.:
		\be
		\Ad_\S([a,b])=\Ad_\S^{(2)}(b)~~\forall a\in \Spin(V,h)~~,~~\forall b\in \Pin_2(\alpha_{p,q})~~
		\ee
		is equivalent with the representation
		$\Ad_+:\Spin^o_{\alpha_{p,q}}(V,h)\rightarrow
		\Aut_\Alg(\Cl_2^+(\alpha))$.
		\item The representation $\Ad_\T:
		\Spin^o_{\alpha_{p,q}}(V,h)\rightarrow \Aut_\Alg(\T)$ defined
		through $\Ad_\T(g)\eqdef (\Ad\circ\gamma_o)(g)|_\T$, i.e.:
		\be
		\Ad_\T([a,b])=\Ad_\T^{(2)}(b)~~\forall a\in \Spin(V,h)~~,~~\forall b\in \Pin_2(\alpha_{p,q})\, ,
		\ee
		is equivalent with the representation $\Ad_+\oplus
		\mu:\Spin^o_{\alpha_{p,q}}(V,h)\rightarrow \Aut_\Alg(\Cl_2(\alpha_{p,q}))$.
	\end{enumerate}
\end{prop}

\begin{proof} Follows from Proposition \ref{prop:basicreps} and
	Definition \ref{def:elementaryreps}.   
\end{proof}

Let $i\colon \Spin^o_{\alpha_{p,q}}(V,h)\rightarrow \Aut_\R(V \otimes
S_0)\simeq \Aut_\R(V)\otimes \Aut_\R(S_0)$ denote the inner tensor
product of the twisted vector representation
$\tlambda_{\alpha}:\Spin^o_{\alpha}(V,h)\rightarrow \Aut_\R(V)$ with the
twisted elementary representation
$\gamma_o:\Spin^o_{\alpha}(V,h)\rightarrow \Aut_\R(S_0)$:
\be
i (g)\eqdef \tlambda_{\alpha}(g)\otimes \gamma_o(g)~~\forall g\in \Spin^o_{\alpha}(V,h)\, .
\ee

\begin{lemma}
	\label{lemma:commutingo}
	For all $g\in \Spin^o_{\alpha_{p,q}}(V,h)$ and all $v\in V$, we have:
	\ben
	\label{CMcond}
	\gamma_0(\tlambda(g)v)\circ \gamma_o(g)=\gamma_o(g)\circ \gamma_0(v)\, .
	\een
\end{lemma}

\begin{proof}
	For any $g=[a,b]\in \Spin^o_{\alpha_{p,q}}(V,h)$ with $a\in \Spin(V,h)$ and $b\in \Pin_2(\alpha)$, we have:
	\beqan
	\label{2exp}
	&&\gamma_0(\tlambda(g)v)\circ \gamma_o(g)=(\det \Ad_0^{(2)}(b)) \gamma_0(\Ad_0(a)v)\circ \gamma_0(a)\circ \gamma^{(2)}(b)\, ,\\
	&& \gamma_o(g)\circ \gamma_0(v)=\gamma_0(a)\circ \gamma^{(2)}(b)\circ \gamma_0(v)=(\det \Ad_0^{(2)}(b)) \gamma_0(\Ad_0(a)v)\circ \gamma_0(a)\circ \gamma^{(2)}(b)~~,\nn
	\eeqan
	where in the second row we used the relation: 
	\be
	\gamma^{(2)}(b)\circ \gamma_0(v)=(\det \Ad_0^{(2)}(b))\gamma_0(v) \circ \gamma^{(2)}(b)~~,
	\ee
	which follows from the fact that $D$ and $\gamma_0(v)$ anticommute and
	the relation:
	\be
	\gamma_0(a)\circ \gamma_0(v)=\gamma_0(\Ad_0(a)v)\circ \gamma_0(a)\, ,
	\ee
	which in turn follows from $\Ad_0(a)(v)=ava^{-1}$ and from the fact that
	$\gamma_0$ is a morphism of algebras. Comparing the two rows in
	\eqref{2exp} gives the conclusion.   
\end{proof}

\begin{lemma}
	\label{lemma:cliffordmeq}
	The linear map $\frc\colon V\otimes S_0\rightarrow S_0$ given by
	$\frc(v,x)\eqdef\gamma_0(v)x$ is a morphism of representations between
	$i$ and $\gamma_{o}$, i.e. it satisfies:
	\begin{equation*}
	\frc\circ i(g)=\gamma_o (g)\circ\frc\, ,
	\end{equation*}
	for all $g\in \Spin^o(V,h)$. 
\end{lemma}

\begin{proof}
	For all $v\in V$, $s\in S_0$ and $g\in \Spin^o(V,h)$, we have: 
	\be
	(\frc\circ i(g))(v\otimes s)=(\gamma_0(\tlambda(g)v)\circ \gamma_o(g))(s)=(\gamma_o(g)\circ \gamma_0(v))(s)=(\gamma_o(g)\circ \frc)(v\otimes s)\, ,
	\ee
	where we used \eqref{CMcond}. Hence $\frc\circ i(g)=\gamma_o
	(g)\circ\frc$ for all $g\in \Spin^o(V,h)$.
	 \end{proof}

\begin{cor}
	The linear map $\frc\colon V\otimes S_{0} \to S_{0}$ defines a
	$\gamma_{o}$-equivariant Clifford multiplication map on $S_{0}$, which
	canonically becomes an elementary module over $\Cl(V,h)$ when
	equipped with the unital morphism of algebras $\gamma_{0}\colon
	\Cl(V,h)\to \End_{\mathbb{R}}(S_{0})$ obtained from $\frc$ by
	extension with respect to the algebra structure of $\Cl(V,h)$.
\end{cor}

\begin{proof}
	Let $(\gamma_{o},S_{0})$ be a twisted elementary representation of
	$\Spin^{o}_{\alpha_{p,q}}(V,h)$. The linear map $\frc\colon V\otimes
	S_{0}\to S_{0}$ defines by evaluation $v\mapsto\frc(v,-)$ on its first
	argument a Clifford multiplication map $\gamma_{V}\colon
	V\to \End_{\mathbb{R}}(S_{0})$. The fact that $\gamma_{V}$ is
	$\gamma_{o}$-equivariant follows from Lemma
	\ref{lemma:cliffordmeq}. The Clifford map $\gamma_{V}$ is in
	particular an injective morphism from $V\subset \Cl(V,h)$ to
	$\End_{\mathbb{R}}(S_{0})$. Since $V$ generates $\Cl(V,h)$, we extend 
	this map to all of $\Cl(V,h)$, thus obtaining the elementary
	Clifford representation $\gamma_{0}\colon
	\Cl(V,h)\to \End_{\mathbb{R}}(S_{0})$.
	 
\end{proof}

\section{Elementary real pinor bundles for $p-q\equiv_8 3,7$}
\label{sec:elementary}

\noindent Throughout this section, we assume that $(M,g)$ is a smooth,
connected and paracompact pseudo-Riemannian manifold of signature
$(p,q)$ with $p-q\equiv_8 3,7$ and dimension $d=p+q$. We define:
\begin{equation*}
\w_1^\pm(M)\eqdef \w_1^\pm(P_{\O(V,h)}(M,g))\, ,
\end{equation*}
where $P_{\O(V,h)}(M,g)$ denotes the orthonormal coframe bundle of
$(M,g)$. Let $\Cl(M,g)$ denote the Clifford bundle of the
pseudo-Euclidean {\em cotangent}\footnote{We use the Clifford bundle
	of the {\em cotangent} (rather than of the tangent) bundle of $M$ in order
	to avoid using musical isomorphisms in certain formulas later on.}
bundle $(T^\ast M, g^\ast)$, where $g^\ast$ is the metric induced by
$g$ on $T^\ast M$.

\begin{definition}
	An {\bf elementary real pinor bundle} is a bundle of
	finite-dimensional simple modules over the Clifford bundle $\Cl(M,g)$,
	i.e. a pair $(S,\gamma)$ where $S$ is a real vector bundle over
	$(M,g)$ and $\gamma:\Cl(M,g)\rightarrow End_{\mathbb{R}}(S)$ is a
	morphism of bundles of unital associative $\R$-algebras that restricts
	to an elementary real pinor representation $\gamma_{m}\colon
	\Cl(T^{\ast}_{m}M,g_{m})\rightarrow \End_{\mathbb{R}}(S_{m})$ at every
	point $m\in M$.
\end{definition}

Given an elementary real pinor bundle $(S,\gamma)$, we define its {\bf
	type } $[\eta]$ as the isomorphism class of $\gamma_{m}\colon
\Cl(T^{\ast}_{m}M,g^{\ast}_{m})\rightarrow \End_{\mathbb{R}}(S_{m})$
in the category $\ClRep$ introduced in \cite{Lipschitz}, where $m\in
M$ denotes any point of $M$. It can be easily seen that the type does
not depend on the point chosen since $M$ is connected. Here $\ClRep$
denotes the category of real Clifford representations 
and unbased morphisms, see op. cit. for details. For ease of notation, in the following we will sometimes denote $\Spin^o_{\alpha_{p,q}}(V,h)$ simply by $\Spin^o$.

\begin{definition}
	An {\bf adapted $\Spin^o$ structure} $(Q,\tilde{\Lambda})$ on $(M,g)$
	is a $\Spin^o_{\alpha_{p,q}}$ structure
	$(Q,\tilde{\Lambda}_{\alpha_{p,q}})$ on $(M,g)$, where $\alpha_{p,q}$
	was defined in \eqref{alpha}.
\end{definition}

\noindent
Assume that $(M,g)$ admits an adapted $\Spin^o$ structure and let $(V,h)$ be a 
quadratic vector space which is isometric with $(T_m^\ast, g_m^\ast)$, where 
$m$ is any point of $M$. 

\begin{definition} A {\bf twisted $\Spin^{o}$ vector bundle} over
	$(M,g)$ is a real vector bundle $S_{o} = Q\times_{\gamma_o}S_0$
	associated to the principal $\Spin^o(V,h)$-bundle $Q$ of an adapted
	$\Spin^o$ structure $(Q,\tilde{\Lambda})$ on $M$ through a twisted
	elementary representation $\gamma_o:\Spin^o(V,h)\rightarrow
	\Aut_\R(S_0)$ of $\Spin^{o}_{\alpha_{p,q}}(V,h)$.
\end{definition}

When necessary we will denote by $S_{o}(Q,\gamma_{o})$ the twisted
$\Spin^{o}$ vector bundle associated to $Q$ through
$\gamma_{o}$. Notice that a twisted $\Spin^{o}$ vector bundle is a
real vector bundle of rank $\rk S=N=\dim S_0=2^{\frac{d+1}{2}}$ (see
equation \eqref{N}). We will see in a moment that such a bundle admits
a well-defined Clifford multiplication, even though $(M,g)$ need not
be orientable. In particular, the existence of a $\Spin^{o}$ structure
allows for a useful notion of ``irreducible real pinors'' on {\em
	non-orientable}\footnote{This is in marked contrast with the kind of
	pinors considered in \cite{Nakamura1,Nakamura2}, which require $(M,g)$
	to admit a $\Spin^c_\alpha$ structure (see Appendix \ref{Spincalpha})
	and hence require that $M$ be orientable.}  pseudo-Riemannian
manifolds of signature $p-q\equiv_8 3,7$.

\begin{proposition}
	\label{prop:Clifformult}
	Let $S_{o}$ be the twisted $\Spin^{o}$ vector bundle
	associated to an adapted $\Spin^{o}$ structure
	$(Q,\tilde{\Lambda})$ through the  twisted
	elementary representation $\gamma_o:\Spin^o(V,h)\rightarrow
	\Aut_\R(S_0)$. Then the based bundle map $\frC\colon
	T^{\ast}M\otimes S_{o} \to S_{o}$ given by:
	\begin{eqnarray}
	\label{eq:frC}
	\frC\colon T^{\ast}M\otimes S_{o} &\to & S_{o}\, ,\nn\\
	\left[p,\xi \right]\otimes [p,s] &\mapsto &  [p, \gamma_{0}(\xi) s]\, , \qquad \forall p\in M\, ,
	\end{eqnarray}
	defines a Clifford multiplication on $S_o$. This makes $S_o$ into an
	elementary real pinor bundle when the latter is equipped with the unital
	morphism of bundles of algebras $\gamma\colon
	\Cl(M,g)\to \End_{\mathbb{R}}(S)$ obtained from $\frC$ by extension.
\end{proposition}

\begin{proof}
	Recall that an elementary real pinor bundle is a bundle of simple
	modules over $\Cl(M,g)$. In signature $p-q\equiv_{8}3,7$ such a bundle
	is of complex type. For convenience, we consider the cotangent bundle
	$T^{\ast}M$ as an associated bundle to $Q$ through the twisted vector
	representation $\tilde{\lambda}\colon\Spin^{o}(V,h)\to \O(V,h)$, that
	is, $T^{\ast}M = Q\times_{\tilde{\lambda}} V$. The fact that $\frC$ is
	well defined follows from the following computation:
	\begin{eqnarray*}
		\frC([pg,\tilde{\lambda}(g^{-1}) \xi],[pg,\gamma_{o}(g^{-1})s]) = [pg,\gamma_{0}(\tilde{\lambda}(g^{-1})\xi)\circ\gamma_{o}(g^{-1})s] \\
		= [pg,\gamma_{o}(g^{-1}) \circ\gamma_{0}(\xi) s] = \frC([p,\xi],[p,s])\, ,
	\end{eqnarray*}
	where we have used Lemma \eqref{lemma:commutingo}. Using $\frC$ we pointwise
	define an injective morphism $\gamma_{T^{\ast}M}\colon
	T^{\ast}M\to End_{\mathbb{R}}(S_{o})$ as follows:
	\begin{eqnarray}
	\gamma_{T^{\ast}M}\colon T^{\ast}M &\to & End_{\mathbb{R}}(S_{o})\, ,\nn\\
	\left[p,\xi\right]  &\mapsto &  \frC([p,\xi],-)\, , 
	\end{eqnarray}
	Note that $\gamma_{T^{\ast}M}(\xi)^{2} = g^{\ast}(\xi,\xi)
	\id_{S_{o}}$, and hence $\gamma_{T^{\ast}M}$ extends to a
	unique unital morphism of bundles of algebras:
	\begin{equation*}
	\gamma:\Cl(M,g)\rightarrow End_{\mathbb{R}}(S_{o})\, ,
	\end{equation*}
	whose fiber at any $p\in M$ is equivalent with the unique, up to
	unbased isomorphism, elementary Clifford representation
	$\gamma_0:\Cl(V,h)\rightarrow \End(S_0)$ through an unbased
	isomorphism of Clifford representations.   
\end{proof}

\noindent
The following result unifies Theorem \ref{thm:SpinoObs} with the
previous discussion.

\begin{theorem}
	\label{thm:equiv}
	Let $(M,g)$ be a pseudo-Riemannian manifold of signature $(p,q)$ with
	$p-q\equiv_8 3,7$. Then $(M,g)$ admits an elementary real pinor bundle
	$(S,\gamma)$ if and only if there exists a principal $\O(2)$-bundle
	$E$ over $M$ such that the following conditions are satisfied:
	\begin{eqnarray}
	\label{obsoadapted}
	&& \w_1^+(M)+\w_1^-(M)=\w_1(E)\\
	&& \w_2^+(M)+\w_2^-(M)+\w_1(E)(p\w_1^+(M)+q\w_1^-(M))=\w_2(E) \\ && + \left[\delta(p,q)+\frac{p(p+1)}{2}+\frac{q(q+1)}{2}\right]\w_1(E)^2\, .\nn
	\end{eqnarray}
	where: 
	\begin{equation*}
	\delta(p,q)=\twopartdef{1}{p-q\equiv_8 3}{0}{p-q\equiv_8 7}~~.
	\end{equation*}
	Let $[\eta]$ be the type of $(S,\gamma)$. In that case, and relative
	to $[\eta]$, there exists an adapted $\Spin^{o}_{\alpha}(V,h)$ structure
	$Q(S,\gamma)$ on $(M,g)$, unique up to isomorphism, such that
	$(S,\gamma)$ is isomorphic to $(S_{o}(Q(S,\gamma)),\gamma_{o})$ as a bundle of
	irreducible Clifford modules, and Clifford multiplication in $S$ is
	implemented by the morphism of vector bundles $\mathfrak{C}\colon
	T^{\ast}M\otimes S\to S$ defined in equation
	\eqref{eq:frC}.
\end{theorem}

\begin{remark}
	Notice from conditions \eqref{obsoadapted} that existence of a
	bundle of elementary real pinors on $(M,g)$ does not necessarily imply
	that $M$ is orientable. We will present examples of such situation in
	section \eqref{sec:examples}.
\end{remark}

\begin{proof}	
	Theorem 6.2 of \cite{Lipschitz} implies that $(M,g)$ admits a bundle
	of irreducible real Clifford modules $(S,\gamma)$ if and only if
	admits a reduced Lipschitz structure, which in signature
	$p-q\equiv_{8} 3, 7$ corresponds to a $\Spin^o$ structure
	$Q(S,\gamma)$ as defined in Section
	\ref{sec:spinostructures}. Furthermore, $(S,\gamma)$ is associated to
	$Q$ through the tautological representation $\gamma_{0}\colon
	\Spin^{o}(V,h)\to \Aut_{\mathbb{R}}(S_{0})$. From this we conclude
	that $(M,g)$ admits a bundle of irreducible real Clifford modules if
	and only if conditions \eqref{obso} of Theorem \ref{thm:SpinoObs} are
	satisfies for $\alpha=\alpha_{p,q}$. Since $d$ is odd, we have
	$d\equiv_2 1$. Moreover, we have
	$\delta_{\alpha_{p,q},-1}=\delta(p,q)$, so conditions \eqref{obso}
	reduce to \eqref{obsoadapted}.
	
	Let $(S,\gamma)$ be an elementary pinor bundle of type $[\eta]$ and
	let $Q(S,\gamma)$ the unique (up to isomorphism) associated adapted
	$\Spin^{o}$ structure relative to $[\eta]$. Let $(S_{o},\gamma_{o})$
	denote the corresponding twisted $\Spin^{o}$ vector bundle. The fact
	that $(S,\gamma)$ is isomorphic to $(S_{o},\gamma_{o})$ follows from
	the equivalence between:
	\begin{equation*}
	\gamma_{0}\colon \Spin^{o}(V,h)\to \Aut_{\mathbb{R}}(S_{0})\, ,
	\end{equation*}
	
	\noindent
	and $\gamma_{o}\colon \Spin^{o}(V,h)\to \Aut_{\mathbb{R}}(S_{0})$
	together with the equivalence of the associated Clifford
	multiplications upon use of Corollary \ref{prop:Clifformult}.
	 
\end{proof}

\begin{remark}
	Theorem \ref{thm:equiv} implies that twisted $\Spin^{o}$ vector
	bundles $(S_{o},\gamma_{o})$ over $(M,g)$ are essentially equivalent
	to bundles of irreducible real Clifford modules over $\Cl(M,g)$ in
	signature $p-q\equiv_{8} 3,7$. The classification of $\Spin^o$ structures is surprisingly subtle more and deserves a separate study.
\end{remark}

\noindent
A based isomorphism of elementary real pinor bundles
$f:(S,\gamma)\rightarrow (S^{\prime},\gamma^{\prime})$ induces an
isomorphism of $\Spin^{o}_{\alpha}$ structures relative to $\eta\colon
\Cl(V,h)\to\End_{\mathbb{R}}(S_{0})$:
\begin{equation}
P_{\eta}(f)\colon Q(S,\gamma)\rightarrow Q(S^{\prime},\gamma^{\prime})\, ,
\end{equation}
where $Q(S,\gamma_{o})$ (respectively
$Q_{\alpha}(S^{\prime},\gamma^{\prime}_{o})$) denotes the unique
adapted $\Spin^{o}$ structure corresponding to $(S,\gamma_{o})$
(respectively $(S^{\prime},\gamma^{\prime}_{o})$) relative to
$\eta$. Given an elementary real pinor representation $\eta \colon
\Cl(V,h)\to\End_{\mathbb{R}}(S_{0})$, an isomorphism of $\Spin^{o}$
structures $f: Q\rightarrow Q^{\prime}$ induces a based isomorphism:
\begin{equation}
S_{\eta}(f)=S(Q,\eta)\rightarrow S(Q^{\prime},\eta)
\end{equation}
of elementary real pinor bundles, where $S(Q,\eta)$ (respectively
$S(Q^{\prime},\eta)$) denotes the elementary real pinor bundle
associated to $Q$ (respectively $Q^{\prime}$) through the
representation $\eta$. 

Relative to $[\eta]$, the correspondence defined above gives mutually
quasi-inverse functors:
\begin{equation}
P_{\eta}\colon\ClB^{\times}_{\eta}(M,g)\rightarrow \mathrm{C}^{o}_{\eta}(M,g)\, ,
\end{equation}
and 
\begin{equation}
S_{\eta}:\mathrm{C}^{o}_{\eta}(M,g)\rightarrow \ClB^{\times}_{\eta}(M,g)\, ,
\end{equation}
between the grupoid $\mathrm{C}^{o}_{\eta}(M,g)$ of adapted
$\Spin^{o}_{\alpha}$ structures on $(M,g)$ of type $[\eta]$ and the
grupoid $\ClB^{\times}_{\eta}(M,g)$ of elementary real pinor bundles
of type $[\eta]$ over $(M,g)$ and based pinor bundle
isomorphisms. This is a particular case of the general correspondence
between Lipschitz structures and bundles of irreducible real Clifford
modules presented in reference \cite{Lipschitz}.

\subsection{Natural sub-bundles of $End_{\mathbb{R}}(S)$}

By the associated bundle construction, the subspaces $\S,
\T,\A\subset \End_\R(S_0)$ of Section \ref{sec:elementarypinor}
globalize to linear sub-bundles $\Sigma, T$ and $A$ of
$End_{\mathbb{R}}(S)$ which can be described intrinsically as follows:

\begin{definition}
	The {\bf natural fiber sub-bundles of $End_{\mathbb{R}}(S)$} are
	defined as follows:
	\begin{enumerate}
		\itemsep 0.0em
		\item The {\bf Schur bundle} $\Sigma$ has fibers: 
		\be
		\Sigma_p\eqdef \{T\in \End_\R(S_p)|T\gamma_p(v)=\gamma_p(v)T~~\forall v\in T_p^\ast M\}
		\ee
		\item The {\bf anticommutant bundle} $A$ has fibers: 
		\be
		A_p\eqdef \{T\in \End_\R(S_p)|T\gamma_p(v)=-\gamma_p(v)T~~\forall v\in T_p^\ast M\}~~.
		\ee
		\item The {\bf twist bundle} $T$ is the direct sum: 
		\be
		T=\Sigma\oplus A~~.
		\ee
		\item The {\bf conjugation bundle} $\cD$ has fibers: 
		\be
		\cD_p\eqdef \{D\in \End_{\Sigma_p}^a(S_p)|D\gamma_p(v)=-D\gamma_p(v)~~\forall v\in T_p^\ast M~~\&~~D^2=\alpha_{p,q} \id_{\Sigma_p}\}~~.
		\ee
	\end{enumerate}
\end{definition}

\noindent Notice that $\Sigma$ and $T$ are bundles of unital
subalgebras of $End_\R(S)$ while $A$ is only a vector bundle. On the
other hand, $\cD$ is a principal $\U(1)$-bundle. We have
$\rk\Sigma=2$, $\rk A=2$ and $\rk T=4$. Moreover, $\Sigma$ is a
$\C$-bundle while $T$ is a bundle of unital algebras isomorphic with
$\Cl_2(\alpha_{p,q})$. We have $A_p\cap \Sigma_p=\{0\}$ for all $p\in
M$, hence $T=\Sigma \oplus A$. We also have $\Sigma A=A\Sigma=A$,
hence $A$ is a bundle of bimodules over $\Sigma$. The vector bundles
$\Sigma, T, A$ are associated to $Q$ through the representations
$\Ad_\S$, $\Ad_\T$ and $\Ad_\A$ of $\Spin^o(V,h)$ respectively. The
fiber bundle $\cD$ is associated to $Q$ through the left action of
$\Spin^o(V,h)$ on the homogeneous space $\Spin^o(V,h)/{\widehat
	\Pin}(V,h)$. Thus $\fD$ is isomorphic with the fiber bundle
$Q/\Pin(V,h)$.

\subsection{Certain reductions of structure group}

Let $L=\det T^\ast M=\wedge^d T^\ast M$ denote the orientation line
bundle of $M$.  The metric $g$ of $M$ determines a principal
$\Z_2$-sub-bundle $P_{\Z_2}(M)$ of $L$ whose fiber at $p\in M$
consists of the normalized volume forms $\nu_p,-\nu_p$ of $(T_pM,g_p)$
determined by the two orientations of $T_pM$. Since
$\nu_p^2=(-\nu_p)^2=-1$ in $\Cl(T_p^\ast M,g_p^\ast)$, the space
$\R\oplus L_p$ is a subalgebra of $\Cl(T_p^\ast M, g_p^\ast)$ and each
choice of orientation of $T_pM$ determines an isomorphism from this
subalgebra to $\C$.

\begin{definition}
	The {\bf complex orientation bundle} $L^c$ of $(M,g)$ is the bundle of
	unital subalgebras of $\Cl(T^\ast M, g^\ast)$ whose fiber at $p\in M$
	equals $L^c_{p}\eqdef \R \oplus L_p$.
\end{definition}

\noindent It is clear from the above that $L^c$ is a $\C$-bundle whose
imaginary line sub-bundle equals $L$. Thus $\w_c(L^c)=\w_1(L)$ and
hence $L^c$ is trivial iff $M$ is orientable. Notice that $L^c$ is
orientable (i.e. the structure group of $L^c$ reduces from $\O(2)$ to
$\SO(2)$) iff $L^c$ is trivial.

\begin{prop}
	$\gamma$ restricts to an isomorphism from $L^c$ to $\Sigma$. Moreover,
	the following statements are equivalent:
	\begin{enumerate}[(a)]
		\itemsep 0.0em
		\item $M$ is orientable.
		\item The Schur bundle $\Sigma\subset \End_{\mathbb{R}}(S)$ is
		trivial.
		\item The structure group of $S$ reduces from $\Spin^o(V,h)$ to
		$\Spin^c(V,h)$.
	\end{enumerate}
\end{prop}

\begin{proof}
	Since $\gamma_p:\Cl(T_p^\ast M, g_p^\ast)\rightarrow \End_\R(S)$ is
	faithful for all $p\in M$ and $\gamma(L)=\Sigma$, it is clear that
	$\gamma$ restricts to an isomorphism from $L_c$ to
	$\Sigma$. Equivalence of (a) and (b) is now immediate.  To prove
	equivalence of (b) and (c), recall that $\Sigma$ is associated to $Q$
	through the representation $\Ad_\S$, which, by Proposition
	\ref{prop:basicreps}, is equivalent with the representation
	$\Ad_+:\Spin^o(V,h)\rightarrow \Aut_\Alg(\Cl_2^+(\alpha))\simeq \Z_2$.
	The later can be identified with the grading morphism
	$\eta':\Spin^o(V,h)\rightarrow \Z_2$. The exact sequence
	\eqref{Spincseq} induces an exact sequence of pointed sets:
	\be
	H^1(M,\Spin^c(V,h))\stackrel{j_\ast}\longrightarrow H^1(M,\Spin^o(V,h))\stackrel{\eta'_\ast}{\longrightarrow} H^1(M,\Z_2)~~,
	\ee
	where $j_\ast$ is the map induced by the inclusion of $\Spin^c(V,h)$
	into $\Spin^o(V,h)$ and $\eta'_\ast([Q])=\w_c(\Sigma)$. The
	$\C$-bundle $\Sigma$ is trivial iff $\w_c(\Sigma)=0$, which amounts to
	$[Q]\in \ker \eta'_\ast$ i.e. $[Q]\in \im j_\ast$. The conclusion
	follows since the condition $[Q]\in \im j_\ast$ is equivalent with the
	reduction of structure group stated at (c).  
\end{proof}

\begin{prop}
	The following statements are equivalent:
	\begin{enumerate}[(a)]
		\itemsep 0.0em
		\item The anticommutant bundle $A\subset \End_{\mathbb{R}}(S)$ is
		trivial.
		\item The structure group of $S$ reduces from $\Spin^o(V,h)$ to
		$\Spin(V,h)$.
	\end{enumerate}
	In this case, $M$ is orientable and both the Schur bundle $\Sigma$ and
	the twist bundle $T$ are also trivial.
\end{prop}

\begin{proof}
	The vector bundle $A$ is associated to $Q$ through the representation
	$\Ad_\A$ which, by Proposition \ref{prop:basicreps}, is equivalent
	with the untwisted characteristic representation
	$\mu:\Spin^o(V,h)\rightarrow \O(2)$. The third exact sequence in
	\eqref{ses} induces an exact sequence of pointed sets:
	\be
	H^1(M,\Spin(V,h)\stackrel{j_\ast}\longrightarrow H^1(M,\Spin^o(V,h))\stackrel{\mu_\ast}{\longrightarrow} H^1(M,\O(2))
	\ee
	where $j_\ast$ is the map induced by the inclusion of $\Spin(V,h)$
	into $\Spin^o(V,h)$ and $\mu_\ast([Q])$ equals the class of the
	principal $\O(2)$-bundle corresponding to $A$.  This shows that $A$ is
	trivial iff $Q\in \ker\mu_\ast=\im j_\ast$, i.e. iff the reduction of
	structure group stated at (b) holds. It is clear that (b) implies the
	remaining statements.   
\end{proof}

\begin{prop}
	The following statements are equivalent:
	\begin{enumerate}[(a)]
		\itemsep 0.0em
		\item The principal $\U(1)$-bundle $\cD$ is trivial.
		\item The structure group of $S$ reduces from $\Spin^o(V,h)$ to
		${\widehat \Pin}(V,h)\simeq \Pin(V,\alpha_{p,q} h)$.
	\end{enumerate}
	In this case, the structure group of $S$ reduces further from
	${\widehat \Pin}(V,h)$ to $\Spin(V,h)$ iff the Schur bundle $\Sigma$
	is also trivial.
\end{prop}

\begin{proof} The reduction stated at (b) takes place iff the fiber bundle
	$Q/{\widehat \Pin}(V,h)$ (which is isomorphic with $\fD$) admits a
	section. Since this is a principal $\U(1)$-bundle, it admits a
	section iff it is trivial i.e. iff (b) holds. In this case, further
	reduction to $\Spin(V,h)$ occurs iff $M$ is orientable. This amounts
	to triviality of the complex orientation bundle of $M$, which is
	isomorphic with the Schur bundle.   
\end{proof}

\begin{remark}
	Recall that ${\hat D}\in {\widehat \Pin}(V,h)$ generates the cyclic
	group \eqref{Gammao}. This gives rise to a short exact sequence:
	\be
	1\longrightarrow \Gamma_{o,\alpha} \longrightarrow {\widehat \Pin}(V,h)\longrightarrow G\rightarrow 1~~,
	\ee
	where: 
	\begin{equation*}
	G=\twopartdef{\Spin(V,h)/\Z_2=\SO(V,h)}{p-q\equiv_8 3}{\Spin(V,h)}{p-q\equiv_8 7}\, .
	\end{equation*}
	The twisted vector representation $\tlambda: \Spin^o(V,h)\rightarrow
	\O(V,h)$ restricts to the twisted vector representation\footnote{We
		assume that we have chosen a vector $v\in V$ such that $\epsilon_{v}
		= 1$ to define the isomorphism ${\widehat
			\Pin}_{\alpha_{p,q}}(V,h)\simeq \Pin(V,\alpha_{p,q} h)$. We leave
		to the reader the details of the case $\epsilon_{v} = -1$.}
	$\tAd_0:{\widehat \Pin}(V,h)\simeq \Pin(V,\alpha_{p,q} h)\rightarrow
	\O(V,h)$, which descends to a surjective group morphism from
	$s:G\rightarrow \O(V,h)/\{-\id_V,\id_V\}\simeq \SO(V,h)$.
	Accordingly, the principal $G$-bundle $Q'\eqdef Q/\Gamma_{o,\alpha}$
	covers the principal $\SO(V,h)$-bundle $P'_{\SO(V,h)}(M,g)\eqdef
	P_{\O(V,h)}(M,g)/\Z_2$.  When $p-q\equiv_8 7$, we have
	$\alpha_{p,q}=+1$ and $s$ coincides with the double cover
	$\Ad_0:G=\Spin(V,h)\rightarrow \SO(V,h)$ given by the vector
	representation of $\Spin(V,h)$. In this case, the principal
	$\Spin(V,h)$-bundle $Q'$ covers the bundle $P'_{\SO(V,h)}(M,g)$. When
	$M$ is oriented, $P'_{\SO(V,h)}(M,g)$ is isomorphic with the special
	pseudo-orthogonal coframe bundle $P_{\SO(V,h)}$ and $Q'$ becomes a
	spin structure on $(M,h)$. When $p-q\equiv_8 3$, we have
	$\alpha_{p,q}=-1$ and $s$ is an isomorphism of groups
	$s:G=\SO(V,h)\rightarrow \SO(V,h)$. In this case, the principal
	$\SO(V,h)$-bundle $Q$ is isomorphic with $P'_{\SO(V,h)}(M,g)$.
\end{remark}

\subsection{Majorana spinor bundles and Majorana spinor fields when $p-q\equiv_8 7$}

Let us assume that $p-q\equiv_8 7$ (thus $\alpha_{p,q}=+1)$ and that the
bundle $\cD$ is trivial, i.e. that the structure group reduces to
${\widehat \Pin}(V,h)=\Pin(V,h)$. Then any global section $D\in
\Gamma(M,\cD)$ satisfies $D^2=+\id_S$ and hence gives an almost
product structure on the vector bundle $S$. This allows one to define
{\em spin projectors} $\cP_\pm\eqdef \frac{1}{2}(\id_S\mp D)\in
\Gamma(M,\End(S))$, which satisfy $\cP_\pm^2=\cP_\pm$, $\cP_+\oplus
\cP_-=\id_S$ and $\cP_\pm \cP_\mp=0$. The vector sub-bundles
$S^\pm\eqdef \cP_\pm(S)$ have fibers at $p\in M$ given by:
\be
S^\pm_p=\{x\in S_p|D_px=\pm x\}\, ,
\ee
and give a direct sum decomposition $S=S^+\oplus S^-$. Since $D_p$
anticommutes with $\gamma_p(\xi)$ for all $\xi\in T^{\ast}_p M$, we
have $\gamma_p(\xi)(S^\pm)= S^\mp$ and
$\gamma_p(\Cl^+(M,g))(S^\pm)=S^\pm$. Hence the bundle morphisms
$\gamma_\ev^\pm:\Cl^+(M,g)\rightarrow End_{\mathbb{R}}(S^\pm)$
defined through:
\be
\gamma^\pm_{\ev,p}(w)\eqdef \gamma_p(w)|_{S^\pm}\, , \quad \forall w\in \Cl^+(M,g)_{p}\, ,
\ee
are unital morphisms of bundles of algebras, making $S^\pm$ into
elementary real {\em spinor} bundles i.e. bundles of simple 
modules over the {\em even} sub-bundle $\Cl^+(M,g)$ of $\Cl(M,g)$. This
allows us to define ``Majorana spinors'' when $(M,g)$ admits a $\Pin(V,h)$
structure, even though $M$ may be unorientable.

\begin{definition}
	The real vector bundle $S^+$ is called the {\bf elementary bundle of
		Majorana spinors} defined by the {\bf global conjugation} $D\in
	\Gamma(M,\fD)$ and by the elementary real pinor bundle $S$. A global
	section $\epsilon\in \Gamma(M,S^+)$ is called an {\bf ordinary
		Majorana spinor field} while a global section $\epsilon\in
	\Gamma(M,S^-)$ is called an {\bf imaginary Majorana spinor field}.
\end{definition}

\noindent 
Any global section $\epsilon\in \Gamma(M,S)$ decomposes uniquely as
$\epsilon = \epsilon^+\oplus \epsilon^-$ with $\epsilon^\pm\eqdef
\cP_\pm(\epsilon)\in \Gamma(M,S^\pm)$ and hence $\epsilon$ can be
identified with a pair consisting of one ordinary and one imaginary
Majorana spinor field.

\begin{prop}
	There exists a natural isomorphism of semilinear vector bundles: 
	\be
	f:L^c\otimes S^+\stackrel{\sim}{\rightarrow} S\, .
	\ee
\end{prop}

\begin{proof} 
	For any $p\in M$, $z\in L^c_p=\alpha+\beta \nu_p\in \R\oplus \R \nu_p$
	(where $\alpha,\beta\in \R$) and $x\in S_p^+$, let $f_p:L_p^c\otimes
	S^+_p\rightarrow S_p$ be the linear map defined through:
	\ben
	\label{fp}
	f_p(z\otimes x)\eqdef (\alpha \id_S+\beta J_p)x~~,
	\een
	where $J_p\eqdef \gamma_p(\nu_p)$ and $\nu_p\in P_{\Z_2}(M,g)_p$ is
	one of the two Clifford volume forms of $(T_p M, g_p)$. This is
	well-defined since changing $\nu_p$ in $-\nu_p$ takes $\beta$ to
	$-\beta$ and $J_p$ to $-J_p$, leaving the operator $\beta
	J_p\in \End_\R(S_p)$ unchanged. Since $D$ is $\Sigma$-antilinear, we
	have $D_pJ_p=-J_pD_p$, hence $J_p(S^\pm)=S^\mp$. Thus $J_p(x)\in S^-$
	and \eqref{fp} gives:
	\be
	\cP_+(f_p(z\otimes x))=\alpha x~~,~~\cP_-(f_p(z\otimes x))=\beta J_p x~~.
	\ee
	Thus $f_p=(\cP_+\circ f_p)\oplus (\cP_-\circ f_p)$ (where
	$\cP_\pm\circ f_p:L^c_p\otimes S^+_p\rightarrow S_p^\pm$).  Since
	$J_p$ gives a bijection from $S^+_p$ to $S^-_p$, this implies that
	$f_p$ is bijective and hence that $f$ is an isomorphism of vector
	bundles. It is clear by construction that $f$ takes the semilinear
	structure of $L^c\otimes S^+$ into that of $S$.  \end{proof}

\begin{remark} 
	Recall that the complexification of the real vector bundle $S^+$ is
	the complex vector bundle $(S^+)^\C\eqdef \C_M\otimes_\R S^+$, where
	$\C_M$ is the trivial complex line bundle over $M$. The semilinear
	vector bundle $L^c\otimes_\C S$ is an analogue of this construction
	where $\C_M$ has been replaced by the complex orientation line bundle
	$L^c$ of $M$. The proposition above shows that the semilinear vector
	bundle $S$ is isomorphic with the ``twisted complexification'' $L_c\otimes_\R
	S^+$ of $S^+$. When $M$ is oriented, we can set $\epsilon_R\eqdef
	\epsilon^+$ and define $\epsilon_I\eqdef -J\epsilon^- \in
	\Gamma(M,S^+)$, where $J=\gamma(\nu)$ and $\nu$ is the normalized
	volume form of $(M,g)$. In that case, we have $L^c\simeq \C_M$,
	$S\simeq \C_M\otimes S^+=(S^+)^\C$ and $\epsilon =
	\epsilon_R+J\epsilon_I$, which means that we can identify a global
	section of $S$ with a pair of two ordinary Majorana spinor
	fields. Such an identification is {\em not} possible when $M$ is
	unorientable. The fact that (when $p-q\equiv_8 7$) one can define
	Majorana spinor fields when $(M,g)$ is unorientable but admits a
	$\Pin$ structure is important when considering global aspects of
	certain supergravity theories, such as eleven-dimensional supergravity
	(including its Euclidean version). Table \ref{table:sugra} lists some
	dimensions and signatures which are of interest in that context.
\end{remark}

\begin{table}[H]
	\centering
	\resizebox{1.3\textwidth}{!}{
		\hskip 3em \begin{minipage}{\textwidth}
			\begin{tabular}{|c|c|c|c|c|c|}
				\hline
				Type&$d$&$p$&$q$&$\begin{array}{c} p-q\\ {\rm mod}~8 \end{array}$&$\alpha_{p,q}$  \\
				\hline
				\hline
				\multirow{3}{*}{Riemannian}&\cellcolor{Cyan}$3$& \cellcolor{Cyan}$3$&\cellcolor{Cyan}$0$&\cellcolor{Cyan}$3$&\cellcolor{Cyan}$-1$\\
				\hhline{~-----}
				&$7$&$7$&$0$&$7$&$+1$\\
				\hhline{~-----}
				&\cellcolor{Cyan}$11$ &\cellcolor{Cyan}$11$&\cellcolor{Cyan}$0$&\cellcolor{Cyan}$3$&\cellcolor{Cyan}$-1$  \\
				\hline 
				\multirow{3}{*}{Lorentzian mostly minus}&$3$&$1$&$2$&$7$&$+1$\\
				\hhline{~-----}
				&\cellcolor{Cyan}$7$&\cellcolor{Cyan}$1$&\cellcolor{Cyan}$6$&\cellcolor{Cyan}$3$&\cellcolor{Cyan}$-1$ \\
				\hhline{~-----}
				&$11$&$1$&$10$&$7$&$+1$\\
				\hline 
				\multirow{2}{*}{Lorentzian mostly plus} &\cellcolor{Cyan}$5$&\cellcolor{Cyan}$4$&$1$\cellcolor{Cyan}&\cellcolor{Cyan}$3$ &\cellcolor{Cyan}$-1$ \\
				\hhline{~-----}
				&$9$&$8$&$1$&$7$&$+1$\\
				\hline 
			\end{tabular}
	\end{minipage}}
	\vskip 0.2in
	\caption{Dimensions $d\leq 11$ and Riemannian or Lorentzian signatures
		(mostly minus or mostly plus) which belong to the ``complex case''
		$p-q\equiv_8 3,7$. Majorana spinor fields can be defined when
		$\alpha_{p,q}=+1$ and when $(M,g)$ admits a $\Pin^+$ structure. Cases
		with $\alpha_{p,q}=-1$ are displayed in blue for clarity}
	\label{table:sugra}
\end{table}

%%%%%%%%%%%%%%%%%%%%%%%%%%%%%%%%%%%%%%%%%%%%%%%%%%%%%%%%%%%%%%%%%%%%%%%%%%%%
\section{Some examples of manifolds admitting adapted $\Spin^{o}$ structures}
\label{sec:examples}
%%%%%%%%%%%%%%%%%%%%%%%%%%%%%%%%%%%%%%%%%%%%%%%%%%%%%%%%%%%%%%%%%%%%%%%%%%%%

In this section we construct several examples of manifolds that admit
adapted $\Spin^{o}$ structures. We start with a simple result on the
existence of $\Spin^{o}$ structures induced by a number of
classical spinorial structures.

\begin{prop} Every pseudo-Riemannian manifold of signature $(p,q)$
	satisfying the condition $p-q \equiv_{8}3,7$ and admitting a $\Spin$,
	$\Pin$ or $\Spin^{c}$ structure also admits an adapted $\Spin^{o}$
	structure.
\end{prop}

\begin{proof}
	Follows from the natural embeddings of the groups $\Spin$, $\Pin$ and
	$\Spin^{c}$ into $\Spin^{o}_\pm$.  
\end{proof}

\noindent The orthonormal frame bundle of the normal bundle of a
submanifold of codimension two is a principal $\O(2)$-bundle which is
a natural candidate for the characteristic bundle $E$ of a
$\Spin^o_\pm$ structure. This observation leads to the following
result.

\begin{prop}
	\label{prop:cod2spino}
	Let $X$ be a $(2k+1)$-dimensional manifold which is oriented and spin
	and let $Y$ be an embedded $(2k-1)$-dimensional submanifold of $X$.
	
	\begin{enumerate}
		\item Assume that $2k-1 \equiv_{8} 7$ and that $X$ is endowed with a
		Riemannian metric $g$. Then $(Y,g|_Y)$ admits a
		$\Spin^{o}_{+}$ structure whose characteristic $\O(2)$-bundle $E$ is
		the orthogonal frame bundle of the normal bundle to $Y$ in $X$.
		\item Assume that $2k-1 \equiv_{8} 3$ and that $X$ is endowed with a
		negative Riemannian metric $g$. Then $(Y,g|_Y)$ admits a
		$\Spin^{o}_{-}$ structure whose characteristic $\O(2)$-bundle $E$ is
		the orthogonal frame bundle of the normal bundle to $Y$ in $X$.
	\end{enumerate}
\end{prop}

\begin{proof}
	We give the proof for $2k-1 \equiv_{8} 7$ since the other case is
	analogous. Let $TY$ and $NY$ denote the tangent and normal bundles to
	$Y$, thus $TX= TY \oplus NY $. Since by assumption we have $\w_{1}(TX)
	= \w_{2}(TX) = 0$, we obtain:
	\begin{equation}
	\w_{1}(TY) = \w_{1}(NY)\, , \qquad \w_{2}(TY) + \w_{1}(NY)^2 + \w_{2}(NY) = 0\, ,
	\end{equation}
	and the Proposition follows from Theorem \ref{thm:SpinoObs}.
	 \end{proof}

\begin{remark}
	Note that the $\Spin^o$ manifolds characterized in the previous
	Proposition will not, in general, admit any $\Spin$, $\Pin$ or
	$\Pin^c$ structure.
\end{remark}

\noindent Another construction related to the previous one is provided
by the following result.

\begin{prop}
	Let $(X,g)$ be a pseudo-Riemannian manifold of dimension $d\equiv_8
	2$, which is orientable and spin and $Y\subset X$ be an embedded
	submanifold of dimension $k$.
	\begin{enumerate}
		\item Assume that $k\equiv_{8} 3$ and that $g$ is
		positive-definite. Then $(Y,g|_Y)$ admits a
		$\Spin^{o}_{-}$ structure iff its orthogonal frame bundle admits a
		$\Spin^{o}_{+}$ structure (when $NY$ is endowed with the metric
		induced from $X$).
		\item Assume that $k\equiv_{8} 7$ and that $g$ is negative-definite.
		Then $(Y,g|_Y)$ admits a $\Spin^{o}_{+}$ structure iff its
		orthogonal frame bundle admits a $\Spin^{o}_{-}$ structure (when
		$NY$ is endowed with the metric induced from $X$).
	\end{enumerate}
\end{prop}

\begin{proof}
	We prove the first case, the other being similar. Since $Y$ is
	$k\equiv_{8} 3$-dimensional and $X$ is $d\equiv_{8}2$-dimensional, we
	have $\rk N Y\equiv_8 7$. Assume that $Y$ admits a
	$\Spin^{o}_{-}$ structure, so there exists a rank two vector bundle
	$E$ over $Y$ such that:
	\begin{equation}
	\label{eq:spino3}
	\w_{1}(M) = \w_{1}(E)\, , \qquad \w_{2}(M) = \w_{2}(E)\, .
	\end{equation}
	Since $TX = T Y\oplus T N$ while $X$ is orientable and spin, we have:
	\begin{equation}
	\label{eq:Xspin}
	\w_{1}(TY) = \w_{1}(NY) \, , \quad \w_{2}(TY) + \w_{2}(NY) + \w_{1}(TY)^2 = 0\, .
	\end{equation}
	Using \eqref{eq:spino3}, this gives:
	\begin{equation}
	\label{eq:TNspino}
	\w_{1}(NY) = \w_{1}(E) \, , \quad \w_{2}(NY) + \w_{2}(E) + \w_{1}(NY)^2 = 0\, ,
	\end{equation}
	so $NY$ admits an adapted $\Spin^{o}_{+}$ structure with
	characteristic bundle $E$. Conversely, assume that $NY$ admits a
	$\Spin^{o}_{+}$ structure. Then there exists a rank-two vector bundle
	$E$ such that \eqref{eq:TNspino} holds. Using \eqref{eq:Xspin} in
	\eqref{eq:TNspino} gives:
	\begin{equation}
	\w_{1}(T Y) = \w_{1}(E)\, , \qquad \w_{2}(TY) = \w_{2}(E)\, ,
	\end{equation}
	which implies that $(Y,g|_Y)$ admits a $\Spin^o_-$ structure.
	 \end{proof}

Let us present an explicit family of manifolds admitting adapted $\Spin^{o}$ structures but not admitting $\Pin^c$ structures. 

\begin{definition}
	A pseudo-Riemannian manifold $(M,g)$ is said to be {\bf stably
		$\Spin^{o}$} if $M\times\mathbb{R}^{j}$ admits an adapted
	$\Spin^{o}$ structure for some $j\geq 0$.
\end{definition}
Let:
\begin{equation}
\Gr_{k,n} \eqdef \frac{\O(n)}{\O(k)\times\O(n-k)}
\end{equation}
denote the real Grassmann variety of unoriented $k$-planes in
$\mathbb{R}^{n}$ and $\cL_{k,n}$ denote its tautological vector
bundle. We endow $\Gr_{k,n}$ with a (positive or negative) invariant metric. 
Define:
\begin{equation*}
\Gr^{j}_{k,n} \eqdef \Gr_{k,n}\times \mathbb{R}^{j}\, .
\end{equation*}

\begin{thm}
	\label{thm:grasspino}
	Assume that $n+1\equiv_4 0$. Then $\Gr_{2,n}$ is stably $\Spin^{o}$,
	namely:
	\begin{enumerate}
		\item For $j \equiv_{8} 7-2n$, the manifold $\Gr^{j}_{2,n}$ carries a
		$\Spin^{o}_{+}$ structure of positive-definite signature with
		characteristic $\O(2)$-bundle given by the orthogonal frame bundle of
		$\mathcal{L}_{2,n}$.
		\item For $j \equiv_{8} -(3+2n)$, the manifold $\Gr^{j}_{2,n}$ carries
		a $\Spin^{o}_{-}$ structure of negative-definite signature
		characteristic $\O(2)$-bundle given by the orthogonal frame bundle
		of $\mathcal{L}_{2,n}$.
	\end{enumerate}
\end{thm}

\begin{proof}
	We prove the statement for stable $\Spin^{o}_{+}$ structures, the other being similar. 
	Let $T_{2,n}$ denote the tangent bundle of $\Gr_{2,n}$. Then we have:
	\begin{equation}
	\label{eq:tangentgrass}
	T_{2,n}\oplus \mathcal{L}_{2,n}\otimes\mathcal{L}^{\ast}_{2,n} = (n + 2)\, T_{2,n}\, ,
	\end{equation}
	where $\mathcal{L}^{\ast}_{2,n}$ denotes the dual of the rank two vector bundle 
	$\mathcal{L}_{2,n}$. Equation \eqref{eq:tangentgrass} gives:
	\begin{equation}
	\label{eq:SWgrass}
	\w(T_{2,n})\, \w(\mathcal{L}_{2,n}\otimes\mathcal{L}_{2,n}^{\ast}) = \w(\mathcal{L}_{2,n})^{2+n}\, ,
	\end{equation}
	where $\w$ denotes the total Stiefel-Whitney class. Using
	$\w(\mathcal{L}\otimes\mathcal{L}^{\ast}) = 1
	+\w_{1}(\mathcal{L})^2+\ldots$, relation \eqref{eq:SWgrass} implies:
	\begin{equation*}
	\w_{1}(T_{2,n}) = (2 + n)\,\w_{1}(\mathcal{L}_{2,n})\, ,
	\end{equation*}
	as well as:
	\begin{equation*}
	(n+2)\w_{2}(\mathcal{L}_{2,n}) + \w_{2}(T_{2,n}) +(1+\frac{(n+2)(n+1)}{2})\w_{1}(\mathcal{L}_{2,n})^{2} = 0\, .
	\end{equation*}
	Since $n+1 \equiv_4 0$, the last two relations simplify to:
	\begin{equation*}
	\w_{1}(T_{2,n}) = \w_{1}(\mathcal{L}_{2,n})\, ,
	\end{equation*}
	and:
	\begin{equation*}
	\w_{2}(\mathcal{L}) + \w_{2}(T_{2,n}) +\w_{1}(\mathcal{L}_{2,n})^{2} = 0\, ,
	\end{equation*}
	which shows that $\Gr_{2,n}$ satisfies the topological obstruction to
	carry a Riemannian $\Spin^{o}_{+}$ structure of positive-definite
	type. However, $\Gr_{2,n}$ has dimension $2n$ and therefore it cannot
	carry a adapted $\Spin^{o}_{+}$ structure. Taking the direct product
	with $\mathbb{R}^{j}$ where $j \equiv_{8} (7-2n)$ gives the manifold
	$\Gr^{j}_{2,n}$, which has dimension:
	\begin{equation*}
	\dim\, \Gr^{j}_{2,n} = 2n + j =8 t + 7
	\end{equation*}
	for some positive integer $t$. Thus $\dim\, \Gr^{j}_{2,n} \equiv_8 7$.
	Using stability of Stiefel-Whitney classes, we conclude that
	$\Gr^{j}_{2,n}$ carries an adapted $\Spin^{o}_{+}$ structure. 
	 \end{proof}

\begin{remark}
	The odd-dimensional manifolds $\Gr^{j}_{2,n}$ appearing in the
	proposition above can be equipped with a bundle of irreducible real
	Clifford modules, which allows one to use the tools of spinorial geometry (such as
	the corresponding Dirac operator) to study them. These manifolds do
	not admit other classical spinorial structures, such as $\Spin$,
	$\Pin$ or $\Pin^c$ structures.
\end{remark}

\noindent
The simplest case occurring in Theorem \ref{thm:grasspino} is
$\Gr^{1}_{2,3}$.  The theorem shows that the seven-dimensional
Riemannian manifold $\Gr^{1}_{2,3}$ admits an elementary real pinor
bundle which has rank $16$ and is associated to a
$\Spin^{o}_{+}$ structure.

% % % % % % % % % % % % % % % % % % % % % % % % % % % % % % % % % % % % % % 
% % % % % % % % % % % % % % % % % % % % % % % % % % % % % % % % % % % % % % 

\section*{Acknowledgements}
The work of C. I. L. was supported by grant IBS-R003-S1. The work of C.S.S. is supported by the Humboldt foundation through the Humboldt grant ESP 1186058 HFST-P.

% % % % % % % % % % % % % % % % % % % % % % % % % % % % % % % % % % % % % % 
% % % % % % % % % % % % % % % % % % % % % % % % % % % % % % % % % % % % % % 

\appendix

\section{$\Spin^c_\alpha$ structures}
\label{Spincalpha}

In this appendix, we discuss $\Spin^c_\alpha$ structures (which for
definite signature were considered in \cite{Nakamura1, Nakamura2}) in
order to clarify the difference between them and the $\Spin^o_\alpha$
structures considered in this paper. Unlike the $\Spin^o_\alpha$
structures considered in the present paper, $\Spin^c_\alpha$ structures can
exist only for {\em orientable} principal pseudo-orthogonal bundles
(i.e. only for $\SO(V,h)$ bundles). In particular, a pseudo-Riemannian
manifold must be orientable in order to admit a $\Spin^c_\alpha$
structure.

\begin{definition}
	Let $P$ be a principal $\SO(V,h)$-bundle over $M$. A $\Spin^c_\alpha$
	structure on $P$ is a triplet $(E,P,f)$, where $E$ is a principal
	$\O(2)$-bundle over $M$ and $(P,f)$ is a $\rho_\alpha$-reduction of
	$P\times_M E$ to $\Spin^o_\alpha(V,h)$.
\end{definition}

Notice that the structure group of $P\times_M E$ need {\em not} reduce
to $\SO(V',h'_\alpha)$, because the image of $\rho_\alpha$ equals
$\SO(V,h)\times \O(2)$, which contains elements of negative
determinant. Also notice that a $\Spin^c_\alpha$ structure is defined
using the morphism $\rho_\alpha:\Spin_\alpha^o(V,h)\rightarrow
\SO(V,h)\times \O(2)$, while a $\Spin^o_\alpha$ structure is defined
using the morphism $\trho_\alpha:\Spin_\alpha^o(V,h)\rightarrow
\O(V,h)\hat{\times} \O(2)$. The case of $\Spin^c_\pm$ structures in
definite signature were considered in \cite{Nakamura1, Nakamura2}.

\begin{thm}
	Let $P$ be a principal $\SO(V,h)$-bundle (in particular, $P$ is
	orientable and hence we have $w_1(P)=w_1^+(P)=w_1^-(P)=0$). Then the
	following statements are equivalent:
	\begin{enumerate}[(a)]
		\itemsep 0.0em
		\item $P$ admits a $\Spin^c_\alpha$ structure.
		\item There exists a principal $\O(2)$-bundle $E$ such that the
		principal $\O({\hat V},{\hat h}_\alpha)$-bundle $\hP\eqdef P\times_M
		E$ admits a $\Pin$ structure.
	\end{enumerate}
\end{thm}

\begin{proof} 
	Follows by applying Lemma \ref{lemma:seq} to the morphism of exact
	sequences given in diagram \eqref{Gdiagram2}.  \end{proof} 

\begin{thm}
	The following statements hold:
	\begin{enumerate}[1.]
		\itemsep 0.0em
		\item A principal $\SO(V,h)$-bundle $P$ admits a $\Spin^c_+(p,q)$
		structure iff there exists a principal $\O(2)$-bundle $E$ such that
		the following condition is satisfied:
		\ben
		\label{obsplus}
		\w_2^+(P)+\w_2^-(P)=\w_2(E)
		\een
		\item A principal $\SO(V,h)$-bundle $P$ admits a $\Spin^c_-(p,q)$
		structure iff there exists a principal $\O(2)$-bundle $E$ such that
		the following condition is satisfied:
		\ben
		\label{obsminus}
		\w_2^+(P)+\w_2^-(P)=\w_2(E)+\w_1(E)^2~~.
		\een
	\end{enumerate}
\end{thm}

\begin{proof}
	It was shown in \cite{Karoubi} that the principal $\O({\hat V},{\hat
		h}_\alpha)$-bundle $\hP:=\hP_\alpha$ admits a $\Pin$ structure iff:
	\ben
	\label{Kpincond}
	\w_2^+(\hP)+\w_2^-(\hP)+\w_1^-(\hP)^2+\w_1^-(\hP)\w_1^+(\hP)=0~~.
	\een
	Distinguish the cases:
	\begin{enumerate}[1.]
		\itemsep 0.0em
		\item When $\alpha=+1$, we have $\hP^+=P^+\times E$ and
		$\hP^-=P^-$. Thus:
		\beqa
		&&\w_1^+(\hP)=\w_1^+(P)+\w_1(E)~~,~~\w_1^-(\hP)=\w_1^-(P)~~\nn\\
		&&\w_2^+(\hP)=\w_2^+(P)+\w_2(E)+\w_1^+(P)\w_1(E)~~,~~\w_2^-(\hP)=\w_2^-(P)~~\nn
		\eeqa
		Hence \eqref{Kpincond} becomes: 
		\ben
		\label{Kpinplus}
		\w_2^+(P)+\w_2^-(P)+\w_2(E)+(\w_1^+(P)+\w_1^-(P))(\w_1^-(P)+\w_1(E))=0
		\een
		Since $P$ is an $\SO(V,h)$ bundle, we have: 
		\be
		\w_1(P)=\w_1^+(P)+\w_1^-(P)=0~~,
		\ee
		which shows that \eqref{Kpinplus} reduces to \eqref{obsplus}.
		\item When $\alpha=-1$, we have $\hP^+=P^+$ and $\hP^-=P^-\times E$. Thus:
		\beqa
		&&\w_1^+(\hP)=\w_1^+(P)~~,~~\w_1^-(\hP)=\w_1^-(P)+\w_1(E)~~\nn\\
		&&\w_2^+(\hP)=\w_2^+(P)~~,~~\w_2^-(\hP)=\w_2^-(P)+\w_2(E)+\w_1^-(P)\w_1(E)~~\nn
		\eeqa
		and \eqref{Kpincond} becomes: 
		\be
		\w_2^+(P)+\w_2^-(P)+\w_2(E)+(\w_1^+(P)+\w_1^-(P))(\w_1^-(P)+\w_1(E))+\w_1(E)^2=0~~, 
		\ee
		which reduces to \eqref{obsminus} upon using $\w_1(P)=0$. 
	\end{enumerate}
	 \end{proof}

\begin{remark} 
	For positive signature with $\alpha=-1$ and $P=TM$, the statement of
	the Theorem recovers \cite[Proposition 3.4]{Nakamura1}. Notice that
	the proof in op. cit. is based on applying Lemma \ref{lemma:seq} to
	diagram \ref{Fdiagram2}.
\end{remark}

\begin{definition}
	Let $(M,g)$ be an orientable pseudo-Riemannian manifold. A
	$\Spin^c_\alpha$ structure on $(M,g)$ is a $\Spin^c_\alpha$ structure
	for the principal orthonormal frame bundle of $(M,g)$.
\end{definition}

\noindent When $(M,g)$ admits a $\Spin^c_\alpha$ structure, one can
introduce a corresponding notion of spinor bundles admitting Clifford
multiplication as in \cite{Nakamura1, Nakamura2}. However, that
construction requires that $(M,g)$ be orientable and hence it is quite
different from the construction considered in this paper since, unlike
$\Spin^c_\alpha$ structures, $\Spin^o_\alpha$ structures can exist on
{\em non-orientable} pseudo-Riemannian manifolds.

\section{Semilinear structures}
\label{app:semilinear}

In this section we discuss semilinear structures on real vector
bundles. A semilinear structure on a real vector bundle $S_0$ of even
rank $2r$ is an unordered pair of mutually conjugate complex
structures on $S_0$; equivalently, it is a reduction of the structure
group of $S_0$ from $\GL(2r,\R)$ to the general semilinear group
$\rGamma(r)$. Semilinear structures appear in Mathematical Physics when
studying discrete symmetries such as charge conjugation and time
reversal, through generally their theory is not clearly formalized in
that context. In the differential geometry literature, such structures
have also appeared in \cite[Section 5]{Crabb}, where they were called
``twisted complex structures'' and were considered from a point of
view different from ours. Semilinear structures will be
useful in Section \ref{sec:elementary} when studying elementary real
pinor bundles of ``complex type''. We start with a discussion of
semilinear structures on finite-dimensional real vector spaces.

\subsection{Semilinear structures on a real vector space}
\label{subsec:semilinear_vs} 

Let $S_0$ be an $\R$-vector space of finite even dimension. The {\em
	twistor set} of $S_0$ is defined as the set of all complex structures 
on $S_0$:
\be
\Tw(S_0)\eqdef \{J\in \End_\R(S_0)|J^2=-\id_{S_0}\}~~.
\ee
This set is stabilized by the involution $J\rightarrow -J$, which
generates a fixed-point free action of $\Z_2$ on $\Tw(S_0)$. Define the 
{\em reduced twistor set} of $S_0$ through:
\be
\Tw_0(S_0)\eqdef \Tw(S_0)/\Z_2=\{\{J,-J\} \,|\, J\in \Tw(S_0)\}\, .
\ee

\begin{definition}
	A {\bf semilinear structure} on $S_0$ is an element $\s \in
	\Tw_0(S_0)$, i.e. an {\em unordered} pair of mutually conjugate
	complex structures on $S_0$. A {\bf semilinear space} is a pair
	$(S_0,\s)$, where $S_0$ is an even-dimensional $\R$-vector space and
	$\s$ is a semilinear structure on $S_0$.
\end{definition}

\noindent Semilinear structures on $S_0$ can be identified with
certain subalgebras of the endomorphism algebra of $S_0$, as clarified
by the following:

\begin{prop}
	\label{prop:S}
	There exists a bijection between semilinear structures on $S_0$ and
	those unital subalgebras $\S$ of the $\R$-algebra $\End_\R(S)$ which
	have the property that the isomorphism type of $\S$ in the category
	$\Alg$ of unital associative $\R$-algebras equals the isomorphism type
	of the $\R$-algebra $\C$ of complex numbers.
\end{prop}

\begin{proof}
	It is clear that a unital subalgebra $\S$ of $\End_\R(S_0)$ is unitally
	isomorphic with $\C$ iff there exists $J\in \Tw(S_0)$ such that
	$\S=\R\id_S\oplus \R J$. In this case, $J$ is determined by $\S$ up to
	a sign change\footnote{This operation induces a unital $\R$-algebra
		automorphism of $\S$ corresponding to the unique nontrivial unital
		$\R$-algebra automorphism of $\C$ (which is the complex
		conjugation). Since $\Aut_\Alg(\C)\simeq \Z_2$, there are {\em two}
		distinct isomorphisms of unital $\R$-algebras from $\S$ to $\C$.}
	$J\rightarrow -J$. Hence any such subalgebra of $\End_\R(S_0)$
	determines a semilinear structure on $S_0$. Conversely, any semilinear
	structure $\s\in \Tw_0(S_0)$ determines a unital subalgebra $\S\eqdef
	\R\id_S\oplus \R J=\R\id_S\oplus \R(-J)$ of $\End_\R(S)$, where $J\in
	\s$ is any of the two conjugate complex structures belonging to $\s$. It
	is obvious that $\S$ is unitally isomorphic with
	$\C$.  
\end{proof}

\noindent Given a semilinear structure $\s\in\Tw_{0}(S_{0})$ we denote
by $\S \subset \End_{\mathbb{R}}(S_{0})$ the subalgebra associated to
$\s$ as in Proposition \ref{prop:S}. Any complex structure $J\in
\Tw(S_0)$ determines semilinear structure $[J]\eqdef \{J,-J\}$ (which
is the preimage of $J$ through the natural surjection
$\pi:\Tw(S_0)\rightarrow \Tw_0(S_0)$), but a semilinear structure only
determines a complex structure {\em up to sign}. Consider a semilinear
structure $\s$ on $S_0$ with corresponding subalgebra
$\S\subset \End_\R(S_0)$. Since $\Aut_{\Alg}(\S)\simeq
\Aut_{\Alg}(\C)\simeq \Z_2$, $\S$ has a unique non-trivial unital
$\R$-algebra automorphism which we denote by $c$; this corresponds to
complex conjugation through any of the two algebra automorphisms which
map $\S$ to $\C$. Accordingly, we can endow $S_0$ with two left
$\S$-module structures upon defining external multiplication through:
\be
sx\eqdef s(x)~~\forall s\in \S~~,~~\forall x\in S_0\, ,
\ee
or through:
\be
sx\eqdef c(s)(x)~~\forall s\in \S~~,~~\forall x\in S_0~~.
\ee
We usually consider only the first of these $\S$-module structures and
reserve the notation ${\bar S}_0$ to refer to the $\S$-module obtained
by using the second of these external multiplications.

\paragraph{Semilinear automorphisms}

Let $(S_0,\s)$ be a semilinear vector space. 

\begin{definition}
	An endomorphism $T\in \End_\R(S_{0})$ is called:
	\begin{enumerate}[1.]
		\itemsep 0.0em
		\item $\s$-linear, if $T\in \End_\S(S_0)$, i.e. if $T(sx)=sT(x)$
		for all $s\in \S$ and $x\in S_{0}$.
		\item $\s$-antilinear, if $T\in \Hom_\S(S_0,{\bar S}_0)$, i.e. if
		$T(sx)=c(s)T(x)$ for all $s\in \S$ and $x\in S_{0}$.
		\item $\s$-semilinear, if it is either $\s$-linear or
		$\s$-antilinear.
	\end{enumerate}
\end{definition}

\noindent The composition of two $\s$-semilinear maps is
$\s$-semilinear. For any invertible $\R$-linear operator $T\in
\Aut_\R(S_0)$, we have $\Ad(T)(\Tw(S_0))=\Tw(S_0)$ and $T$ is
$\s$-semilinear iff $\Ad(T)(\s)=\s$ (where we view $\s$ as a
two-element subset of $\Tw(S_0)$). In this case, $T$ is $\s$-linear
iff $\Ad(T)|_{\s}=\id_\s$ and $\s$-antilinear iff $\Ad(T)|_\s$ is the
non-trivial permutation of the two-element set $\s$. 

Invertible $\s$-semilinear maps form the {\em group of automorphisms}
of the semilinear space $(S_0,\s)$. This group is denoted by
$\Aut(S_0,\s)$ or by $\Aut_\S^\tw(S_0)$ and admits the $\Z_2$-grading:
\be
\Aut_\S^\tw(S_0)=\Aut_\S(S_0)\sqcup \Aut_\S^a(S_0)~~,
\ee
where $\Aut_\S(S_0)$ denotes the group of invertible $\s$-linear
transformations of $S_{0}$ and $\Aut^a_\S(S_0)$ denotes the {\em set}
of invertible $\s$-antilinear transformations of $S_0$. The
homogeneous components of this grading coincide with the two connected
components of $\Aut_\S^\tw(S_0)$, where $\Aut_\S(S_0)$ is the
connected component of the identity. The grading morphism
$f:\Aut_\S^\tw(S_0)\rightarrow \Z_2$ of this $\Z_2$-grading gives a
short exact sequence:
\be 
1\longrightarrow \Aut_\S(S_0)\hookrightarrow \Aut_S^\tw(S_0)\stackrel{f}{\longrightarrow} \Z_2\longrightarrow 1~~. 
\ee
Any choice of $\s$-antilinear operator $B\in \Aut_\S^a(S_0)$ such that
$B^2 = \id_{S_0}$ (existence can be easily proven) induces a right-splitting
morphism $g_B:\Z_2\rightarrow \Aut_\S^\tw(S_0)$ given by $g_B({\hat
	0})=\id_{S_0}$ and $g_B({\hat 1})=B$. Hence $\Aut_\S^\tw(S_0)$ is
isomorphic with the semidirect product $\Aut_\S(S_0)\rtimes_{\Ad\circ
	g_B} \Z_2$. 

\paragraph{The general semilinear group}

The space $\R^{2r}$ has a {\em canonical semilinear structure}
$\s=[\mathbf{J}_r]=\{\mathbf{J}_r,-\mathbf{J}_r\}$, where
$\mathbf{J}_r$ is the canonical complex structure of $\C^r$ under the
identification $\C^r\equiv \R^{2r}=\R^r\times \R^r$:
\be
\mathbf{J}_r(x,y)=(-y, x)~~\forall x,y\in \R^r\, .
\ee

\begin{definition} The {\bf general semilinear group}
	$\rGamma(r)\subset \GL(2r,\R)$ is the group of automorphisms of the
	semilinear space $(\R^{2r},[\mathbf{J}_r])$.  The {\bf semilinear
		orthogonal group} is the group $\TU(r)\eqdef \rGamma(r)\cap
	\O(2r,\R)$.
\end{definition}

\noindent We have $\rGamma(r)\simeq \GL(r,\C)\rtimes \Z_2$. One can
describe $\rGamma(r)$ explicitly as the set of all pairs of the form
$(A,\0)$ and $(A,\1)$, where $A\in \GL(r,\C)$ and $\Z_2=\{\0,\1\}$,
with group multiplication given by:
\beqa
& (A,\0)(B,\0)=(AB,\0)~~,~~(A,\0)(B,\1)=(AB,\1)~~\\
& (A,\1)(B,\0)=(A\bar{B},\1)~~,~~(A,\1)(B,\1)=(A\bar{B},\0)~~.
\eeqa
In this presentation, the group $\GL(r,\C)$ of complex-linear
transformations of $\C^r$ identifies with the subgroup of $\rGamma(r)$
consisting of elements of the form $(A,\0)$, while the set of
non-degenerate complex-antilinear transformations of $\C^r$ identifies
with the subset of $\rGamma(r)$ consisting of elements of the form
$(A,\1)$. Using the cannical real basis $e_1,\ldots, e_r, \i
e_1,\ldots, \i e_r$ of $\C^r$ (where $e_1,\ldots, e_r$ is the
canonical complex basis of $\C^r$), an element $A\in \GL(r,\C)$ can be
represented as an element of $\GL(2r,\R)$ having the special form:
\be
A_\lin=\left[\begin{array}{cc} A^R & -A^I \\ A^I & A^R\end{array}\right]~~,
\ee
where $A^R$ and $A^I$ are the real and imaginary parts of the non-degenerate complex matrix $A$. 
Similarly, a complex-antilinear tranformation of $\C^r$ with matrix $A\in \GL(r,\C)$ in the complex 
basis $e_1,\ldots, e_r$ can be represented as an element of 
$\GL(2r,\R)$ having the special form: 
\be
A_\alin=\left[\begin{array}{cc} A^R & A^I \\ A^I & -A^R\end{array}\right]=\left[\begin{array}{cc} I_r & 0 \\0 & -I_r\end{array}\right] (A_\lin)^T~~.
\ee
Using these forms, it is easy to see that we have: 
\be
\det (A,\0)=\det A_\lin=|\det A|^2~~,~~\det (A,\1)=\det A_\alin=(-1)^r|\det A|^2~~.
\ee
In particular, the determinant of any element of $\rGamma(r)$ is
positive when $r$ is even and in this case $\rGamma(r)$ is a subgroup
of the group $\GL_+(2r,\R)$ of general linear transformations with
positive determinant. Finally, notice that $\TU(r)$ is a maximal
compact form of $\rGamma(r)$. When $r$ is even, the discussion above
shows that $\TU(r)$ is a subgroup of $\SO(2r,\R)$.

\paragraph{Description of $\Tw(S_0)$ as a homogeneous space}

Recall that $\Aut_\R(S_0)$ acts transitively on $\Tw(S_0)$ through the
adjoint representation:
\be
\Ad(a)(J)= aJa^{-1}~~\forall a\in \Aut_\R(S_0)~~,~~\forall J\in \Tw(S_0)~~
\ee
and that the stabilizer of $J\in \Tw(S_0)$ in $\Aut_\R(S_0)$ equals the
group $\Aut_J(S_0)$ of those transformations which are complex-linear
with respect to $J$. Hence $\Tw(S_0)$ is the homogeneous space
$\Aut_\R(S_0)/\Aut_J(S_0)\simeq \GL(2r,\R)/\GL(r,\C)$. The adjoint
action of $\Aut_\R(S_0)$ on $\Tw(S_0)$ commutes with the $\Z_2$
action generated by $J\rightarrow -J$ and descends to a transitive
action on $\Tw_0(S_0)$. The stabilizer of a semilinear structure
$\s=\{J,-J\}\in \Tw_0(S_0)$ under this action consists of all
transformations which are semilinear with respect to $\s$:
\be
\Stab(\s)=\Aut(S_0,\s)=\Aut_\S^\tw(S_{0})\, ,
\ee
where $\S$ is the subalgebra of $\End_\R(S_{0})$ determined by
$\s$. Hence the set $\Tw_0(S_0)$ is the homogeneous space:
\begin{equation*}
	\Aut_\R(S_0)/\Aut_\S^\tw(S_0)\simeq \GL(2r,\R)/\rGamma(r)\simeq_{\mathrm{hty}}
	\O(2r,\R)/\TU(r)\, .
\end{equation*}

\paragraph{s-Hermitian metrics}

Let $(S_0,\s)$ be a semilinear vector space. 

\begin{definition}
	A $\R$-bilinear map $b:S_{0}\times S_{0} \rightarrow \R$ is called
	{\bf $\s$-compatible} if any $J\in \s$ satisfies $b(J x,J y)=b(x,y)$
	for all $x,y\in S_{0}$.
\end{definition}

\begin{definition}
	\label{def:s-Hermitian}
	Let $\s$ be a semilinear structure on $S_0$ with associated subalgebra
	$\S \subset \End_{\mathbb{R}}(S_{0})$. An $\S$-valued $\R$-bilinear
	form $h:S_0\times S_0\rightarrow \S$ is called an {\bf $\s$-Hermitian
		metric} on $S_0$ if it satisfies the following conditions:
	\begin{enumerate}[1.]
		\itemsep 0.0em
		\item It is conjugate-symmetric, i.e. $h(x, y)=c(h(y, x))$ for all
		$x,y\in S_0$.
		\item It is $\s$-sesquilinear, i.e. $h(sx,y)=c(s)h(x,y)$ and
		$h(x,sy)=sh(x,y)$ for all $x,y\in S_0$ and any $s\in \S$.
		\item It is positive-definite, i.e. $h(x,x)\in \R_{\geq 0}\id_{S_0}$
		for all $x\in S_0$ and $h(x,x)=0$ iff $x=0$.
	\end{enumerate}
\end{definition}

\noindent Recall that any complex structure $J\in \s$ belonging to the
semilinear structure $\s$ determines a direct sum decomposition
$\S=\R\id_{S_0}\oplus \R J$ of $\S$ as an $\R$-vector space. Thus any
element $s\in \S$ can be written uniquely as $s=s_0+s_1 J$, where
$s_0,s_1\in \R$. Moreover, changing $J$ into $-J$ does not change
$s_0$ but changes the sign of $s_1$. In particular, the quantity: 
\be
|s|_\S\eqdef \sqrt{s_0^2+s_1^2}\in \R_{\geq 0}~~\forall s\in \S
\ee
depends only on $\s$ and vanishes iff $s\in S$ does. Let $\Re_\S:\S
\rightarrow \R$ be the $\R$-linear map defined through:
\be
\Re_\S(s)=s_0~~\forall s\in \S
\ee
and $\Im_J:\S\rightarrow \R$  be the $\R$-linear map defined through: 
\be
\Im_J(s)=s_1~~\forall s\in \S~~.
\ee
Notice that $\Re_\S$ depends only on the semilinear structure $\s$,
while $\Im_J$ depends on the choice of a complex structure $J\in
\s$. In fact, we have $\Im_{-J}=-\Im_J$. We also have: 
\be
|s|_{\S}=\sqrt{(\Re_\S s)^2+(\Im_J s)^2}~~\forall s\in \S~~.
\ee
If $h$ is an $\s$-Hermitian metric on $S_0$, then the $\R$-valued map:
\begin{equation*}
	h_R\eqdef \Re h \colon S_0\times S_0\rightarrow \R\, ,
\end{equation*}
is an $\R$-valued {\bf $\s$-compatible Euclidean metric} on
$S_{0}$, i.e. it is an $\R$-bilinear map which satisfies the
conditions:
\begin{enumerate}[1.]
	\itemsep 0.0em
	\item It is symmetric, i.e. $h_R(x, y)=h_R(y, x)$ for all
	$x,y\in S_0$.
	\item It is $\s$-modular, i.e. $h_R(sx,sy)=|s|_\S^2 h_R(x,y)$ 
	for all $x,y\in S_0$ and any $s\in \S$.
	\item It is positive-definite, i.e. $h(x,x)\geq 0$
	for all $x\in S_0$, with eqality iff $x=0$.
\end{enumerate}
Similarly, the $\R$-valued map:
\begin{equation*}
	h_I^J\eqdef \Im_J h \colon S_0\times S_0\rightarrow \mathbb{R} \, 
\end{equation*}
is an {\bf $\s$-compatible symplectic pairing} on $S_{0}$, i.e. it is
an $\R$-bilinear map which satisfies the conditions:
\begin{enumerate}[1.]
	\itemsep 0.0em
	\item It is antisymmetric, i.e. $h_I^J(x, y)=-h_I^J(y, x)$ for all
	$x,y\in S_0$.
	\item It is $\s$-modular, i.e. $h_I^J(sx,sy)=|s|_\S^2 h_I^J(x,y)$ for
	all $x,y\in S_0$ and any $s\in \S$.
	\item It is non-degenerate, i.e. $h_I^J(x,y)=0$ for all $y\in S_0$ implies
	$x=0$.
\end{enumerate}

\noindent Moreover, we have $h_{I}^J(x,y)= h_{R}(J x,y)$ for all $x,y\in
S_{0}$. This relation defines a bijection between $\s$-compatible
Euclidean metrics on $S_{0}$ and $\s$-compatible non-degenerate
symplectic pairings on $S_{0}$. Moreover, any such $h_R$ or $h_I$
determines the other and determines a unique $\s$-Hermitian metric $h$
on $S_{0}$ such that $h_R=\Re h$ and $h_I=\Im_J h$, namely
$h(x,y)=h_{R}(x,y)+h_{I}(x,y) J$ for all $x,y\in S_0$. Given any
$\s$-Hermitian metric $h$ on $S_0$, the Euclidean metric $h_R=\Re_\S
h$ determines a Euclidean metric on the real determinant line $\det_\R
S_{0}\eqdef \wedge^{2r} S_0$, where $\dim_\R S_0 = 2r$.  Let
$\rS^0(\det_\R S_0,h)$ be the two-element set consisting of those
elements of $\det_\R S_0$ which have unit norm with respect to this
induced metric. Notice that $\rS^0(\det_\R S_0,h)$ can be identified
with the set of orientations of $S_0$.

\begin{prop}
	\label{prop:Lis}
	Let $(S_0,\s)$ be a seminilear vector space with $\dim_\R S_0=2r$. Then
	any $\s$-Hermitian metric $h$ on $S_0$ determines a map
	$f_h:\s\stackrel{\sim}{\rightarrow} \rS^0(\det_\R S_0,h)$. Moreover:
	\begin{enumerate}[1.]
		\itemsep 0.0em
		\item We have: 
		\ben
		\label{fhrel}
		f_h(-J)=(-1)^r f_h(J)~~\forall J\in \s~~.
		\een
		\item The map $f_h$ is bijective when $r$ is odd.
		\item When $r$ is even, the image $f_h(\s)$ coincides with one of the
		two elements of the set $\rS^0(\det_\R S_0,h)$, so in this case $\s$
		determines an orientation of the $\R$-vector space $S_0$.
	\end{enumerate}
\end{prop}

\begin{proof} 
	Let $J\in \s$ be any of the two elements of $\s$. Then the $\R$-linear
	map $\alpha:\S=\R\id_{S_0}\oplus \R J\rightarrow \C$ which sends
	$\id_{S_0}$ to $1$ and $J$ to the imaginary unit $\i$ is a unital
	isomorphism of $\R$-algebras from $\S$ to $\C$. Consider the
	$\R$-bilinear map $h_0\eqdef \alpha\circ h:S_0\times S_0\rightarrow
	\C$. Then Definition \ref{def:s-Hermitian} implies that $h_0$ is
	Hermitian with respect to the complex structure $J$ on $S_0$. Notice
	that we have $\Re h_0=\Re_\S h$. Moreover, identifying $\S$ with $\C$
	through the map $\alpha$ allows us to view the left $\S$-module $S_0$
	as a $\C$-vector space of dimension $r$. Let $e_1,\ldots, e_r$ be any
	$h_0$-orthonormal basis of this $\C$-vector space. Setting
	$e_{r+j}\eqdef J(e_j)$ for all $j=1,\ldots, r$, the collection
	$e_1,\ldots, e_{2r}$ is a basis of the $\R$-vector space $S_0$ which
	is orthonormal with respect to the Euclidean scalar product $\Re
	h_0=\Re_\S h$. Setting $f_h(J)\eqdef e_1\wedge_\R \ldots \wedge_\R
	e_{2r}\in \det_\R S_0$, this implies $||f_h(J)||_{\Re_\S
		h}=||f_h(J)||_{\Re h_0} =1$, hence $f_h(J)\in \rS^0(\det_\R
	S_0,h)$. It is easy to see that $f_h(J)$ is independent of the choice
	of the $h_0$-orthonormal basis $e_1,\ldots, e_{r}$ of the $\C$-vector
	space $S_0$. Indeed, any other $h_0$-orthonormal basis has the form
	$e'_k=\sum_{j=1}^r A_{jk}e_j$, where $A_{jk}\in \C$ are the
	coefficients of a unitary matrix $A\eqdef (A_{ij})_{i,j=1,\ldots,
		r}\in \U(r,\C)$. Decomposing $A_{jk}=A^R_{jk}+\i A^R_{jk}$
	with $A^R_{jk}, A^I_{jk}\in \R$ and setting $A^R\eqdef
	(A^R_{ij})_{i,j=1,\ldots. r}$, $A^I\eqdef
	(A^I_{ij})_{i,j=1,\ldots, r}$, we find that the real basis
	$e'_1,\ldots, e'_{2r}$ (where $e'_{r+k}\eqdef J(e'_k)$ for
	$k=1,\ldots, r$) is related to the real basis $e_1,\ldots, e_{2r}$
	through $e'_a=\sum_{b=1}^{2r} B_{ba}e_b$, where $B$ is the matrix
	given by:
	\be
	B=\left[\begin{array}{cc} A^R & -A^I \\ A^I & A^R \end{array} \right]~~.
	\ee
	Thus $e'_1\wedge_\R \ldots \wedge_\R e'_{2r}=(\det_\R B) e_1\wedge_\R
	\ldots \wedge_\R e_{2r}=e_1\wedge_\R \ldots \wedge_\R e_{2r}$, where
	we used the relation $\det_\R B=|\det_\C A|^2=1$ (the last equality
	follows from the fact that $A$ is a unitary matrix). Since $J\in \s$
	was chosen arbitrarily, we thus have a well-defined map
	$f_h:\s\rightarrow \rS^0(\det_\R S_0,h)$. The replacement
	$J\rightarrow -J$ induces the changes $\alpha\rightarrow -\alpha$ and
	$h_0\rightarrow \overline{h_0}$, but does not change the set of
	$h_0$-orthonormal bases of $S_0$ over $\C$. Hence replacing $J$ by its
	conjugate can be implemented by replacing $e_{r+j}\rightarrow
	-e_{r+j}$ for all $j=1,\ldots, r$, which shows that \eqref{fhrel}
	holds. This relation immediately implies the remaining statements of
	the proposition.
	
\end{proof}

\subsection{Semilinear structures on real vector bundles}

Let $S$ be a smooth real vector bundle of even rank $N=2r$ over a
connected, smooth and paracomoact manifold $M$ and let
$P_{\GL(N,\R)}(S)$ be the principal bundle of linear frames of $S$.
Let $End_\R(S)$ denote the bundle of endomorphisms of $S$. The latter
is a bundle of unital associative $\R$-algebras of type
$\End_\R(S_0)\simeq \Mat(N,\R)$, where $S_0$ is the typical fiber of
$S$.

\begin{definition}
	The {\bf twistor bundle} $\Tw(S)$ of $S$ is the fiber sub-bundle of
	$End_\R(S)$ whose fiber at $p\in M$ is the set:
	\be
	\Tw_p(S)=\{J\in \End_\R(S_p)|J^2=-\id_{S_p}\}
	\ee
	of all complex structures on the $2r$-dimensional real vector space
	$S_p$.
\end{definition}

\noindent Notice that $\Tw(S)$ is a bundle of homogeneous spaces whose
fibers are isomorphic with $\GL(2 r,\R)/\GL(r,\C)$. The {\em sign
	automorphism} is the involutive vector bundle automorphism
$\sigma:End_\R(S)\rightarrow End_\R(S)$ which acts on each fiber of
$End_\R(S)$ as $\End_\R(S_p)\ni T\rightarrow
-T\in \End_\R(S_p)$. Notice that $\sigma$ is {\em not} an automorphism
of $End_\R(S)$ as a bundle of algebras. We have
$\sigma(\Tw(S))=\Tw(S)$ and $\sigma|_{\Tw(S)}$ acts freely on each
fiber of $\Tw(S)$. 

\begin{definition}
	The {\bf reduced twistor bundle} $\Tw_0(S)$ is the fiber bundle
	defined as quotient $\Tw_0(S)\eqdef \Tw(S)/\Z_2$.
\end{definition}

\noindent Notice that $\Tw_0(S)$ is a bundle of homogeneous spaces
whose fibers $\Tw_0(S_p)$ are isomorphic with $\GL(2
r,\R)/\rGamma(r)$. The points of the fiber $\Tw_0(S)_p=\Tw_0(S_p)$ of
$\Tw_0(S)$ at a point $p\in M$ correspond to the semilinear structures
on the $\R$-vector space $S_p$. Hence $\Tw_0(S)$ is the bundle of
semilinear structures on the fibers of $S$.

\begin{definition}
	A semilinear structure on $S$ is a smooth global section $s\in
	\Gamma(M,\Tw_0(S))$. A {\em semilinear vector bundle} on $M$ is a pair
	$(S,s)$, where $S$ is a real vector bundle of even rank defined on $M$
	and $s$ is a semilinear structure on $S$.
\end{definition}

\begin{prop}
	There exists a bijection between semilinear structures on $S$ and
	reductions of structure group $P_{\rGamma(r)}(S)\rightarrow P_{\GL(2
		r,\R)}(S)$ of the principal bundle of linear frames of $S$ from
	$\GL(2r,\R)$ to $\rGamma(r)$.
\end{prop}

\begin{proof} 
	Since $\rGamma(r)$ is a closed subgroup of $\GL(2 r,\R)$, reductions
	of structure group of $P_{\GL(2r,\R)}(S)$ from $\GL(2r,\R)$ to
	$\rGamma(r)$ are in bijection with smooth sections of the fiber bundle
	$P_{\GL(2r,\R)}(S)/\rGamma(r)\simeq \Tw_0(S)$.  
\end{proof}

\noindent Given a $\C$-bundle $\Sigma$ over $M$, its bundle of imaginary units
$P_{\Z_2}(\Sigma)$ is the principal $\Z_2$-bundle whose fiber at $p\in
M$ equals the natural semilinear structure of $\Sigma_p$. If $\Sigma$
is a $\C$-bundle which is a sub-bundle of unital $\R$-subalgebras of
$End_\R(S)$, then $P_{\Z_2}(\Sigma)$ is a fiber sub-bundle of $\Tw(S)$
which can also be viewed as a section of $\Tw_0(S)$, i.e. as a
semilinear structure on $S$.
 
\begin{prop}
	The map: 
	\be
	\Sigma\rightarrow s_\Sigma \eqdef P_{\Z_2}(\Sigma)
	\ee 
	gives a bijection between sub-bundles of unital $\R$-subalgebras
	$\Sigma\subset End_\R(S)$ such $\Sigma$ is a $\C$-bundle and
	semilinear structures on $S$, whose inverse is given by:
	\be
	s\rightarrow \Sigma_s \eqdef \R_M\oplus L_s~~,
	\ee
	where $L_s$ is the real line sub-bundle of $End_\R(S)$ which is
	generated by $s$. 
\end{prop}

\begin{proof} Follows from Proposition \ref{prop:S}
\end{proof}
 
\noindent Given a semilinear line bundle $(S,s)$ on $M$, the
sub-bundle of unital algebras $\Sigma_{s}\subset End(S)$ endows $S$
with the structure of a bundle of left $\Sigma_s$-modules in the sense
that the fiber $S_p$ of $S$ at every point $p\in M$ becomes a left
module over the unital associative sub-algebra
$\Sigma_s(p)\subset \End_\R(S_p)$ upon defining exterior
multiplication through:
\be
\sigma x \eqdef \sigma(x)~~\forall \sigma \in \Sigma_s(p)~~\forall x\in S_p~~.
\ee
This left module is free and of rank equal to $\frac{1}{2}\rk_\R S$.

Notice that any complex structure
$J\in \End_\R(S)=\Gamma(M,End_\R(S))$ on $S$ induces a semilinear
structure $s_J$ on $S$ given by $s_J(p)=[J_p]\eqdef \{J_p,-J_p\}\in
\Tw_0(S_p)$ for all $p\in M$. This semilinear structure corresponds to
the rank two sub-bundle $\Sigma_{s_J}=L_{\id_S}\oplus L_{J}\subset
End_{\R}(S)$.  The latter is a trivial $\C$-bundle, being trivialized
by the linearly independent global sections $\id_S$ and $J$. We have
$s_{-J}=s_{J}$.
 
\begin{definition}
	A semilinear structure $s$ on $S$ is called {\bf trivial} if there exists
	a complex structure $J\in \Gamma(M,\Tw(S))$ on $S$ such that $s=s_J$.
\end{definition}
 
\noindent The following proposition follows immediately from the
previous discussion:
 
\begin{prop}
	A semilinear structure $s$ on $S$ is trivial if and only if $\Sigma_s$ is a
	trivial $\C$-bundle.
\end{prop}

\noindent In particular, $s$ is trivial iff $\w_c(\Sigma_s)=0$.

\paragraph{$s$-Hermitian metrics}

\begin{definition}
	Let $(S,s)$ be a semilinear vector bundle on $M$. A global section
	$h\in \Gamma(M,Hom_\R(S\times S,\Sigma_s))$ is called an $s$-Hermitian
	metric on $S$ if its value $h_p\in \Hom_\R(S_p\times S_p,\Sigma_s(p))$
	at any point $p\in M$ is an $s_p$-Hermitian metric on the vector space
	$S_p$ in the sense of Definition \ref{def:s-Hermitian}.
\end{definition}

\noindent Let $(S,s)$ be a semilinear vector bundle on $M$ and $h$ be
an $s$-Hermitian metric on $S$. Applying fiberwisely the construction
of Subsection \ref{subsec:semilinear_vs}, we define the real part
$h_R\in \Gamma(M,Hom_\R(S\times S,\Sigma_s,\R_M))\simeq
\Gamma(M,S^\vee\otimes_\R S^\vee)$ of $h$ (which is a Euclidean scalar
product on $S$). Let $P_{\Z_2}(S, h)$ denote normalized orientation
bundle of the Euclidean vector bundle $(S,h_R)$, i.e. the principal
$\Z_2$-bundle whose fiber at $p\in M$ equals the two-element set
$\rS^0(\det_\R S_p,h_p)$. Notice that the real determinant line bundle
$\det_\R S\eqdef \wedge^{\rk S} S$ is associated to the principal
$\Z_2$-bundle $P_{\Z_2}(S,h)$ through the sign representation of
$\Z_2$ on $\R$. Finally, let
$P_s(S)$ denote the principal $\Z_2$-bundle whose fiber at $p\in M$ is
given by $s(p)\subset \End_{\mathbb{R}}(S_p)$. The $\C$-bundle
$\Sigma_s$ is associated to $P_s(S)$ through the conjugation representation 
of $\Z_2$: 
\ben
\label{Sigma_assoc}
\Sigma_s=P_s(S)\times_\conj \C~~.
\een

\begin{prop}
	\label{prop:Lisom}
	Let $(S,s)$ be a semilinear vector bundle on $M$ with $\rk S=2r$. Then
	any $s$-Hermitian metric $h$ on $S$ determines a fiber map
	$f_h:P_s(S)\rightarrow P_{\Z_2}(S,h)$. Moreover:
	\begin{enumerate}[1.]
		\itemsep 0.0em
		\item If $r$ is odd, then the map $f_h$ is an isomorphism of
		principal $\Z_2$-bundles.
		\item If $r$ is even, then $S$ is orientable and $f_h(s)$ coincides
		with a global section of the determinant line bundle $\det_\R S$
		which has unit norm with respect to $h_R$, thus determining an
		orientation of $S$. In this case, the bundles $\det_\R S$ and
		$P_{\Z_2}(S,h)$ are topologically
		trivial.
	\end{enumerate}
\end{prop}

\begin{proof} Follows immediately from Proposition \ref{prop:Lis}. 
	
\end{proof}

\begin{remark}
	Another way to see that $S$ must be orientable in case 2. of the
	proposition is to recall that $\TU(r)$ (which is homotopy-equivalent
	with $\rGamma(r)$) is a subgroup of $\SO(2r,\R)$ when $r$ is even (see
	Subsection \ref{subsec:semilinear_vs}).
\end{remark}

\subsection{Classification of semilinear line bundles}

Let $(S,s)$ be a semilinear line bundle on $M$ and let $h$ be an
$s$-Hermitian metric on $S$. Since $\rk_{\mathbb{R}}(S)=2$, the line bundle $L_s\subset End_\R(S)$
generated by $s$ (which is associated to $P_s(S)$ through the sign
representation of $\Z_2$) is isomorphic with the real determinant line
bundle $\det_\R S$.
 
\begin{prop}
	Any rank two Euclidean vector bundle $(S,g)$ over $M$ admits a
	canonical semilinear structure. Moreover, there exist bijections
	between the following sets:
	\begin{enumerate}[(a)]
		\itemsep 0.0em
		\item The set of isomorphism classes of rank two vector bundles $S$
		over $M$.
		\item The set of isomorphism classes of rank two Euclidean vector
		bundles $(S,g)$ over $M$, where an isomorphism of Euclidean vector
		bundles is defined to be an invertible isometry.
		\item The set of isomorphism classes of principal $\O(2)$-bundles $P$
		over $M$.
		\item The set of isomorphism classes of semilinear line bundles
		$(S,s)$ over $M$.
	\end{enumerate}
	Moreover, first Stiefel-Whitney classes correspond to each other under
	these bijections and we have $\w_1(S)=\w_1(P_s(S))$.
\end{prop}

\begin{proof} 
	Any Euclidean rank two vector bundle $(S,g)$ on $M$ admits a canonical
	semilinear structure structure $s\eqdef P_{\Z_2}(\det_\R S,g)$.
	Explicitly, let us fix any point $p\in M$. Then any of the two
	elements $\nu_p\in (\det_\R S)_p=\wedge^2_\R S_p$ which has unit norm
	with respect to $g_p$ (i.e. any of the two elements of the set
	$\rS^0(\det_\R S_p, g_p)$) determines a complex structure $J_p\in
	\Tw(S_p)$ given by the cross product of the two-dimensional Euclidean
	vector space $(S_p,h_p)$ taken with respect to the orientation induced
	by $\nu_p$. Explicitly, $J_p$ is given by: \be \sharp(J_p
	x)=\iota_{x}\nu^\sharp_p~~,~~\forall x\in S_p~~, \ee where
	$\nu^\sharp_p\eqdef \sharp (\nu_p) \in \wedge^2 S_p^\vee$ is the
	normalized volume form of $S_p$ determined by $\nu_p$ and $\iota$
	denotes the contraction of forms on $S_p$ with vectors of $S_p$. Here,
	$\sharp_p:\wedge_\R^k S_p\stackrel{\sim}{\rightarrow} \wedge_\R^k
	S_p^\vee$ is the musical isomorphism. The canonical semilinear
	structure of the fiber $S_p$ is then given by $s_p=\{J_p,-J_p\}$; it
	is clear that this semilinear structure does not depend on the choice
	of $\nu_p$. As $p$ varies in $M$, $s_p$ determines the canonical
	semilinear structure $s$ of $(S,g)$. It is clear that $L_s\simeq
	\det_\R S$.
	
	This construction gives a natural map from the set of ismorphism
	classes of ran two Euclidean vector bundles and the set of isomorphism
	classes of seminilear line bundles. An inverse map is obtained by
	choosing an $s$-Hermitian metric on a semilinear line bundle $(S,s)$. This establishes the bijection between the sets at points (b)
	and (d) of the proposition and also shows that $\w_1(S)=\w_1(P_s(S))$.
	
	The bijections between (a), (b) and (c) are well-known and follow from
	the associated bundle construction and from the fact that the group
	$\O(2)$ is a maximal compact form of $\GL(2,\R)$ (which implies that
	any vector bundle admits a Euclidean metric). The fact that any rank
	two real vector bundle admits a semilinear structure as well as the
	bijection between (b) and (c) also follow by noticing that
	$\TU(1)=\O(2,\R)\simeq \U(1)\rtimes \Z_2$.
\end{proof}

\paragraph{The characteristic classes of a semilinear line bundle}

For each $w\in H^1(M,\Z_2)$, let $Z_w$ denote the
unique ismorphism class of principal $\Z$-bundles $Z$ such that
$\w_1(Z)=w$.

To any semilinear line bundle $(S,s)$, we associate the corresponding
$\C$-bundle $\Sigma:=\Sigma_s$ and its integer coefficient bundle
$Z_\Sigma=P_{\Z_2}(\Sigma)\times_{\sigma_\Z} \Z$. Then $(S,s)$ is a bundle of rank one
modules over $\Sigma$, i.e. a ``$\Sigma$-line bundle'' in the sense
of\footnote{ As opposed to the present paper (where we use the
	notation $\Sigma$), reference \cite{Froyshov} uses the notation $K$
	for a $\C$-bundle.  What op. cit. calls a ``$K$-line bundle'' is what
	we call a semilinear line bundle.} \cite[Section
2]{Froyshov}. Semilinear line bundles with {\em fixed} $\C$-bundle
$\Sigma$ form a symmetric groupoid $\cPic_\Sigma(M)$ with symmetric
monoidal product given by the tensor product of bundles of rank one
modules over $\Sigma$.  Let $\Pic_\Sigma(M)$ be the Abelian group of
isomorphism classes of this groupoid.  Then it was shown in
op. cit. that the so-called {\em twisted first Chern class} ${\tilde
	c}_1^\Sigma$ gives an isomorphism of groups:
\be
{\tilde c}_1^\Sigma:\Pic_\Sigma(M)\stackrel{\sim}{\rightarrow} H^2(M,Z_\Sigma)~~,
\ee
where $H^\ast(M,Z_\Sigma)$ denotes the singular cohomology of $M$ with
coefficients in $Z_\Sigma$. Moreover, we have $\w_1(S)=\w_1(Z_\Sigma)$ and 
the natural morphism
$H^2(M,Z_\Sigma)\rightarrow H^2(M,\Z_2)$ sends ${\tilde
	c_1}^\Sigma(S,s)$ to the second Stiefel-Whithey class $\w_2(S)$. On
the other hand, isomorphism classes of principal $\Z$-bundles $Z$ over
$M$ form an Abelian grop $\Prin_\Z(M)$ under the (fiber) product and
the first Stiefel-Whitney class gives an isomorphism of groups
$\w_1:\Prin_\Z(M)\stackrel{\sim}{\rightarrow} H^1(M,\Z_2)$.  If
$(S,s)$ is a semilinear line bundle with $Z_{\Sigma_s}=Z$, then we
have $\w_1(S)=\w_1(Z)$. We also have:
\be
\w_1(S)=\w_1({\det}_{\R}S)=\w_1(L)~~,
\ee
where $L=\R_M\otimes_\Z Z$. 

\paragraph{Relation to characteristic classses of principal $\O(2)$-bundles}

For each $w\in H^1(M,\Z_2)$, let $Z_w$ denote the unique ismorphism
class of principal $\Z$-bundles $Z$ such that $\w_1(Z)=w$.

Recall that principal $\O(2)$-bundles $P$ on $M$ are classified by
their first Stiefel-Whitney class $\w_1(P)\in H^1(M,\Z_2)$ and by
their twisted Euler class $\w_2(P)\in H^1(M,Z_{w_1(P)})$.

If $P$ is a principal $\O(2)$-bundle whose isomorphism class
corresponds to that of $(S,s)$, then we set ${\tilde c}_1(P)\eqdef
{\tilde c}_1(S,s)$. The set:
\be
\cE(M)\eqdef \sqcup_{w\in H^1(M,\Z_2)} H^2(M,Z_w)=\{(w,c)|w\in H^1(M,\Z_2), c\in H^2(M,Z_w)\} 
\ee
fibers (in the category of sets) over $H^1(M,\Z_2)$. Any semilinear
line bundle $(S,s)$ over $M$ determines an element
$w_1(Z_{\Sigma_s})=\w_1(S)\in H^1(M,\mathbb{Z}_{2})$ and an element ${\tilde
	c}_1(S,s)\in H^2(M,Z_{\Sigma_s})\simeq H^2(M,Z_{w_1(S)})$. Hence
$(w_1(S),{\tilde c}_1(S))\in \cE(M)$ and the map $(w_1,{\tilde c}_1)$
is a bijection between the set of isomorphism classes of semilinear
line bundles and the set $\cE(M)$. Composing this with the bijection
which takes isomorphism classes of principal $\O(2)$-bundles to
semilinear line bundles gives the classification of principal
$\O(2)$-bundles over $M$.

\paragraph{Relation to principal $\O_2(\alpha)$-bundles}

Recall that the group $\O_{2}(\alpha)$ introduced in Section
\ref{sec:O} fits into the following short exact sequence:
\begin{equation*}
	1 \to \U(1) \to \O_{2}(\alpha) \xrightarrow{\eta_{\alpha}} \mG_2 \to 1\, .
\end{equation*}
This induces a long exact sequence of pointed sets in
$\mathrm{\check{C}ech}$-cohomology, of which we are interested in the
following piece:
\begin{equation*}
	\mG_2 \xrightarrow{\delta} \check{H}^1(M,\U(1))\to \check{H}^1(M,\O_{2}(\alpha)) \xrightarrow{\eta_{\alpha \ast}} \check{H}^1(M,\mG_{2})\, ,
\end{equation*}
where $\delta$ denotes the connecting map. The pointed set
$\check{H}^1(M,\U(1))$ can be endowed with the structure of an Abelian
group, whereas the pointed set
$\check{H}^1(M,\O_{2}(\alpha))$ is in bijection with isomorphism
classes of principal $\O_{2}(\alpha)$-bundles over $M$.

% % % % % % % % % % % % % % % % % % % % % % % % % % % % % % % % % % % % % % 
% % % % % % % % % % % % % % % % % % % % % % % % % % % % % % % % % % % % % % 

\end{document}